\documentclass[12pt,a4paper]{amsart}

\usepackage{amsfonts}
\usepackage{amsmath}
\usepackage{amssymb}
\usepackage{amsthm}
\usepackage{mathrsfs}
\usepackage{color}

\usepackage{graphicx}
\usepackage[all]{xy}

\usepackage{enumerate}

\usepackage{hyperref}

\newcommand{\f}{\varphi}

\newcommand{\aA}{{\mathcal{A}}}
\newcommand{\bB}{{\mathcal{B}}}
\newcommand{\cC}{{\mathcal{C}}}
\newcommand{\dD}{{\mathcal{D}}}
\newcommand{\eE}{{\mathcal{E}}}
\newcommand{\fF}{{\mathcal{F}}}

\newcommand{\hH}{{\mathcal{H}}}
\newcommand{\iI}{{\mathcal{I}}}
\newcommand{\kK}{{\mathcal{K}}}
\newcommand{\mM}{{\mathcal{M}}}
\newcommand{\nN}{{\mathcal{N}}}
\newcommand{\lL}{{\mathcal{L}}}
\newcommand{\oO}{{\mathcal{O}}}
\newcommand{\pP}{{\mathcal{P}}}
\newcommand{\qQ}{{\mathcal{Q}}}

\newcommand{\tT}{{\mathcal{T}}}
\newcommand{\yY}{{\mathcal{Y}}}

\newcommand{\ve}{\varepsilon}

\newcommand{\lle}{\mbox{\raisebox{0.25ex}{${\scriptscriptstyle\le}$}}}
\newcommand{\gge}{\mbox{\raisebox{0.25ex}{${\scriptscriptstyle\ge}$}}}

\newcommand{\tr}{{\mbox{$t$-struc}\-tu\-r}}

\newcommand{\bcdot}{{\mbox{\boldmath{$\cdot$}}}}

\newcommand{\Bimod}{\mathop{\textbf{Bimod}}\nolimits}
\newcommand{\FM}{\mathop{\textbf{FM}}\nolimits}
\newcommand{\Cat}{\mathop{\textbf{Cat}}\nolimits}
\newcommand{\Hom}{\mathop{\textrm{Hom}}\nolimits}
\newcommand{\Ext}{\mathop{\textrm{Ext}}\nolimits}
\newcommand{\End}{\mathop{\textrm{End}}\nolimits}

\newcommand{\Coh}{\mathop{\textrm{Coh}}\nolimits}
\newcommand{\opp}{\mathop{\textrm{op}}\nolimits}
\newcommand{\hHom}{\mathop{\mathcal{H}\textrm{om}}\nolimits}

\newcommand{\Ex}{\mathop{\textrm{Ex}}\nolimits}
\newcommand{\Spec}{\mathop{\textrm{Spec}}\nolimits}

\newcommand{\Per}[1]{\mathop{{}^{#1}\textrm{Per}}\nolimits}
\newcommand{\Id}{\mathop{\textrm{Id}}\nolimits}
\newcommand{\Perf}{\mathop{\textrm{Perf}}\nolimits}
\newcommand{\Hot}{\mathop{\textrm{Hot}}\nolimits}

\newcommand{\wt}[1]{\widetilde{#1}}
\newcommand{\ol}[1]{\overline{#1}}

\newtheorem*{THM*}{Theorem}
\newtheorem*{COR*}{Corollary}
\newtheorem*{PROP*}{Proposition}
\newtheorem*{LEM*}{Lemma}

\newtheorem{LEM}{Lemma}[section]

\newtheorem{THM}[LEM]{Theorem}
\newtheorem{PROP}[LEM]{Proposition}
\newtheorem{COR}[LEM]{Corollary}

\theoremstyle{definition}

\newtheorem{REM}[LEM]{Remark}
\newtheorem{DEF}[LEM]{Definition}

\begin{document}
\title{Flops and spherical functors}
\author{Agnieszka Bodzenta}
\author{Alexey Bondal}
\date{\today}

\address{Agnieszka Bodzenta\\
	Faculty of Mathematics, Informatics and Mechanics 
	University of Warsaw \\ Banacha 2 \\ Warsaw 02-097,
	Poland} \email{A.Bodzenta@mimuw.edu.pl}

\address{Alexey Bondal\\
	Steklov Mathematical Institute of Russian Academy of Sciences, Moscow, Russia, and \\
	AGHA Laboratory, Moscow Institute of Physics and Technology, Russia, and\\
	Kavli Institute for the Physics and Mathematics of the Universe (WPI), The University of Tokyo, Kashiwa, Chiba 277-8583, Japan,  and \\
	National Research University Higher School of Economics, Russia} \email{bondal@mi.ras.ru}

	\begin{abstract}
	We study derived categories of Gorenstein varieties $X$ and $X^+$ connected by a flop. We assume that the flopping contractions $f\colon X\to Y$, $f^+ \colon X^+ \to Y$ have fibers of dimension bounded by 1 and $Y$ has canonical hypersurface singularities of multiplicity 2. We consider the fiber product $W=X\times_YX^+$ with projections $p\colon W\to X$, $p^+\colon W\to X^+$ and prove that the flop functors $F = Rp^+_*Lp^* \colon \dD^b(X) \to \dD^b(X^+)$, $F^+= Rp_*L{p^+}^* \colon \dD^b(X^+) \to \dD^b(X)$ are  equivalences, inverse to those constructed by M. Van den Bergh.
	
	The composite $F^+ \circ F \colon \dD^b(X) \to \dD^b(X)$ is a non-trivial auto-equivalence. When variety $Y$ is affine, we present $F^+ \circ F$ as the spherical cotwist associated to a spherical functor $\Psi$. The functor $\Psi$ is constructed by deriving the inclusion of the null-category $\mathscr{A}_f$ of sheaves $\fF \in \Coh (X)$ with $Rf_*(\fF )=0$ into $\Coh (X)$.
	
	We construct a spherical pair ($\dD^b(X)$, $\dD^b(X^+)$) in the quotient $\dD^b(W) /\kK^b$, where $\kK^b$ is the common kernel of the derived push-forwards for the projections to $X$ and $X^+$, thus implementing in  geometric terms a schober for the flop.
	
	A technical innovation of the paper is the $L^1f^*f_*$ vanishing for the Van den Bergh's projective generator. We construct a projective generator in the null-category and prove that its endomorphism algebra is the contraction algebra. 
\end{abstract}

\maketitle

\tableofcontents
\section{Introduction}\label{sec_intro}
A homological interpretation of the Minimal Model Program (MMP) in Birational Geometry was proposed in \cite{BonOrl,BO}. The basic idea is that MMP is about "minimisation" of the derived category $\dD^b(X)$ of coherent sheaves on a variety $X$ for varieties in a given birational class. More precisely, it is expected that if $X$  allows a divisorial contraction or a flip $X\dashrightarrow Y$, then $\dD^b(X)$ has a semi-orthogonal decomposition (SOD) with one component of SOD equivalent to $\dD^b(Y)$. Thus, minimizing the birational model should have the categorical meaning of chopping off semi-orthogonal factors of the derived category. A minimal model is expected to be a representative in the birational class of a variety whose derived category does not allow semi-orthogonal factors equivalent to the derived category of a variety birationally equivalent to $X$.

Since MMP in dimension greater than 2 deals with singular varieties, the right choice of the derived category to consider also matters. In particular, for ${\mathbb Q}$-Gorenstein varieties the derived category of a suitable stack is relevant (cf. \cite{Kaw}, \cite{Kaw1}). Minimal models are not unique, and MMP considers birational maps, called flops,  that link various minimal models. Conjecturally, flops induce derived equivalences \cite{BonOrl}, \cite{Kaw}, \cite{BO}.

There is a numerous evidence in favour of this conjectural picture, starting from the original paper \cite{BonOrl}, where various instances of flops
were proved to induce derived equivalences, and for simple higher dimensional flips, the required SOD was constructed. 

We expect that the whole zoology of categories, functors and natural transformations relevant to MMP should be governed by interesting hidden homotopy types, maps and higher homotopies (as it is proposed in \cite{BonTok}).
The present work can be considered as taking steps in this direction for categories and functors invoked by flops.

It was mentioned by the authors of \cite{BonOrl} that the functor that provides an equivalence $\dD^b(X)\to \dD^b(X^+)$, where $X$ and $X^+$ are connected by a flop, when composed with the analogous functor in the opposite direction $\dD^b(X^+)\to \dD^b(X)$ is not the identity but produces a non-trivial auto-equivalence of $\dD^b(X)$ (nowadays called a flop-flop functor). For Atiyah flop, the functor is given by what is now known as the spherical twist with respect to the spherical object ${\mathcal O}_C(-1)$, where $C$ is the (rational) exceptional curve. It was probably one of the first appearance of the spherical twists (though we should also mention here the work of S. Mukai \cite{Muk2} and S. Kuleshov \cite{Kul2} who used the action of spherical twists in their non-derived version to describe moduli of sheaves on K3 surfaces. A quick generalization to Calabi-Yau/derived case was understood by authors of \cite{BonOrl} in those days of {\em sturm und drang} on derived categories in Moscow in early 90's). M. Kontsevich suggested that the spherical twist of Atiyah flop should be transferred by Mirror Symmetry into the equivalence of the Fukaya category induced by the Dehn twist along a vanishing cycle. Properties of spherical twists with respect to spherical objects were later scrutinized by P. Seidel and R. Thomas in \cite{SeiTho} and for more general spherical functors by R. Anno and T. Logvinenko in \cite{AnnLog}.

Spherical (co)twist is a unification tool for various non-trivial auto-equivalences of $\dD^b(X)$ (cf. \cite{Add}) such as tensor products with line bundles, twists around spherical objects \cite{SeiTho}, EZ-twists \cite{Horja} or window shifts \cite{DonSeg}. 

The homotopy meaning of 
spherical (co)twists can be read off from their interpretation via schobers, i.e. categorifications of perverse sheaves on stratified topological spaces, suggested by M. Kapranov and V. Schechtman \cite{KapSche}. The homotopy type of the underlying stratification (the punctured disc for the case of one spherical functor) encodes the algebra of functors and natural transformations in the schober. The categorical incarnation of a schober on a punctured disc is a spherical pair, i.e. a pair of admissible subcategories of a triangulated category satisfying conditions that imply a spherical functor between them.

We study functors and natural transformations for flops with dimension of fibers of the flopping contractions bounded by 1. We construct the spherical pair $(\dD^b(X), \dD^b(X^+))$ in the appropriate quotient of the derived category of $X\times_Y X^+$.
The orthogonal complement to $\dD^b(X)$ is the bounded derived category of the abelian null-category 
\begin{equation}\label{eqtn_A_f} 
\mathscr{A}_f= \{E \in \Coh(X)\,|\, Rf_* E  =0 \}.
\end{equation} 
By deriving the embedding $\aA_f\to \Coh(X)$ we get the spherical functor 
$\Psi \colon\dD^b(\mathscr{A}_f)\to \dD^b(X)$. Its spherical cotwist is the flop-flop functor.

We lift our functors and natural transformations to bicategories $\Bimod$ and $\FM$ (see Appendix \ref{sec_sph_funct_and_enh}). We systematically consider 2-categorical adjunctions instead of the usual adjunctions of functors, scrutinise the uniqueness of adjoints and of associated twists and cotwists, i.e. the cones of adjunction units and counits.

Now we describe our results in more detail.

\vspace{0.3cm}
\noindent
\textbf{The flop functor and Van den Bergh's functor.}

We consider a flopping contraction $f:X \to Y$ and its flop $f^+: X^+\to Y$ with dimension of fibers bounded by 1. Exceptional loci of $f$ and $f^+$ are assumed to have codimension greater than 1 in $X$, respectively in $X^+$, while varieties $X$, $X^+$ and $Y$ to be Gorenstein and $Y$ to have canonical hypersurface singularities of multiplicity 2. 

Consider the diagram for the fiber product: 
\[
\xymatrix{& X\times_Y X^+ \ar[dl]_p \ar[dr]^{p^+} & \\ X \ar[dr]_f && X^+ \ar[dl]^{f^+} \\ & Y &}
\]
where $p$ and $p^+$ stand for the projections to the factors.
The flop functor is:
$$
F = Rp^+_*Lp^*\colon\dD_{qc}(X)\to \dD_{qc}(X^+).
$$
Note that this functor does not necessarily induce an equivalence if the dimension of fibers of $f$ is greater than 1 (see \cite{Nam2}), which means that the functor needs an adjustment for fibers of  higher dimension.

The above assumptions on the flopping contraction are those adopted by M. Van den Bergh in \cite{VdB}. We use also some technics borrowed from his work. M. Van den Bergh constructed an equivalence ${\Sigma} \colon \dD^b(X)\xrightarrow{\simeq} \dD^b(X^+)$, which is given via the identification of both categories with the derived category of modules over a sheaf of non-commutative algebras on affine $Y$. We first give a new interpretation for functor ${\Sigma}$. To this end, we identify $\dD^b(X)$ with $\Hot^{-,b}(\mathscr{P}_{-1})$, the homotopy category of complexes of projective objects $\mathscr{P}_{-1}$ in the heart $\Per{-1}(X/Y)$ of the perverse t-structure, introduced by T. Bridgeland in \cite{Br1}. We show that ${\Sigma}$ can be defined as the functor $(f^{+*}f_*(-))^{\vee \vee }$ applied term-wise to complexes of objects in $\mathscr{P}_{-1}$.  We give an independent proof that the flop functor is an equivalence and that it can be extended to an equivalence $\Sigma_{\textrm{qc}}$ between unbounded categories of quasi-coherent sheaves (Proposition \ref{prop_DG_enha_of_sigma}).

Then we show that the flop functor is given in the same way by the term-wise application of the functor $f^{+*}f_*(-)$. To this end, we need one of the technical innovations of our paper, the $L^1f^*f_*$-\emph{vanishing} for objects in $\mathscr{P}_{-1}$. We prove (Lemma \ref{lem_L1ffOM_i=0}) that if $f \colon X \to Y$ is a flopping contraction satisfying the above conditions, then $L^1f^*f_* \mM$ is zero, for any $\mM \in \mathscr{P}_{-1}$. The proof of the vanishing is based on a local presentation of $Y$ as a divisor in a smooth variety $\yY$.
Note that, since $Y$ is singular, object $Lf^*f_* \mM$ in general has infinitely many non-zero cohomology sheaves, more precisely, they satisfy 2-periodicity (see Section \ref{ssec_periodicity}), which is a reminiscent of matrix factorizations \cite{Buch}, \cite{Orl3}.
Finally, we prove that the composite $F^+ F$, the `flop-flop' functor, is the term-wise extension of $f^*f_*(-)|_{\mathscr{P}_{-1}}$.

We consider a divisorial embedding $i \colon Y \to \yY$ into a smooth $\yY$ together with  $g=i \circ f$, $g^+ = i \circ f^+$. If the base $Y$ of the flopping contraction is affine, then the flop functor $F$ and the Van den Bergh's equivalence $\Sigma_{\textrm{qc}}\colon \dD_{\textrm{qc}}(X)\to \dD_{\textrm{qc}}(X^+)$ fit into a functorial exact triangle (Proposition \ref{prop_triangle_for_flop}):
\begin{equation}\label{tria1}
\Sigma_{\textrm{qc}}[1] \to Lg^{+*}Rg_* \to F \to \Sigma_{\textrm{qc}}[2].
\end{equation}

Since both $Lg^{+*}Rg_*$ and $\Sigma_{\textrm{qc}}$ take $\dD^b(X)$ to $\dD^b(X^+)$, this allows us to conclude that the flop functor also preserves the boundedness of the derived categories.

Theorem \ref{thm_flop_is_inv_of_vdB} states that $\Sigma_{\textrm{qc}}$ is actually the inverse of the opposite flop functor
$$
F^+ = Rp_*Lp^{+*} \colon \dD_{\textrm{qc}}(X^+) \to \dD_{\textrm{qc}}(X).
$$
This implies that flop functors $F$ and $F^+$ yield (not mutually inverse) equivalences between $\dD^b(X)$ and $\dD^b(X^+)$. Following an argument of Chen \cite{Chen}, we generalize this statement to the case when base $Y$ is quasi-projective (see Section \ref{ssec_red_to_affine}).

\vspace{0.3cm}
\noindent
\textbf{The null-category and spherical functor $\Psi$.}

One of the original motivations for our work was to recover the importance of the null-category for $f$ as in (\ref{eqtn_A_f}).
Category $\mathscr{A}_f$ admits a projective generator $\pP$, which we obtain from Van den Bergh's projective generator $\mathcal{M}$ for the perverse heart $\Per{-1}(X/Y)$ by a projection to the null-category (see Proposition \ref{prop_proj_gen_of_A_f}).
Deriving the embedding $\mathscr{A}_f \to \Coh(X)$ gives us functor
$$
\Psi \colon \dD^b(\mathscr{A}_f) \to \dD^b(X).
$$

We prove that, for a  flopping contraction $f \colon X \to Y$ with affine $Y$, functor $\Psi$ is spherical. The flop-flop functor $F^+F$ is its spherical cotwist, i.e. the $\Psi^* \dashv \Psi$ adjunction unit fits into a functorial exact triangle:
$$
F^+F \to \Id_{\dD^b(X)} \to \Psi \Psi^* \to F^+F[1].
$$
We also show that the spherical twist $\dD^b(\mathscr{A}_f) \to \dD^b(\mathscr{A}_f)$ associated to $\Psi$ is the shift $\Id_{\dD^b(\mathscr{A}_f)}[4]$ of the identity functor (see Corollary \ref{cor_co_twist_for_Psi}).

\vspace{0.3cm}
\noindent
\textbf{Spherical pairs.}

We  assume again $f \colon X\to Y$ to be a flopping contraction with affine $Y$. Functor $p^*$  induces an isomorphism of endomorphism algebras of a projective generator  $\pP\in \mathscr{A}_f$ and of $p^* \pP\in \Coh(X \times_Y X^+)$. Moreover, $\Ext^i_{X\times_Y X^+}(p^*\pP, p^* \pP) = 0$, for $i>0$  (see Proposition \ref{prop_pP_no_higher_ext}). Thus, we construct a fully faithful functor $\dD(\mathscr{A}_f) \to \dD_{\textrm{qc}}(X\times_Y X^+)$.

We consider the 'common kernel subcategories':
\begin{align*} 
&\kK^b = \{ E\in \dD^b(X\times_Y X^+)\,|\, Rp_*(E) =0,\, Rp^+_*(E) =0\},&\\
&\kK^- = \{ E\in \dD^-(X\times_Y X^+)\,|\, Rp_*(E) =0,\, Rp^+_*(E) =0\}.&
\end{align*}
The composite $\dD^b(X) \xrightarrow{Lp^*} \dD^-(X\times_Y X^+) \to \dD^-(X\times_Y X^+)/\kK^-$ (not {
	$Lp^*$ itself!) factors via $\dD^b(X\times_Y X^+)/\kK^b$, thus inducing a fully faithful functor (Proposition \ref{prop_Lp_has_adjoint})
	$$
	\wt{L}p^* \colon \dD^b(X) \to \dD^b(X\times_Y X^+)/\kK^b.
	$$
	We prove semi-orthogonal decompositions (see \cite{B}):
	\begin{equation}\label{eqtn_intro_SOD}
	\dD^b(X\times_Y X^+)/\kK^b = \langle \dD^b(\mathscr{A}_{f^+}), \wt{L}p^* \dD^b(X)\rangle = \langle \wt{L}p^* \dD^b(X), \dD^b(\mathscr{A}_f)\rangle.
	\end{equation}
	As a result, we obtain a geometric description of the category $\dD^b(\mathscr{A}_{f^+})$:
	$$
	\dD^b(\mathscr{A}_{f^+}) \simeq \{E \in \dD^b(X\times_Y X^+)\,|\, Rp_*(E) = 0\}/\kK^b.
	$$
	By exchanging the roles of $X$ and $X^+$ in (\ref{eqtn_intro_SOD}) we get semi-orthogonal decompositions
	\begin{equation}\label{eqtn_intro_SOD1}
	\dD^b(X\times_Y X^+)/\kK^b = \langle \dD^b(\mathscr{A}_{f}), \wt{L}{p^+}^* \dD^b(X^+)\rangle = \langle \wt{L}{p^+}^* \dD^b(X^+), \dD^b(\mathscr{A}_{f^+})\rangle.
	\end{equation}
	Together decompositions (\ref{eqtn_intro_SOD}) and (\ref{eqtn_intro_SOD1}) provide us with a geometric description of \emph{4-periodical SODs} (see Proposition \ref{prop_quad_of_recol}), whose relation to spherical functors was basically discovered by D. Halpern-Leistner and I. Shipman \cite{HLSchi} (we thank M. Kapranov for explanations on this). Two pairs of subcategories $(\dD^b(X), \dD^b(X^+))$, $(\dD^b(\mathscr{A}_f), \dD^b(\mathscr{A}_{f^+}))$ are \emph{spherical pairs} \cite{KapSche} (see Theorem \ref{thm_schober}). The corresponding spherical functor for the second spherical pair is $\Psi$.
	
	\vspace{0.3cm}
	\noindent
	\textbf{The contraction algebra.}
	
	Assume the base $Y$ of the flopping contraction $f$ to be the spectrum of a complete Noetherian local ring.  Then the reduced fiber of $f$ over the unique closed point of $Y$ is a union of $n$ smooth irreducible rational curves $C_1, \ldots, C_n$ (see Theorem \ref{thm_fiberstructure}). The category $\Per{-1}(X/Y)$ has $n+1$ irreducible projective objects $\mM_0, \ldots, \mM_n$, with $\mM_0 \simeq \oO_X$ (see \cite{VdB}). $\mM = \bigoplus_{i=0}^n \mM_i$ is a projective generator for $\Per{-1}(X/Y)$.

	We prove that the endomorphism algebra
	$$
	A_P = \Hom_{\mathscr{A}_f}(\pP, \pP)
	$$
	of the corresponding projective generator for $\mathscr{A}_f$, $\pP= \ker(f^*f_* \mM \to \mM)$, is isomorphic to the \emph{contraction algebra} introduced in \cite{IyaWem2}, which is defined as the quotient of $\Hom_X( \mM, \mM)$ by the ideal of morphisms that factor via direct sums of copies of $\oO_X$ (Theorem \ref{thm_A_P_deform_alge}). This theorem relates our work to the results of W. Donovan and M. Wemyss \cite{DonWem}, where contraction algebras appear in the context of non-commutative deformations for the case of a flopping contraction of threefolds with an irreducible fiber over the unique closed point of $Y$.

	\vspace{0.3cm}	
	There are five Appendices attached to the main body of the paper.
	
	\vspace{0.3cm}
	\noindent
	\textbf{Appendix \ref{sec_Grot-Ver-dual}.}
	It is an extract of some properties of the functor $f^!$, right dual to $Rf_*$, and the Grothendieck duality.
	
	\vspace{0.3cm}
	\noindent
	\textbf{Appendix \ref{sec_func_ex_tr}.}
	We introduce functorial exact triangles and use them to define spherical functors. We recall after \cite{KapSche} the notion of a spherical pair and the associated spherical functor. 4-periodical SODs introduced by D. Halpern-Leistner and I. Shipman \cite{HLSchi} produce spherical pairs.		
	
	\vspace{0.3cm}
	\noindent
	\textbf{Appendix \ref{sec_sph_funct_and_enh}.} 
	We discuss a bicategory $\cC$ and a pair of 1-morphisms $s\in \Hom_{\cC}(A,B)$, $r\in \Hom_{\cC}(B,A)$ that fit into a 2-categorical adjunction $(s,r,\eta,\ve)$. When $\cC$ is 1-triangulated (meaning the categories of 1-morphisms are triangulated), we define the twist $t_s\in \Hom_{\cC}(B,B)$ and the cotwist $c_s\in \Hom_{\cC}(A,A)$ as the cones of the counit $\ve\colon sr\to \Id_B$ and the unit $\eta \colon \Id_A\to rs$. By using pseudo-functors and 2-categorical equivalences we show that the twist and cotwist are in a suitable sense invariant under replacing $A$ and $B$ by equivalent objects. If $t_s$ and $c_s$ are invertible in the 2-categorical sense, we say that $(s,r,\eta, \ve)$ is a \emph{spherical couple}. 

	We lift exact functors between triangulated categories to 1-morphisms in appropriate 1-triangulated bicategories. The first bicategory is $\Bimod$ whose objects are DG algebras and categories of 1-morphisms are defined as the derived categories of DG bimodules. The second one is the bicategory $\FM$ of schemes and the derived categories of their products as categories of 1-morphisms. Both bicategories admit 2-functors to the bicategory $\textbf{Cat}$ of categories, functors and natural transformations. More precisely, we have  $\Phi \colon \Bimod \to \textbf{Cat}$, $\Phi(A) = \dD(A)$ and $\Xi \colon \FM \to \textbf{Cat}$, $\Xi(X) = \dD_{\textrm{qc}}(X)$. 
	
	The above theory of 2-categorical adjunctions ensures that once a lift of an exact functor $\dD_{\textrm{qc}}(X) \to \dD_{\textrm{qc}}(Y)$ to a 1-morphism in $\Bimod$ or $\FM$ admitting an adjoint is fixed, we get essentially unique exact functors corresponding to the twist and the cotwist. 
	
	We describe 2-categorical adjunctions in $\Bimod$ using formulae for adjoint bimodules as in \cite{AnnLog2}. 
	For $\eE\in \dD_{\textrm{qc}}(X\times Y)$ we discuss functors adjoint to $\Xi_\eE\colon \dD_{\textrm{qc}}(X) \to \dD_{\textrm{qc}}(Y)$ and the conditions under which they are FM functors. Results of \cite{LunSchnu} allow us to transfer between $\Bimod$ and $\FM$ via fixing compact generators for $\dD_{\textrm{qc}}(X)$ and $\dD_{\textrm{qc}}(Y)$. We use the 2-categorical adjunctions in $\Bimod$ to construct functorial exact triangles for $\Xi_\eE$ and its adjionts. We check that these triangles are, up to isomorphism, independent of the choice of compact generators.
	In particular, given a morphism $f\colon X \to Y$, we discuss the lift of $Rf_*$, its adjoints, and the adjunction (co)units to $\Bimod$ and $\FM$. We also describe a 2-morphism in $\Bimod$ whose image under $\Phi$ is the base-change morphism.

	\vspace{0.3cm}
	\noindent
	\textbf{Appendix \ref{sec_fiber_of_f}.} This appendix is devoted to the description of the reduced fiber of a flopping contraction $f \colon X\to Y$ with fibers of relative dimension bounded by one over a closed point of $Y$.
	
	\vspace{0.3cm}
	\noindent
	\textbf{Appendix \ref{sec_spectr_seq}.} Here we show that cohomology of an appropriate complex allow us to calculate morphisms in the derived category of an abelian category between bounded above complexes with bounded cohomology.
	
	\vspace{0.3cm}
	\noindent
	\textbf{Notation.}~\\

	We denote by $k$ an algebraically closed field of characteristic zero. For a Noetherian $k$ scheme $X$, by $\Coh(X)$, respectively $\textrm{QCoh}(X)$, we denote the category of coherent, respectively quasi-coherent, sheaves on $X$.
	
	For an abelian category $\aA$, we denote by $\dD^b(\aA)$ and $\dD(\aA)$ the bounded and unbounded derived categories of $\aA$. We write $\dD^b(X) = \dD^b(\Coh(X))$, $\dD_{\textrm{qc}}(X) = \dD(\textrm{QCoh}(X))$.
	
	For an abelian category $\aA$, by $\textrm{Perf}(\aA)$ we denote the full subcategory of $\dD^b(\aA)$ of objects that are quasi-isomorphic to finite complexes of projective objects in $\aA$. For a scheme $X$, by $\textrm{Perf}(X)$ we denote the category of perfect complexes on $X$, i.e. objects of $\dD^b(X)$  locally quasi-isomorphic to finite complexes of locally free sheaves.
	
	For a $k$-algebra $A$, we denote by $\textrm{Mod--}A$ the abelian category of right $A$ modules. 
		
	For a \tr e $(\tT_{\lle 0}, \tT_{\gge 0})$ on a triangulated category $\tT$ with heart $\aA = \tT_{\lle 0 }\cap \tT_{\gge 0}$ and an object $T \in \tT$, we denote by $\hH_{\aA}^i(T)$ the $i$-th cohomology of $T$ with respect to the \tr e $(\tT_{\lle 0}, \tT_{\gge 0})$. The truncation functors are denoted by $\tau_{\gge i}^{\aA}$ and $\tau^{\aA}_{\lle i}$. If $\tT = \dD_{\textrm{qc}}(X)$ or $\dD^b(X)$ with the standard \tr e with heart $\textrm{QCoh}(X)$, respectively $\Coh(X)$, we shorten the notation to $\hH_X^i(T)$, $\tau_{\gge i}^X$ and $\tau_{\lle i}^X$ respectively.
	
	\vspace{0.3cm}
	\noindent
	\textbf{Assumptions.}~\\

	Throughout the paper we assume $X$ and $Y$ to be normal varieties over $k$.   We usually work under one of the following assumptions on a morphism $f\colon X\to Y$:
	\begin{itemize}
		\item \textit{(r)} $f\colon X\to Y$ is a proper morphism with fibers of \textbf{r}elative dimension bounded by one and such that $Rf_* \oO_X \simeq \oO_Y$.
		
		\item \textit{(p)} $f\colon X\to Y$ is a projective birational morphism with relative dimension of fibers bounded by one between quasi-\textbf{p}rojective Gorenstein varieties of dimension $n\geq 3$. The exceptional locus of $f$ is of codimension greater than 1 in $X$. Variety $Y$ has canonical hypersurface singularities of multiplicity two.
		
		\item \textit{(a)} The same as in (p) and we further assume that variety $Y$ is \textbf{a}ffine and is embedded as a principal divisor into a smooth variety $\yY$ of dimension $n+1$.
		
		\item \textit{(c)} The same as in (p) with an extra assumption that $Y= \Spec R$, where $R$ is a \textbf{c}omplete local $k$-algebra.
	\end{itemize}
	Let $f$ satisfy (p). Since $Y$ has canonical singularities, $Rf_* \oO_X \simeq \oO_Y$ (cf. \cite{Elk}),
	i.e. the condition (r) is satisfied when $Y$ has rational singularities and $f\colon X\to Y$ is a smooth birational resolution.
	
	\vspace{0.3cm}
	\noindent
	\textbf{Acknowledgements.} We are indebted to Pieter Belmans, Chris Brav, Alexander Efimov, Mikhail Kapranov, Yujiro Kawamata, Alexander Kuznetsov, Timothy Logvinenko, Valery Lunts, Yuri Prokhorov, {\L}ukasz Sienkiewicz, Greg Stevenson, Michel Van den Bergh and Maciej Zdanowicz for useful discussions. The first named author would like to thank Kavli IPMU and Higher School of Economics for their hospitality. The first named author was partially supported by Polish National Science Centre grants No. DEC-2013/11/N/ST1/03208 and 2012/07/B/ST1/03343. The second named author was partially supported by RFBR grant 18-01-00908 and by Laboratory of Mirror Symmetry NRU HSE, RF Government grant, ag. No. 14.641.31.0001. This research was partially supported by World Premier International Research Center Initiative (WPI Initiative), MEXT, Japan.
	
	\section{The null-category $\mathscr{A}_f$}\label{sec_null_cat_A_f}
	
	Let $f\colon X \to Y$ be a proper morphism of Noetherian schemes. In this section we introduce the \emph{null-category} of $f$: 
	\begin{equation}\label{eqtn_def_A_f}
	\mathscr{A}_f = \{ E\in \Coh(X)\,|\, Rf_*(E) = 0\}.
	\end{equation}
	Under the assumption that the dimension of fibers of $f$ is bounded by 1, $Y$ is affine, and $Rf_*\oO_X = \oO_Y$, we construct a projective generator $\pP$ for $\mathscr{A}_f$. We also study the behaviour of the null-category under decomposition $f = h \circ g$ and under restriction to fibers.

	\vspace{0.3cm}
	\subsection{The triangulated and abelian null-category of a morphism of schemes}\label{ssec_null-category}~\\
	
	For a proper morphism of Noetherian schemes $f \colon X \to Y$, we consider the triangulated null-category $\cC_f$ defined as the kernel of functor $Rf_*$ restricted to $\dD^b(X)$:
	\begin{equation}\label{eqtn_def_of_C_f}
	\cC_f = \{ E^\bcdot \in \dD^b(X)\,|\, Rf_*(E^\bcdot) = 0\}.
	\end{equation}
	
	Similarly, we define $\cC_f^- \subset \dD^-(X)$ and ${\cC_f}_{\textrm{qc}} \subset \dD_{\textrm{qc}}(X)$. Denote by ${\iota_f}_* \colon {\cC_f}_{\textrm{qc}} \to \dD_{\textrm{qc}}(X)$ the inclusion functor. As $Lf^*\colon \dD_{\textrm{qc}}(Y) \to \dD_{\textrm{qc}}(X)$ is fully faithful with right adjoint $Rf_*$, $\dD_{\textrm{qc}}(X)$ admits a semi-orthogonal decomposition $\dD_{\textrm{qc}}(X) = \langle {\cC_f}_{\textrm{qc}}, Lf^*\dD_{\textrm{qc}}(Y)\rangle$, \cite[Lemma 3.1]{B}. In particular, ${\iota_f}_*$ admits a left adjoint $\iota_f^*\colon \dD_{\textrm{qc}}\to {\cC_f}_{\textrm{qc}}$.
	
	For the null-category of $f$ as in (\ref{eqtn_def_A_f}), we have: $\mathscr{A}_f=\Coh(X)\cap \cC_f$.

	\begin{LEM}
		Let $f \colon X \to Y$ be a proper morphisms with fibers of relative dimension bounded by 1. Then category $\mathscr{A}_f$ is abelian. The embedding functor $\mathscr{A}_f \to \Coh(X)$ is exact and fully faithful. Its image is closed under extensions.
	\end{LEM}
	\begin{proof}
		Let $A, B$ be objects in $\mathscr{A}_f$ and let $\f\in \Hom_X(A,B)$. Consider the kernel $K$, the cokernel $C$, and the image $I$ of $\f$. Since $Rf_*A \simeq 0 \simeq Rf_* B$, the long exact sequences of higher derived functors for $f_*$ applied to short exact sequences on $X$
		\begin{align*}
		& 0 \to K \to A \to I\to 0,& &0 \to I \to B \to C \to 0,&
		\end{align*}
		together with vanishing of $R^if_*$, for $i>1$, imply that first $I$, hence $K$ and $C$ lie in $\mathscr{A}_f$. That $\mathscr{A}_f$ is closed under extensions is obvious. 		
	\end{proof}
	It follows that $\Ext^1_X(A,B)=\Ext^1_{\mathscr{A}_f}(A,B)$, for $A, B \in \mathscr{A}_f$.
	\begin{LEM}\cite[Lemma 3.1]{Br1}\label{lem_A_f_is_a_heart}
		Let $X$ and $Y$ be Noetherian schemes and $f \colon X \to Y$ a proper morphism with fibers of dimension bounded by 1. Then $E\in \cC_f$ if and only if $\hH^i_X(E)\in \mathscr{A}_f$, for all $i\in \mathbb{Z}$. In particular, $\mathscr{A}_f \subset \cC_f$ is the heart of a bounded \tr e.
	\end{LEM}
	
	\begin{REM}\label{rem_def_of_perv}
		The restriction of the standard \tr e on $\dD_{\textrm{qc}}(X)$ to 
		$$
		{\cC_f}_{\textrm{qc}} := \{ E^\bcdot\in \dD_{\textrm{qc}}(X)\,|\, Rf_*(E^\bcdot) =0 \}
		$$
		was basically considered by T. Bridgeland \cite{Br1} for a projective morphism $f\colon X\to Y$ satisfying (r). For this case, category $\dD_{\textrm{qc}}(X)$ has an SOD $\dD_{\textrm{qc}}(X) = \langle {\cC_f}_{\textrm{qc}}, \dD_{\textrm{qc}}(Y)\rangle$. The \tr e on ${\cC_f}_{\textrm{qc}}$ with heart $\mathscr{A}_f[-p]$ can be glued with the standard \tr e on $\dD_{\textrm{qc}}(Y)$. Since functor $f^!$ is needed for gluing \tr es, one has to consider unbounded derived categories of quasi-coherent sheaves on $X$ and $Y$. The heart of the resulted \tr e is the category ${}^p\textrm{Per}_{\textrm{qc}}(X/Y)$ of perverse sheaves.
		If $p=-1$ or $p=0$, the \tr es with hearts ${}^p\textrm{Per}_{\textrm{qc}}(X/Y)$ can also be obtained by the tilting in torsion pairs $(\tT_{-1}, \fF_{-1})$, $(\tT_{0}, \fF_{0})$  in $\textrm{QCoh(X)}$. Also, by restricting the torsion pairs to $\Coh(X)$ one can define abelian categories ${}^p\textrm{Per}(X/Y)$ as subcategories of $\dD^{b}(X)$. Thus, for $p=-1,0$,
		$$
		\Per{p}(X/Y) := \{ E\in \dD^b(X)\,|\, \hH^0(E) \in \tT_p,\, \hH^{-1}(E) \in \fF_p\, \textrm{ and }\hH^i(E) =0,\,\textrm{for }\, i\neq -1,0\},
		$$
		where (see \cite{VdB})
		\begin{align}
		&\tT_0= \{ T \in \Coh(X) \, |\, R^1 f_*(T)=0 \},\label{eqtn_def_T_0}&\\
		&\fF_0 = \{ E \in \Coh(X)\,|\, f_*(E) = 0,\, \Hom(\mathscr{A}_f, E) = 0 \},\label{eqtn_def_F_0}&\\
		&\tT_{-1} = \{T \in \Coh(X)\,|\, R^1f_*(T) = 0,\, \Hom(T, \mathscr{A}_f) = 0\}, \label{eqtn_def_T_-1}&\\
		&\fF_{-1} = \{ E \in \Coh(X)\,|\, f_*(E) =0 \}.\label{eqtn_def_F_-1}&
		\end{align}
	\end{REM}
	
	\vspace{0.3cm}
	\subsection{A projective generator for $\mathscr{A}_f$}\label{ssec_exist_of_proj_obj}~\\

	Let $f\colon X\to Y$ satisfy (r) and assume $Y$ is affine. By \cite[Proposition 3.2.5]{VdB}, there exists a vector bundle $\mM$ on $X$ which is a projective generator for $\Per{-1}(X/Y)$.  Let
	\begin{equation}\label{eqtn_def_of_P}
	\pP := \hH^{-1}_X \iota_f^* \mM,
	\end{equation}
	an object in $\mathscr{A}_f$, or, by abuse of notation, its image of $\pP$ under $\iota_{f*}$.
	\begin{PROP}\label{prop_proj_gen_of_A_f}
		Sheaf $\pP$ is a projective generator for the category $\mathscr{A}_f$.
	\end{PROP}
	\begin{proof}
		Since the \tr e on $\dD_{\textrm{qc}}(X)$ with heart $\Per{-1}_{\textrm{qc}}(X/Y)$ is glued, functor $\iota_{f*}|_{\mathscr{A}_f[1]} \colon \mathscr{A}_f[1] \to \Per{-1}(X/Y)$ is exact.
		It implies isomorphisms of functors
		$$
		\tau^X_{\gge -1} \iota_{f*} \simeq  \tau^{\Coh(X)[1]}_{\gge 0} \iota_{f*} \simeq \tau^{\Per{-1}(X/Y)}_{\gge 0} \iota_{f*} \simeq
		\iota_{f*} \tau^{\mathscr{A}_f}_{\gge -1}.
		$$
		It follows that $\tau^X_{\gge -1} \iota_{f*} \iota_f^* \mM$ is isomorphic to $\iota_{f*}\circ{}^{-1}\iota_f^* \mM$, where
		$$
		{}^{-1} \iota_f^*  := \tau_{\geq -1}^{\mathscr{A}_f}\circ \iota_f^* \colon \Per{-1}(X/Y) \to \mathscr{A}_f[1]
		$$
		is the left adjoint functor to the inclusion $\mathscr{A}_f[1] \to \Per{-1}(X/Y)$. By Lemma \ref{lem_restricting_projective} below, $\pP\in \mathscr{A}_f$ is a projective generator.
\end{proof}
\begin{LEM}\label{lem_restricting_projective}
	Let $\dD_0$ and $\dD$ be triangulated categories with \tr es with hearts $\aA_0\subset \dD_0$ and $\aA\subset \dD$, and $i_*\colon \dD_0 \to \dD$ a $t$-exact functor with left adjoint $i^*$. If $\mM \in \aA$ is projective, then $\pP:=\hH^0(i^* \mM) \in \aA_0$ is projective. Moreover, if $\mM \in \aA$ is a projective generator and $i_*$ is fully faithful, then $\pP \in \aA_0$ is a projective generator.  
\end{LEM}		
\begin{proof}
	Let $E$ be an object in $\aA_0$. Then $i_*E \in \aA$, hence $0=\Ext^1_{\dD}(\mM, i_*E) \simeq \Ext^1_{\dD_0}(i^*\mM, E)$. Since functor $i^*$ is left adjoint to a $t$-exact functor, it is right $t$-exact, i.e. $i^* \mM \in \dD_0^{\lle 0}$. Thus we have an exact triangle
	\begin{equation}\label{eqtn_truncation}
	\tau_{\lle -1} i^* \mM \to i^* \mM \to \pP \to \tau_{\lle -1}i^* \mM[1].
	\end{equation} 
	As, for degree reasons $\Hom_{\dD_0}(\tau_{\lle -1}i^*M, E) =0$, by applying $\Hom_{\dD_0}(-,E)$ to (\ref{eqtn_truncation}), we get that $\Ext^1_{\dD_0}(\pP, E)=0$, i.e. $\pP\in \aA_0$ is projective.

	If $\mM\in \aA$ is a projective generator and $i_*$ is fully faithful then, for any non-zero $E\in \aA_0$, object $i_*E$ is non-zero, hence $\Hom_{\dD}(\mM, i_*E) \simeq \Hom_{\dD_0}(i^*\mM, E)\neq 0$. It follows from the long exact sequence obtained by applying $\Hom_{\dD_0}(-,E)$ to $(\ref{eqtn_truncation})$ that $\Hom_{\dD_0}(\pP, E) \neq 0$, i.e. $\pP$ is a projective generator for $\aA_0$.
	\end{proof}
	\begin{REM}\label{rem_pro_gen_from_N}
		If morphism $f$ satisfies (r) and $Y$ is affine, the category $\Per{0}(X/Y)$ has a projective generator $\nN = \mM^{\vee}$ (see \cite[Proposition 3.2.5]{VdB}). Analogous argument as in the proof of Proposition \ref{prop_proj_gen_of_A_f} shows that $\pP_{\nN} := \hH_X^0 \iota_f^* \nN$ is a projective generator for $\mathscr{A}_f$.
	\end{REM}
	
	The decomposition of a sheaf $F$ with $R^1f_*F =0$ with respect to the SOD $\dD_{\textrm{qc}}(X) = \langle {\cC_f}_{\textrm{qc}}, Lf^*\dD_{\textrm{qc}}(Y)\rangle$ 
	gives a cohomology sequence
	\begin{equation}\label{eqtn_triangle_for_coh_of_iota}
	0 \to \hH_X^{-1}(\iota_{f_*}\iota_f^* F) \to f^*f_* F\to F\to \hH_X^0(\iota_{f*}\iota_f^* F) \to 0.
	\end{equation}
	Since by \cite[Lemma 3.1.3]{VdB} morphism  $f^*f_* \mM \to \mM$ is surjective, sequence
	\begin{equation}\label{eqtn_P_M_in_Coh}
	0 \to \pP \to f^*f_* \mM \to \mM \to 0
	\end{equation}
	is exact in $\Coh(X)$.
	
	Let $f\colon X\to Y$ satisfy (r) and $Y = \Spec R$ be a spectrum of a complete Noetherian local ring $R$. The reduced fiber $C_{\textrm{red}} = \bigcup_{i=1}^n C_i$ of $f$ over the unique closed point $y\in Y$ is a tree of rational curves (see a more precise statement in Theorem \ref{thm_fiberstructure}).
	
	The Picard group of $X$ is isomorphic to $\mathbb{Z}^n$, where the isomorphism is given by the degrees of the restriction to irreducible components of $C_{\textrm{red}}$: $\lL \mapsto \deg(\lL|_{C_i})_{i=1,\ldots,n}$.
	
	\begin{REM}(cf. \cite[Lemma 3.4.4]{VdB})\label{rem_divisors_D_i}
		Let $x_i \in C_i\subset X$ be a closed point such that $x_i \notin C_k$, for any $k \neq i$, and $j_i\colon {X}_i \to \mathcal{X}_i$
		a closed embedding of a neighborhood ${X}_i$ of $x_i$ in $X$ into a smooth variety $\mathcal{X}_i$. There exists an effective Cartier divisor $\mathcal{D}_i \subset \mathcal{X}_i$ such that scheme-theoretically
		$\mathcal{D}_i \cap j_{i*} C_i = \{j_{i*} x_i\}$. By pulling back $\mathcal{D}_i$ to ${X}_i$, we obtain an effective divisor $D_i \subset X$ such that scheme-theoretically $D_i.C_i = \{x_i\}$ and $D_i.C_k =0$, for $k\neq i$. We denote by $\iota_{D_i} \colon D_i \to X$ the embedding of $D_i$ into $X$.
	\end{REM}
	
	Denote by $\lL_i$ the line bundle on $X$ defined by:
	\begin{equation}\label{eqtn_def_L_i}
	\lL_i := \oO_X(D_i).
	\end{equation}
	
	For every $i$, M. Van den Bergh defined a vector bundle $\mM_i$ via the exact sequence
	\begin{equation}\label{eqtn_ses_def_M_i}
	0 \to \oO_X^{r_i-1} \to \mM_i \to \lL_i \to 0
	\end{equation}
	corresponding to a choice of $R$-generators in $\Ext^1_X(\lL_i, \oO_X)$. Denote $\mM_0:=\oO_X$. Then 
	$$
	\mM = \bigoplus_{i=0}^n \mM_i
	$$
	is a projective generator for ${}^{-1}\textrm{Per}(X/Y)$ \cite[Proposition 3.5.4]{VdB}.
	
	We put
	\begin{equation}\label{eqtn_def_P_i}
	\pP_i = \hH_X^{-1} \iota_f^* \mM_i.
	\end{equation}
	By Proposition \ref{prop_proj_gen_of_A_f}, sheaf $\pP = \bigoplus_{i=1}^n \pP_i$ is a projective generator for $\mathscr{A}_f$.
	
	\textbf{Example.} Let $X$ be a smooth threefold, $f\colon X\to Y$ a flopping contraction. Assume $Y = \Spec R$ has rational singularities and $C_{\textrm{red}}\simeq \mathbb{P}^1\subset X$. Then the normal bundle $N_{X/C_{\textrm{red}}}$ is isomorphic either to $\oO(-1) \oplus \oO(-1)$, $\oO \oplus \oO (-2)$ or $\oO (1) \oplus \oO (-3)$. Let $D\subset X$ be a divisor such that $D.C_{\textrm{red}} = 1$. For the first two cases, $\mM \simeq \oO_X \oplus \oO_X(D)$, and it is of higher rank for the third case. If $N_{X/C_{\textrm{red}}} \simeq \oO (-1) \oplus \oO (-1)$, the projective generator $\pP$ is $\oO (-1)$. When $N_{X/C_{\textrm{red}}} \simeq \oO  \oplus \oO (-2)$, $\pP$ is an $n$-iterated extension of $
	\oO (-1)$ by $\oO (-1)$, where $n$ is the width of $C_{\textrm{red}}$, see \cite{Reid} (cf. \cite{Toda3}).
	
	\vspace{0.3cm}
	\subsection{Basic properties of $\mathscr{A}_f$}\label{ssec_basic_prop_of_A_f}~\\
	
	Let $f \colon X \to Y$ satisfy (r) and let $E$ be a coherent sheaf on $Y$. By the derived projection formula, we have $Rf_*Lf^*(E) \simeq E$. Since $f$ has fibers of dimension bounded by 1, Leray spectral sequence $R^pf_* L^qf^*(E) \Rightarrow \hH_X^{p-q} E$ degenerates. Hence an exact sequence:
	\begin{equation}\label{eqtn_E_to_ffE}
	0 \to R^1f_*L^1f^*E \to E \to f_*f^* E \to 0
	\end{equation}

	\begin{LEM}\label{lem_projection_formula}
		Let $f \colon X \to Y$ and $E$ be as above. Then $R^1f_*f^*E = 0$ and $Rf_*L^if^*E = 0$, for any $i >1$. Further suppose that $E$ has no torsion supported at the image $f(\textrm{Ex}\, f)$ of the exceptional locus of $f$. Then $f_*f^*(E) \simeq E$ and sheaf $L^1f^*E$ is in $\mathscr{A}_f$.
	\end{LEM}
	
	\begin{proof}
		The derived projection formula $Rf_*Lf^* E \simeq E$ and Leray spectral sequence $R^pf_* L^qf^*(E) \Rightarrow \hH_X^{p-q} E$ imply $R^1f_*f^* E \simeq 0$, $f_*L^1f_* E \simeq 0$, $Rf_* L^if^* E \simeq 0$, for any $i>1$. Since $R^1f_*L^1f^* E$ is supported on $f(\textrm{Ex}\, f)$, the assumption that $E$ has no torsion supported on $f(\Ex f)$ and sequence (\ref{eqtn_E_to_ffE}) imply that $R^1f_*L^1f^*E$ is zero and $E \simeq f_*f^*E$.
	\end{proof}
	
	\begin{LEM}\label{lem_Lif_in_A_f}
		Let $f\colon X \to Y$ and $E$ be as above. If $E=f_* E'$, for some $E'\in\Coh(X)$, then $f_*f^*(E) \simeq E$ and sheaves $L^if^*E$ are in $\mathscr{A}_f$, for $i>0$.
	\end{LEM}	
	\begin{proof}	
		Sequence (\ref{eqtn_E_to_ffE}) implies that morphism $\alpha \colon f_* E' \to f_*f^*f_* E'$ is surjective. Morphism $\alpha$ is the inclusion of a direct summand, because its composition with the canonical map $f_*f^*f_* E'\to f_* E'$ is the identity. Hence, $\alpha$ is an isomorphism. Exact sequence (\ref{eqtn_E_to_ffE}) implies that $R^1f_*L^1f^*f_* E' = 0$. The rest follows from Lemma \ref{lem_projection_formula}.
	\end{proof}

	\begin{PROP}\label{prop_g_is_exact}
		Let $f\colon X\to Y$ be a proper morphism with dimension of fibers bounded by 1. Consider a decomposition for $f$:
		\[
		\xymatrix{X \ar[dr]^g \ar[dd]_f & \\ & Z \ar[dl]^h\\ Y&}
		\]
	 Then, for $E \in \Coh(X)$ with $R^1f_* E =0$, we have $R^1g_* E = 0$. Functor $g_*$ restricts to an exact functor $g_*\colon \mathscr{A}_f \to \mathscr{A}_h$.
	\end{PROP}
	\begin{proof}
		Morphism $g$ is proper by the valuative criterion. Replace $Z$ by the (closed) image of $g$, if necessary. Then $h$ becomes a proper morphism, \emph{cf.} \cite[Corollaire 5.4.3]{EGAII}. As the dimension of fibers for $f$ is bounded by 1, so is the dimension of fibers for $h$. 
		
Since $Rf_*(E)$ is a sheaf on $Y$, spectral sequence $R^qh_* R^sg_* E \Rightarrow R^{q+s}f_* E$ implies that $Rh_*R^1g_*(E) = 0$. Sheaf $R^1g_*(E)$ is supported in the locus of points $z \in Z$ such that the fiber of morphism $g$ over $z$ is one dimensional. Since the null-category of a finite morphism is zero, morphism $h\colon Z\to Y$ restricted to the support of $R^1g_*(E)$ must have fibers of dimension one. Let $y \in Y$ be a point in $h(\textrm{Supp}\, R^1g_*(E))$.
		Then the fiber of $f$ over $y$ is two-dimensional, which contradicts the assumptions. Thus, $R^1g_*(E)$ is zero.
		
		For any $E \in \mathscr{A}_f$, the spectral sequence $R^ph_*R^qg_* E \Rightarrow R^{p+q}f_* E = 0$ degenerates. It follows that $g_* E \in \mathscr{A}_h$. Since $R^1g_*(E) = 0$, functor $g_* \colon \mathscr{A}_f \to \mathscr{A}_h$ is exact.
	\end{proof}
	
	\begin{PROP}\label{prop_vanish_Rpi_Zpi_X_of_E}
		Let $f\colon X\to Y$ be a proper morphism with dimension of fibers bounded by 1 and $g \colon Z \to Y$ a morphism of schemes over field $k$. Assume a coherent sheaf $E$ on $X$ satisfy $R^lf_* E = 0$, for $l\geq l_0$, for some $l_0 \in \{0,1\}$. Then $R^l\pi_{Z*} \pi_X^* E = 0$, for $l\geq l_0$, where $\pi_X \colon W \to X$ and $\pi_Z \colon W \to Z$ are the projections for $W = X\times_Y Z$:
	\begin{equation}\label{eqtn_diag_korean_lem}
		\xymatrix{X \times_Y Z \ar[r]^(0.6){\pi_X} \ar[d]_{\pi_Z} & X \ar[d]^f \\ Z \ar[r]^{g} & Y}
		\end{equation}

	\end{PROP}
	
	\begin{proof}
		The statement is local on $Z$, hence we may assume that $Z = \Spec(A)$, $Y = \Spec(R)$ and morphism $g \colon Z \to Y$ is affine. Thus, we consider square (\ref{eqtn_diag_korean_lem})
		with $g$ and $\pi_X$ affine.
		
		The commutativity of (\ref{eqtn_diag_korean_lem}) implies $R(f\pi_X)_*\pi_X^* E \simeq R(g\pi_Z)_*\pi_X^* E$. Morphisms $g$ and $\pi_X$ are affine, hence we have an isomorphism $R^lf_* (\pi_{X*}\pi_X^* E) \simeq g_*R^l\pi_{Z*}(\pi_X^* E)$, for any $l\geq 0$. Morphism $g$ is affine, thus $g_*R^l\pi_{Z*}(\pi_X^*E)$ is zero if and only if so is $R^l\pi_{Z*}(\pi_X^* E)$. Hence, we need to show that $R^lf_* (\pi_{X*}\pi_X^* E) = 0$, for $l\geq l_0$.
		
		Morphism $\pi_X$ is affine, hence $\pi_{X*}\pi_X^* E \simeq E \otimes \pi_{X*}(\oO_{X\times_Y Z})$. The base change morphism (cf. (\ref{eqtn_def_of_omega})) $f^*g_* \oO_{Z}\to \pi_{X*} \oO_{X\times_Y Z}$ is an isomorphism. Indeed, this can be checked locally on $X$: if $X = \Spec(B)$, then both sheaves $\pi_{X*} \oO_{X\times_Y Z}$ and $f^*g_* \oO_Z$ correspond to $B$-module $B\otimes_R A$. Thus,
		$$
		\pi_{X*}\pi_X^* E \simeq E \otimes f^*g_* \oO_Z \simeq \hH^0_X(E \otimes^L Lf^* g_* \oO_Z).
		$$
		Derived projection formula $Rf_*(E \otimes^L Lf^*g_* \oO_Z) \simeq Rf_*(E) \otimes^L g_*\oO_Z$ implies that $Rf_* (E \otimes^L Lf^* g_* \oO_Z) \in \dD(Y)^{< l_0}$.
		Morphism $f$ is proper and with fibers of relative dimension bounded by one, hence $R^pf_*(F)=0$, for $p>1$ and for any sheaf $F$. It follows that spectral sequence
		$$
		R^pf_*\hH_X^q(E \otimes^L Lf^*g_* \oO_Z) \Rightarrow R^{p+q}f_*(E \otimes^L Lf^*g_* \oO_Z)
		$$
		degenerates. Therefore, $ R^lf_*(\pi_{X*}\pi_X^* E) \simeq R^l f_* \hH_X^0(E \otimes Lf^*g_* \oO_Z) = 0$, for $l\geq l_0$. 	
	\end{proof}
	
	\begin{COR}\label{cor_pull-back_of_A_f}
		Let $f\colon X\to Y$ and $g \colon Z \to Y$ be as in Proposition \ref{prop_vanish_Rpi_Zpi_X_of_E}. For $E\in \mathscr{A}_f$, its pull-back $\pi_X^* E$ is an object in $\mathscr{A}_{\pi_Z}$.
	\end{COR}

	\section{$L^1f^*$ vanishing and 2-periodicity}\label{sec_L1_vanish_and_period}

	Let $f\colon X\to Y$ satisfy (r). For $p=-1,0$, we denote by $\mathscr{P}_p$ the category of locally projective objects in $\Per{p}(X/Y)$. An object $\mM$ belongs to $\mathscr{P}_p$ if there exists an affine open covering $Y = \bigcup Y_i$, inducing $X = \bigcup X_i$ with $X_i = f^{-1}(Y_i)$, such that $\mM|_{X_i}$ is projective in $\Per{p}(X_i/Y_i)$. By \cite[Proposition 3.2.6]{VdB} objects in $\mathscr{P}_{-1}$ and $\mathscr{P}_0$ are locally free sheaves on $X$.

	In this section we discuss a spherical couple associated to a Cartier divisor.Under the assumption that $Y$ has hypersurface singularities we prove that  the sheaf $L^1f^*f_* \mM$ is zero, for any object $\mM \in \mathscr{P}_{-1}$. Though technical, this result is crucial for the various description of the flop and flop-flop functors presented in the subsequent sections.
	
	\vspace{0.3cm}
	\subsection{Cartier divisors and spherical couples}~\\
	
	Let $X$ be a quasi-compact, quasi-separated scheme. By an effective Cartier divisor $i\colon D\to X$ we mean a subscheme whose ideal sheaf $\iI_D$ is invertible. Here we discuss a 2-categorical adjunction and a spherical couple in the bicategory $\FM$ (see Appendix \ref{sec_sph_funct_and_enh}) related to such a  divisor.
	
	Denote by $\Gamma \subset D\times X$, $\Gamma^t \subset X\times D$ the graphs of $i$. Sheaf $\oO_\Gamma\in \dD_{\textrm{qc}}(D\times X)$ is an FM kernel for $Ri_*\colon \dD_{\textrm{qc}}(D) \to \dD_{\textrm{qc}}(X)$, and $\oO_{\Gamma^t}\in \dD_{\textrm{qc}}(X\times D)$ is an FM kernel for $Li^*\colon \dD_{\textrm{qc}}(X) \to \dD_{\textrm{qc}}(D)$.
	
	For a scheme $Y$, we
	denote by $\Delta^Y\subset Y\times Y$ the diagonal and, for closed $Z\subset Y$, by $\Delta^{Z,Y}$ the image of $Z$ under the diagonal morphism $\textrm{diag}^Y\colon Y \to Y \times Y$. For $F\in \dD_{\textrm{qc}}(D\times Y)$, we have: 
	\begin{equation}\label{eqtn_conv_Gamma_t} 
	\oO_{\Gamma^t}\ast F  = (i\times \Id_Y)_*F\in \dD_{\textrm{qc}}(X\times Y).
	\end{equation}
	In particular,
	$$
	\oO_{\Gamma^t} \ast \oO_{\Gamma}  = (i\times \Id_X)_*\oO_{\Gamma} \simeq \oO_{\Delta^{D,X}}.
	$$
	Embedding $\Delta^{D,X}\subset \Delta^X$ gives 
	\begin{align}\label{eqtn_eta_for_divisor}
	\eta \colon  \oO_{\Delta^X} \to  \oO_{\Delta^{D,X}} \simeq \oO_{\Gamma^t} \ast \oO_{\Gamma}.
	\end{align}
	For any scheme $Y$, and objects $F\in \dD_{\textrm{qc}}(X\times Y)$, $G\in \dD_{\textrm{qc}}(Y \times X)$, we have: 
	\begin{align*}
	&\oO_{\Gamma^t}\ast  \oO_{\Gamma} \ast F\simeq (i\times\Id_Y)_*(i\times \Id_Y)^*F,& &G \ast \oO_{\Gamma^t} \ast \oO_{\Gamma} \simeq (\Id_Y \times i)_*(\Id_Y\times i)^*G.&
	\end{align*}
	One checks locally that for $F\in \textrm{QCoh}(X\times Y)$, $G\in \textrm{QCoh}(Y\times X)$ morphisms
	\begin{align*}
	&F\simeq  \oO_{\Delta^X}\ast F \xrightarrow{\eta\ast F } \oO_{\Gamma^t}\ast \oO_{\Gamma} \ast F \simeq (i\times \Id_Y)_*(i\times \Id_Y)^*F \simeq F\otimes \oO_{D\times Y},&\\
	&G \simeq G\ast \oO_{\Delta^X} \xrightarrow{G\ast\eta  } G\ast \oO_{\Gamma^t} \ast\oO_{\Gamma}   \simeq (\Id_Y\times i)_*(\Id_Y\times i)^*G \simeq G\otimes \oO_{Y\times D}&
	\end{align*}
	are induced by  restriction morphisms $\oO_{X\times Y} \to \oO_{D \times Y}$, respectively $\oO_{Y\times X} \to \oO_{Y \times D}$.
	
	For any scheme $Y$ and $F\in \dD_{\textrm{qc}}(X\times Y)$, we have: 
	\begin{equation}\label{eqtn_conv_Gamma} 
	\oO_{\Gamma}\ast F \simeq (i\times\Id_Y)^*F\in \dD_{\textrm{qc}}(D\times Y).
	\end{equation}
	Hence,
	$$
	\oO_{\Gamma} \ast \oO_{\Gamma^t} \simeq (i\times \Id_D)^*\oO_{\Gamma^t} \simeq (i\times \Id_D)^*(i\times \Id_D)_*\oO_{\Delta^D} \in \dD_{\textrm{qc}}(D\times D).
	$$ 
	Subscheme $D\times D\subset X\times D$ is a Cartier divisor with ideal sheaf $\iI_{D\times D}$. Therefore, object $(i\times\Id_D)^*(i\times\Id_D)_*\oO_{\Delta^D}$ has two non-zero cohomology sheaves, $\oO_{\Delta^D}$ in degree $0$ and $\oO_{\Delta^D}\otimes \iI_{D\times D}|_{D\times D}$ in degree $-1$. Truncation to the 0'th cohomology gives morphism
	\begin{align}\label{eqtn_epsilon_for_divisor}
	\ve \colon \oO_{\Gamma}  \ast \oO_{\Gamma^t} \simeq (i\times\Id_D)^*(i\times\Id_D)_*\oO_{\Delta^D} \to \oO_{\Delta^D}
	\end{align} 
	Formulas (\ref{eqtn_conv_Gamma_t}) and (\ref{eqtn_conv_Gamma}) imply that, for any scheme $Y$ and objects $F\in \dD_{\textrm{qc}}(D\times Y)$ and $G\in \dD_{\textrm{qc}}(Y\times D)$, we have: 
	\begin{align*}
	&\oO_{\Gamma}\ast \oO_{\Gamma^t}\ast F\simeq (i\times \Id_Y)^*(i\times\Id_Y)_*F,& &G \ast \oO_{\Gamma} \ast \oO_{\Gamma^t} \simeq (\Id_Y\times i)^*(\Id_Y\times i)_*G.&
	\end{align*}
	One checks locally along $D$ that for $F\in \textrm{QCoh}(D\times Y)$ and $G\in \textrm{QCoh}(Y\times D)$  morphisms
	\begin{align*}
		&(i\times\Id_Y)^*(i\times\Id_Y)_*F \simeq \oO_{\Gamma}\ast \oO_{\Gamma^t} \ast  F\xrightarrow{\ve\ast F }  \oO_{\Delta^D}\ast F \simeq F,&\\
		&(\Id_Y\times i)^*(\Id_Y\times i)_*G \simeq G \ast \oO_{\Gamma} \ast  \oO_{\Gamma^t} \xrightarrow{G \ast \ve}G\ast\oO_{\Delta^D}   \simeq G&
	\end{align*} 
	are the truncations to the 0'th cohomology.
	
	An early version of formula (\ref{eqtn_second_triangle_for_Lii}) in the following theorem was first proven in \cite[Lemma 3.3]{BonOrl} under the assumption that both varieties are smooth. The result was also stated without proof in \cite{Ann, Add}. 
	
	\begin{THM}\label{thm_triangle_for_ii}
		Let $i\colon D \to X$ be the embedding of an effective Cartier divisor. Then $(\oO_{\Gamma^t}, \oO_{\Gamma}, \eta, \ve)$ is a spherical couple in the bicategory $\FM$.
		Object $\textrm{diag}^D_*\iI_D|_D[1]$ is the spherical twist  and  $\textrm{diag}^X_*\iI_D$ the spherical cotwist.
		The associated  functorial exact triangles read:
		\begin{align}
		&\Id_{\dD_{\textrm{qc}}(X)} \otimes \iI_D \to \Id_{\dD_{\textrm{qc}}(X)} \xrightarrow{\eta} Ri_*Li^* \to \Id_{\dD_{\textrm{qc}}(X)} \otimes \iI_D[1],&\label{eqtn_second_triangle_for_Lii}\\
		&\textrm{Id}_{\dD_{\textrm{qc}}(D)}\otimes \iI_D|_D[1] \to Li^*Ri_* \xrightarrow{\ve} \textrm{Id}_{\dD_{\textrm{qc}}(D)} \to \textrm{Id}_{\dD_{\textrm{qc}}(D)}\otimes \iI_D|_{D}[2].&\label{eqtn_triangle_for_Lii}
		\end{align} 
	\end{THM}
	\begin{proof}
		First, we check that $(\oO_{\Gamma^t},\oO_{\Gamma}, \eta, \ve)$ is a 2-categorical adjunction, i.e. that (\ref{eqtn_triangle_equalities}) 
		are equal to the identity morphisms.
		
		As $\oO_{\Gamma} \simeq (\Id_D\times i)_*\oO_{\Delta^D}$, the composite $\oO_{\Gamma}\xrightarrow{\oO_{\Gamma}\ast \eta } \oO_{\Gamma} \ast \oO_{\Gamma^t} \ast \oO_{\Gamma} \xrightarrow{\ve\ast \oO_{\Gamma}  } \oO_{\Gamma}$ reads:
		\begin{align*}
		(\Id_D\times i)_*\oO_{\Delta^D}\xrightarrow{\oO_{\Gamma}\ast\eta  } (\Id_D\times i)_*(\Id_D\times i)^*(\Id_D\times i)_*\oO_{\Delta^D} \xrightarrow{\ve\ast\oO_{\Gamma}  } (\Id_D\times i)_*\oO_{\Delta^D}.
		\end{align*}
		Object $(\Id_D\times i)_*\oO_{\Delta^D}$ is supported on $D\times D$, hence the first morphism, restriction to $D\times D$, is the identity on the 0'th cohomology. The second morphism is the truncation at 0'th cohomology, hence $(\ve\ast\oO_\Gamma   )\circ (\oO_{\Gamma}\ast\eta  ) = \Id_{\oO_{\Gamma}}$.
		
		As $\oO_{\Gamma^t} \simeq (i\times \Id_D)_*\oO_{\Delta^D}$, the composite $\oO_{\Gamma^t}\xrightarrow{ \eta\ast\oO_{\Gamma^t} } \oO_{\Gamma^t} \ast \oO_{\Gamma} \ast \oO_{\Gamma^t} \xrightarrow{\oO_{\Gamma^t}\ast\ve  } \oO_{\Gamma^t}$ reads:
		\begin{align*}
		(i \times \Id_D)_*\oO_{\Delta^D}\xrightarrow{\eta\ast \oO_{\Gamma^t} } (i  \times \Id_D)_*(i \times \Id_D)^*(i \times \Id_D)_*\oO_{\Delta^D} \xrightarrow{\oO_{\Gamma^t}\ast\ve   } (i\times \Id_D)_*\oO_{\Delta^D}.
		\end{align*}		
		The arguments as above show that $(\oO_{\Gamma^t}\ast\ve  ) \circ (\eta \ast \oO_{\Gamma^t}) $ is the identity on $\oO_{\Gamma^t}$.

		Now we calculate the twist and cotwist of the 2-categorical adjunction $(\oO_{\Gamma^t},\oO_{\Gamma}, \eta,\ve)$. Morphism $\eta$ in (\ref{eqtn_eta_for_divisor}) is induced by the morphism $\oO_{X}\to \oO_D$ with kernel $\iI_D$. Hence,
		\begin{align*}
		\textrm{diag}_*^X \iI_D \to \textrm{diag}_*^X\oO_{X} \xrightarrow{\eta} \textrm{diag}_*^X\oO_D \to \textrm{diag}_*^X \iI_D[1] 
		\end{align*}
		is an exact triangle in $\dD_{\textrm{qc}}(X\times X)$. Thus, the cotwist equals $\textrm{diag}_*^X\iI_D$. It is an equivalence whose inverse is $\textrm{diag}_*^X\iI_D^{-1}$. The corresponding functorial exact triangle  is (\ref{eqtn_second_triangle_for_Lii}). 
		
		Morphism $\ve$ in (\ref{eqtn_epsilon_for_divisor}) is induced by the truncation to the 0'th cohomology of the object $(i\times\Id_D)^*(i\times\Id_D)_*\oO_{\Delta^D}$. As we noted before, $(i\times\Id_D)^*(i\times\Id_D)_*\oO_{\Delta^D}$ has two non-zero cohomology sheaves and the truncation to the 0'th cohomology yields an exact triangle
		$$
		\oO_{\Delta^D}\otimes \iI_{D\times D}|_{D\times D}[1] \to (i\times\Id_D)^*(i\times\Id_D)_*\oO_{\Delta^D} \to \oO_{\Delta^D}\to \oO_{\Delta^D}\otimes \iI_{D\times D}|_{D\times D}[2].
		$$
		The twist $\oO_{\Delta^D} \otimes \iI_{D\times D}|_{D\times D}[1] \simeq \textrm{diag}^D_*\iI_D|_D[1]$ is an equivalence whose inverse is $\textrm{diag}^D_*\iI_D^{-1}|_D[-1]$. 
		The corresponding functorial exact triangle is (\ref{eqtn_triangle_for_Lii}). 
	\end{proof}

	\vspace{0.3cm}
	\subsection{$L^1f^*$ vanishing and consequences}\label{ssec_L1_vanish_and_conseq}~\\
	
	Vanishing of $L^1f^*f_*(-)$ is local on $Y$, therefore throughout this section we assume that $Y$ is affine and fix a closed embedding $i\colon Y \to \yY$ of $Y$ as the zero locus of a regular function on a smooth affine $\yY$. Denote $g= i \circ f$, i.e. consider a commutative diagram
		\[
		\xymatrix{X \ar[dr]^g \ar[d]_f & \\ Y \ar[r]^i & \yY.}
		\]
	Since $Y\subset \yY$ is a principal Cartier divisor, the sheaf $\iI_Y$ is isomorphic to $\oO_{\yY}$.
	
	The following lemma is one of the key technical points of this paper.
	\begin{LEM}\label{lem_L1ffOD_i=0}
		Let $f\colon X\to Y$ satisfy (r) and let $Y = \Spec R$ be a spectrum of a complete Noetherian local ring with hypersurface singularities. Let further $D_i \subset X$  be as in Remark (\ref{rem_divisors_D_i}). Then $L^1f^*f_* \oO_{D_i} \simeq 0$.
	\end{LEM}

	\begin{proof}
		Let $\iota_i \colon D_i \to X$ denote the embedding of $D_i$ into $X$. We denote by $h$ the composite $h=if\iota_i \colon D_i \to \mathcal{Y}$.
		
		By Grothendieck duality (\ref{eqtn_G-Vdual}), we have
		\begin{equation}\label{eqtn_111}
		\eE xt^i_{\yY}(Rh_* \oO_{D_i}, \oO_{\yY}) \simeq Rh_* \eE xt^i_{D_i}(\oO_{D_i}, h^!\oO_{\yY}) \simeq Rh_* \hH_{D_i}^i(h^! \oO_{\yY}).
		\end{equation}
		Divisor $D_i$ is Cartier in a Gorenstein scheme, hence it is Gorenstein itself. Since $\yY$ is smooth and a spectrum of a complete ring, we have $\omega_{\yY}^\bcdot \simeq \oO_{\yY}[n+1]$.
		It follows that $h^! \oO_Y=h^!\omega_{\yY}^\bcdot[-n-1]=\omega_{D_i}[-2]$ is a sheaf in homological degree $2$. Since, morphism $h$ is finite, $h_*$ is exact and $Rh_* \hH_{D_i}^i(h^! \oO_{\yY}) \simeq h_* \hH_{D_i}^i(h^! \oO_{\yY})$. Hence, by (\ref{eqtn_111}), we have:
		$$
		\eE xt^i_{\mathcal{Y}}(h_* \oO_{D_i}, \oO_{\mathcal{Y}}) = 0, \, \textrm{ for }\, i\neq 2.
		$$
		
		Since $\mathcal{Y}$ is affine, $\Ext_{\yY}^i(h_*\oO_{D_i}, \oO_{\mathcal{Y}}) = 0$, for $i\neq 2$. Hence, projective dimension of $h_* \oO_{D_i}$ as an $\oO_{\yY}$ module equals two. In other words, $h_* \oO_{D_i}$ has a locally free resolution 	
		$$
		0 \to \eE_{-2} \to \eE_{-1} \to \eE_0 \to h_* \oO_{D_i} \to 0.
		$$
		
		We denote by $\eE_{\bcdot}$ the complex $0 \to \eE_{-2} \to \eE_{-1} \to \eE_0 \to 0$. Cohomology of complex $i_*i^* \eE_{\bcdot}$ is equal to $\textrm{Tor}^{\yY}(h_* \oO_{D_i},\oO_Y)$. Since $Y \subset \mathcal{Y}$ is a Cartier divisor, $\oO_Y$ has a resolution of length two on $\yY$, i.e. the zeroth and first Tor-group only can be non-zero. As functor $i_*$ is exact, we conclude that the morphism $i^*\eE_{-2} \to i^* \eE_{-1}$ is injective.
		
		Morphism $f$ is an isomorphism outside a closed subset of codimension greater than one. Hence, the kernel of $f^*i^* \eE_{-2} \to f^*i^* \eE_{-1}$ is supported in codimension greater than one. Since $f^*i^* \eE_{-2}$ is locally free, it has no torsion, which shows that morphism $f^*i^* \eE_{-2} \to f^*i^* \eE_{-1}$ is also injective, thus complex $\wt{\eE}_\bcdot := f^*i^* \eE_{\bcdot}$ has zero cohomology except for $\hH^{0}_X(\wt{\eE}_{\bcdot})$ and $\hH^{-1}_X(\wt{\eE}_{\bcdot})$.

		Let $E\subset X$ be the exceptional locus of $f$. Assume $L^1f^*f_* \oO_{D_i} \neq 0$. Its support is contained in $E$. Moreover, since $L^1f^*f_* \oO_{D_i} \in \mathscr{A}_f$ (see Lemma \ref{lem_Lif_in_A_f}), the support of this sheaf should contain at least one component of the curve in the fiber over the closed point of $Y$ (see Theorem \ref{thm_fiberstructure} for the structure of the fiber). Therefore, there exists a closed point $x \in E \setminus D_i$ in the support of $L^1f^*f_* \oO_{D_i}$.
		
		By Theorem \ref{thm_triangle_for_ii} for divisor $Y$ in $\yY$, object $i^* \eE_{\bcdot} \simeq Li^*i_*f_* \oO_{D_i}$ fits into an exact triangle
		$$
		f_* \oO_{D_i}[1] \to i^* \eE_{\bcdot} \to f_* \oO_{D_i} \to f_* \oO_{D_i}[2].
		$$
		By applying functor $L f^*$ to this triangle, we obtain an exact sequence
		$$
		f^* f_* \oO_{D_i} \to \mathcal{H}_X^{-1}(\wt{\eE}_{\bcdot}) \to L^1f^*f_* \oO_{D_i}\to 0.
		$$
		Sheaf $f^*f_*\oO_{D_i}$ is supported on the preimage of $f(D_i) \subset Y$, i.e. on the set $E \cup D_i$. Let us restrict our attention to the vicinity $\wt{i} \colon \wt{X}\hookrightarrow X$ of point $x$. The support of both sheaves $f^*f_* \oO_{D_i}$ and $L^1f^*f_*\oO_{D_i}$ is contained in $E$. Therefore, $\hH_X^{-1}(\wt{\eE})$ is a non-zero sheaf with support of some codimension $l>1$.
		
		Let $j\colon \wt{X}\to \mathcal{X}$ be a closed embedding into a smooth variety of dimension $m$. We have $\omega_{\wt{X}} \simeq \wt{i}^!f^! \oO_Y \simeq \wt{i}^! \oO_X \simeq \oO_{\wt{X}}$, hence Grothendieck duality (\ref{eqtn_G-Vdual}) gives
		$$
		\eE xt^{\bcdot}_{\mathcal{X}}(j_* \wt{\eE}_{\bcdot}, \omega_{\mathcal{X}}) \simeq j_* \eE xt^{\bcdot}_{\wt{X}}(\wt{\eE}_\bcdot,\oO_{\wt{X}}[n-m]).
		$$
		Complex $\wt{\eE}_{\bcdot}$ consists of three locally free sheaves, thus $\eE xt^k_X(\wt{\eE}_{\bcdot}, \oO_X)$ is zero, for $k>2$. It follows that $\eE xt^k_{\mathcal{X}}(j_* \wt{\eE}_\bcdot, \omega_{\mathcal{X}}) = 0$, for $k>m-n+2$. By applying functor $\hH om^{\bcdot}_{\mathcal{X}}(-, \omega_{\mathcal{X}})$ to the exact triangle
		$$
		j_*\hH^{-1}_X(\wt{\eE}_{\bcdot})[1] \to j_*\wt{\eE}_{\bcdot} \to j_*\hH^0_X(\wt{\eE}_{\bcdot}) \to j_*\hH^{-1}_X(\wt{\eE}_{\bcdot})[2],
		$$
		we obtain an isomorphism
		\begin{equation}\label{eqtn_iso_ext}
		\eE xt^k_{\mathcal{X}}(j_*\hH^{-1}_X(\wt{\eE}_{\bcdot}), \omega_{\mathcal{X}}) \xrightarrow{\simeq}  \eE xt^{k+2}_{\mathcal{X}}(j_*\hH^0_X(\wt{\eE}_{\bcdot}), \omega_{\mathcal{X}}),
		\end{equation}
		for any $k\geq m-n+2$. Sheaf $j_* \hH^{-1}_X(\wt{\eE}_{\bcdot})$ is non-zero with support of codimension $l+m-n \geq m-n+2$. It follows that sheaf $\eE xt^{l+m-n}_{\mathcal{X}}(j_*\hH^{-1}_X(\wt{\eE}_{\bcdot}), \omega_{\mathcal{X}})$ is non-zero with support of codimension $l+m-n$, cf. \cite[Proposition 1.1.6]{HuyLeh}.	
		
		On the other hand, the codimension of the support of $\eE xt^{l+m-n+2}_{\mathcal{X}}(j_*\hH^0_X(\wt{\eE}_{\bcdot}), \omega_{\mathcal{X}})$ is greater than or equal to $l+m-n+2$, which contradicts (\ref{eqtn_iso_ext}).
	\end{proof}

	\begin{LEM}\label{lem_ses_for_M_N_and_D}
		Let $f\colon X \to Y$ satisfy (r) and let $Y = \Spec R$ be a spectrum of a complete Noetherian local ring. Let further $\mM_i$ be the projective object in $\Per{-1}(X/Y)$ as in (\ref{eqtn_ses_def_M_i}) and $\nN_i = \hH om_X(\mM_i, \oO_X)$ the projective object in $\Per{0}(X/Y)$. Then there exist exact sequences
		\begin{align}
		& 0 \to \oO_X^{r_i} \to \mM_i \to \oO_{D_i} \to 0, \label{eqtn_ses_M_O_D}&\\
		& 0 \to \nN_i \to \oO_X^{r_i} \to \oO_{D_i} \to 0. \label{eqtn_ses_O_N_D} &
		\end{align}
	\end{LEM}
	\begin{proof}
		The sheaf $\mM_i$ is defined as an extension
		$$
		0 \to \oO_X^{r_i-1} \to \mM_i \to \oO_X(D_i) \to 0.
		$$
		For any divisor $D_i$ one can choose a linearly equivalent divisor $D'_i$ such that the intersection $D_i \cap D'_i$ is empty. Hence, $\oO_{D_i}(D_i) \simeq \oO_{D_i}$ and sequence
		\begin{align*}
		&0 \to \oO_X \to \oO_X(D_i) \to \oO_{D_i} \to 0
		\end{align*}
		is exact. Since $\Ext^1_X(\oO_X, \oO_X) \simeq 0$, we have a commutative diagram with rows and columns exact sequences
		\[
		\xymatrix{0 \ar[r]& \oO_{D_i} \ar[r]^{\simeq} & \oO_{D_i} \\ \oO_X^{r_i-1} \ar[r] \ar[u] & \mM_i \ar[r] \ar[u] & \oO_X(D_i) \ar[u] \\ \oO_X^{r_i-1}\ar[r] \ar[u]^{\simeq} & \oO_X^{r_i} \ar[r] \ar[u] & \oO_X. \ar[u]}
		\]
		It gives exact sequence (\ref{eqtn_ses_M_O_D}). By considering local homomorphisms into the structure sheaf we obtain sequence (\ref{eqtn_ses_O_N_D}).
	\end{proof}
	
	\begin{LEM}\label{lem_L1ffOM_i=0}
		Let $f\colon X\to Y$ satisfy (r) and assume that $Y$ has hypersurface singularities. Then the sheaf $L^1f^*f_* \mM$ is zero, for any $\mM\in \mathscr{P}_{-1}$.
	\end{LEM}
	
	\begin{proof}
		Statement is local in $Y$, therefore we might assume that $Y = \Spec R$ is a spectrum of a complete Noetherian local ring with hypersurface singularities. Then $\mM$ is a direct sum of copies of $\mM_i$ and $\oO_X$. Thus, it suffices to verify the lemma for $\mM_i$, as clearly $L^1f^*\oO_Y = 0$.
		
		Since $R^1f_* \oO_X \simeq 0$, applying $Rf_*$ to sequence (\ref{eqtn_ses_M_O_D}) gives an exact sequence
		$$
		0 \to \oO_Y^{r_i} \to f_* \mM_i \to f_* \oO_{D_i} \to 0.
		$$
		It implies an exact sequence
		$$
		0 \to L^1f^*f_* \mM_i \to L^1f^*f_* \oO_{D_i} \to \oO_X^{r_i}.
		$$
		By Lemma \ref{lem_L1ffOD_i=0}, the sheaf $L^1f^*f_* \oO_{D_i}$ is zero, hence the result.
	\end{proof}
	
	\begin{LEM}\label{lem_ses_P_N}
		Let $f \colon X\to Y$ satisfy (r) and  $Y$ have hypersurface singularities. Let $\nN$ be a locally projective object in $\Per{0}(X/Y)$. Then $f^*f_* \nN$ is torsion-free and sequence
		\begin{equation}\label{eqtn_ses_P_N}
		0 \to f^*f_* \nN \to \nN \to \mathcal{Q} \to 0,
		\end{equation}
		is exact for any object $\mathcal{Q}$ locally projective in $\mathscr{A}_f$.
	\end{LEM}
	\begin{proof}
		First, we show that $f^*f_* \nN$ is torsion free. This can be done locally, i.e. we may assume that $Y$ is a spectrum of a complete Noetherian local ring. By \cite[Lemma 3.5.2]{VdB} sheaf $\nN$ is a direct sum of copies of $\oO_X$ and $\nN_i$, where $\nN_i$ are vector bundles dual to $\mM_i$.
		
		By \cite[Lemma 3.1.2]{VdB}, the sheaf $R^1f_* \nN_i$ is zero. By applying $f_* $ to sequence (\ref{eqtn_ses_O_N_D}), we obtain an exact sequence on $Y$
		$$
		0 \to f_* \nN_i \to \oO_Y^{r_i} \to f_* \oO_{D_i} \to 0.
		$$
		By applying $f^*$, we get
		$$
		0 \to L^1f^*f_* \oO_{D_i} \to f^*f_* \nN_i \to \oO_X^{r_i}.
		$$
		It implies that the torsion of $f^*f_* \nN_i$ is $L^1f^*f_* \oO_{D_i}$, which is zero by Lemma \ref{lem_L1ffOD_i=0}. Hence, sheaf $f^*f_* \nN_i$ is torsion free and the counit morphism $f^*f_* \nN_i \to \nN_i$ is an embedding.
		
		The cokernel of the morphism $f^*f_* \nN \to \nN$ is, by the definition of $\iota_f^*$, isomorphic to $\hH_X^0(\iota_f^* \nN)$, hence $\mathcal{Q}$ is an object in $\mathscr{A}_f$. Let now $E$ be any object in $\mathscr{A}_f$. Then $E$, considered as an object in $\dD^b(X)$, lies in $\Per{0}(X/Y)$, hence $\Ext^1_X(\nN, E) \simeq 0$. Moreover, $f^* \vdash f_*$ adjunction implies that $\Hom_X(f^*f_* \nN, E) \simeq 0$. Thus, by applying higher derived functors of $\Hom_X(-,E)$ to sequence (\ref{eqtn_ses_P_N}), we get $\Ext^1_X(\mathcal{Q}, E) \simeq 0$, which proves that $\mathcal{Q} \in \mathscr{A}_f$ is indeed a projective object.
	\end{proof}
	\begin{REM}\label{rem_iota_N_i}
		In the complete local case, applying $\iota_f^*$ to sequences (\ref{eqtn_ses_M_O_D}) and (\ref{eqtn_ses_O_N_D}) and using the fact that $\iota_f^*(\oO_X) \simeq 0$,  we get an isomorphism $ \hH_X^0(\iota_f^* \nN_i) \simeq \hH_X^{-1}(\iota_f^* \oO_{D_i}) \simeq \hH^{-1}_X(\iota_f^* \mM_i)$. It implies that $\pP$ in formula (\ref{eqtn_def_of_P}) is isomorphic to $Q$ of Lemma \ref{lem_ses_P_N}.
	\end{REM}
	
	The \tr e on $\dD^b(X)$ with heart $\Per{-1}(X/Y)$ is obtained from the standard \tr e by a tilt in the torsion pair $(\tT_{-1}, \fF_{-1})$ as in (\ref{eqtn_def_T_-1}) and (\ref{eqtn_def_F_-1}), see \cite[Lemma 3.1.2]{VdB}. Hence, for any $F \in \dD^b(X)$, sequence
	\begin{equation}\label{eqtn_cohom_after_tilt}
	0 \to \fF_{-1}(\hH_X^{i-1}(F))[1] \to \hH_{{}^{-1}\textrm{Per}}^i(F) \to \tT_{-1}(\hH^i_X(F)) \to 0
	\end{equation}
	is exact in $\Coh(X)$, where by $\tT_{-1}(G)$ and $\fF_{-1}(G)$ we denote respectively torsion and torsion-free part of a coherent sheaf $G$ with respect to the torsion pair $(\tT_{-1}, \fF_{-1})$.
	
	\begin{PROP}\label{prop_ff_is_perv}
		Let $f\colon X\to Y$ satisfy (r) and assume that $Y$ has hypersurface singularities. Let $\mM$ be a projective object in ${}^{-1}\textrm{Per}(X/Y)$. Then
		$$
		\hH^0_{{}^{-1}\textrm{Per}}(Lf^*f_* \mM) = f^*f_* \mM.
		$$
	\end{PROP}
	\begin{proof}
		Lemma \ref{lem_ffE_in_Per} below implies that $f^*f_* \mM$ is an object in $\Coh(X) \cap {}^{-1}\textrm{Per}(X/Y)$, hence $\tT_{-1}(f^*f_* \mM) = f^*f_* \mM$. Moreover, $L^1f^*f_* \mM = 0$ by Lemma \ref{lem_L1ffOM_i=0}. We conclude by sequence (\ref{eqtn_cohom_after_tilt}).
	\end{proof}
	
	\begin{LEM}\label{lem_ffE_in_Per}
		Let $f \colon X\to Y$ satisfy (r) and $E$ be a coherent sheaf on $Y$. Then $f^*E$ is an object in ${}^{p}\textrm{Per}(X/Y)$, for $p = -1,0$. If moreover $E$ is locally free and $Y$ is affine, then $f^* E$ is an object in $\mathscr{P}_p$, for $p=-1,0$.
	\end{LEM}
	
	\begin{proof}
		We have $\textrm{Coh}(X) \cap {}^{0}\textrm{Per}(X/Y) = \tT_{0}$ and $\textrm{Coh}(X) \cap {}^{-1}\textrm{Per}(X/Y) = \tT_{-1}$, for categories $\tT_0$ and $\tT_{-1}$ defined in (\ref{eqtn_def_T_0}) and (\ref{eqtn_def_T_-1}) respectively.
		
		Let $E$ be a coherent sheaf on $Y$. By Lemma \ref{lem_Lif_in_A_f}, sheaf $R^1f_*f^* E$ is zero, hence $f^*E$ is an object in $\Per{0}(X/Y)$. From adjunction
		$$
		\Hom_X(f^* E, \aA_f) \simeq \Hom_Y(E, f_* \aA_f) \simeq 0,
		$$
		it follows that $f^*E$ lies also in $\Per{-1}(X/Y)$.
		
		Now we assume further that $E$ is locally free and $Y$ is affine. Let $E'$ be any object in $\Per{p}(X/Y)$, for $p=-1, 0$. Then
		$$
		\Ext^1_X(f^*E, E') \simeq \Ext^1_X(Lf^* E, E') \simeq \Ext^1_Y(E, Rf_* E').
		$$
		Because $Rf_*$ is exact for the perverse \tr es on $\dD^b(X)$, object $Rf_* E'$ is a coherent sheaf on $Y$. Since $E$ is locally free and $Y$ is affine, $\Ext^1_Y(E, Rf_* E') \simeq 0$, i.e. $f^*E$ is projective in $\Per{p}(X/Y)$.
	\end{proof}

	\begin{REM}\label{rem_f_per=f_for_loc_free}
		For any locally free sheaf $\eE$ on $Y$, we also have an isomorphism $\hH^0_{\Per{-1}(X/Y)}(Lf^*\eE) = f^*\eE$. Indeed, $Lf^* \eE$ is isomorphic to $f^* \eE$ and the latter is an object in $\Per{-1}(X/Y)$, by Lemma \ref{lem_ffE_in_Per}.
	\end{REM}
	
	We shall use Lemma \ref{lem_L1ffOD_i=0} and its corollaries to describe object $Lg^*g_* \mM$, for any $\mM$ in $\mathscr{P}_{-1}$. First, we prove that sheaf $g_* \mM$ has a short locally free resolution.
	
	Recall that a morphism $f\colon X\to Y$ is \emph{crepant} if $Lf^*(\omega_Y^\bcdot)\simeq \omega_X^\bcdot.$
	
	\begin{LEM}
	A proper morphism $f\colon X\to Y$ of Gorenstein varieties is crepant if and only if $f^!(\oO_Y) \simeq \oO_X$.
	\end{LEM}
\begin{proof}
	Since $X$ and $Y$ are Gorenstein, $\omega_X^\bcdot$ and $\omega_Y^\bcdot$ are, up to shift, line bundles. In particular, $\omega_Y^\bcdot$ is a perfect complex, hence, by Lemma \ref{lem_two_def_of_f_for_perf_compl}, $f^!(\omega_Y^\bcdot) \simeq f^!(\oO_Y) \otimes Lf^*(\omega_Y^\bcdot)$. Isomorphism $\omega_X^\bcdot \simeq f^!(\omega_Y^\bcdot) \simeq f^!(\oO_Y)\otimes Lf^*(\omega_Y)^\bcdot$ implies $f^!(\oO_Y) \simeq \omega_X^\bcdot \otimes (Lf^*(\omega_Y^\bcdot))^{-1}$. Hence, $Lf^*(\omega_Y^\bcdot) \simeq \omega_X^\bcdot$ if and only $f^!(\oO_Y) \simeq \oO_X$.
\end{proof}
	
	\begin{LEM}\label{lem_loc_free_res_of_g_M}
		Let $X$, $Y$ be Gorenstein and $f\colon X\to Y$ a crepant morphism satisfying (r). Assume $Y$ is affine and $i\colon Y \to \yY$ an embedding of $Y$ as a Cartier divisor into a smooth $\yY$. Denote  $g = i\circ f\colon X\to \yY$. Then $g_* \mM$ and $g_* \nN$, for $\mM \in \mathscr{P}_{-1}$, $\nN \in \mathscr{P}_0$, admit locally free resolutions of length 2.
	\end{LEM}
	\begin{proof}
		Let $\mM$ be an object in $\mathscr{P}_{-1}$. Since $\yY$ is affine, the length of the locally free resolution for $g_* \mM$ is $l+1$, for the maximal $l$ such that $\eE xt^l_{\yY}(g_* \mM, \oO_{\yY})$ is non-zero. Thus, the question is local on $\yY$ and, by restricting to a smaller affine open subset, we can assume that $\omega^{\bcdot}_{\yY} \simeq \oO_{\yY}[n+1]$ and $\omega^{\bcdot}_{Y} \simeq \oO_{Y}[n]$. Then $\omega_X^{\bcdot} \simeq \oO_X[n]$, because $f$ is crepant.
		Let $\nN = R \hH om_X(\mM, \oO_X)=\hH om_X(\mM, \oO_X)$. It is an object in $\mathscr{P}_0$, see \cite[Proposition 3.2.6]{VdB}.  As $\nN$ is an object in $\Per{0}(X/Y)$, we have $Rg_* \nN \simeq g_* \nN$. Then we have by Grothendieck duality (\ref{eqtn_G-Vdual}):
		$$
		R\hH om_{\yY}(g_* \mM, \oO_{\yY})\simeq R\hH om_{\yY}(Rg_* \mM, \omega_{\yY})[-1] \simeq Rg_* R \hH om_X(\mM, \omega_X)[-1]\simeq g_* \nN[-1].
		$$
		Hence, $g_* \mM$ is of projective dimension one. Analogously, for $\nN \in \mathscr{P}_0$, its dual $\mM = R\hH om_X(\nN,\oO_X)$ lies in $\mathscr{P}_{-1}$ and $R\hH om(g_* \nN, \oO_{\yY}) \simeq g_* \mM[-1]$.
	\end{proof}

	\begin{PROP}\label{prop_Lgg_M}
		Let $f\colon X\to Y$ be as in Lemma \ref{lem_loc_free_res_of_g_M} and let $\mM$ be a projective object in ${}^{-1}\textrm{Per}(X/Y)$. Assume further that $Y \subset \yY$ is a principal Cartier divisor. Then object $Lg^*Rg_*\mM$ has cohomology
		$$
		\hH^i_X(Lg^*Rg_* \mM) = \left\{ \begin{array}{ll}\mM, & \textrm{ if }\, i=-1,\\ f^*f_* \mM, & \textrm{ if } i=0, \\ 0, & \textrm{otherwise.} \end{array} \right.
		$$
		The same cohomology are for the \tr e with heart $\Per{-1}(X/Y)$.
	\end{PROP}
	
	\begin{proof}
		By Lemma \ref{lem_loc_free_res_of_g_M}, sheaf $g_*\mM$ has a locally free resolution
		$$
		0 \to \eE_{-1} \to \eE_0 \to g_* \mM \to 0.
		$$
		By Theorem \ref{thm_triangle_for_ii} for $Y\subset \yY$, triangle
		$$
		f_* \mM[1] \to Li^*g_* \mM \to f_* \mM \to f_*\mM[2]
		$$
		is exact. As $Li^* g_* \mM$ is quasi-isomorphic to the complex $i^* \eE_{-1} \to i^* \eE_{0}$ (Lemma \ref{lem_loc_free_res_of_g_M}), we get an exact sequence
		\begin{equation}\label{eqtn_fC}
		0 \to f_*\mM \xrightarrow{\alpha} i^* \eE_{-1} \xrightarrow{\beta} i^* \eE_0 \xrightarrow{\gamma} f_* \mM \to 0
		\end{equation}
		on $Y$.
		
		Let $\wt{\alpha} \colon \mM \to f^*i^*\eE_{-1}$ be a morphism which corresponds to $\alpha$ under the sequence of isomorphisms
		$$
		\Hom_X(\mM, f^*i^* \eE_{-1}) \simeq \Hom_X(\mM, f^! i^*\eE_{-1}) \simeq \Hom_Y(Rf_* \mM, i^*\eE_{-1}) \simeq \Hom_Y(f_*\mM, i^*\eE_{-1}).
		$$
		Above, we use the fact that $f$ is crepant, i.e. $f^!(\oO_Y) \simeq \oO_X$, which yields an isomorphism of functors $f^!|_{\textrm{Perf}(Y)} \simeq Lf^*|_{\textrm{Perf}(Y)}$, see Lemma \ref{lem_two_def_of_f_for_perf_compl}.
		
		Since $\Hom_X(\pP, f^*i^* \eE_{-1}) \simeq \Hom_X(\pP, f^! i^*\eE_{-1}) \simeq \Hom_Y(Rf_* \pP, i^*\eE_{-1})= 0$, sequence (\ref{eqtn_P_M_in_Coh}) implies that morphism $\wt{\alpha}$ is a unique morphism such that the composite $f^*f_* \mM \rightarrow \mM\xrightarrow{\wt{\alpha}} f^* i^*\eE_{-1}$ equals $f^* \alpha$. Since morphism $f^*f_* \mM \to \mM$ is surjective, it follows from $f^*\beta \circ f^* \alpha =0$ that the composition $f^*\beta \circ \wt{\alpha}$ is zero. We consider the complex in degrees $[-3,0]$
		\begin{equation}\label{eqtn_compl_C}
		C^\bcdot \,=\, \{ \,0 \rightarrow \mM \xrightarrow{\wt{\alpha}} f^* i^* \eE_{-1} \xrightarrow{\wt{\beta}} f^*i^* \eE_0 \xrightarrow{\wt{\gamma}} f^*f_* \mM \rightarrow 0\,\},
		\end{equation}
		where $\wt{\beta} = f^*\beta$ and $\wt{\gamma} = f^* \gamma$. In order to prove the proposition, we will show that complex $C^{\bcdot}$ is exact. All its terms are coherent sheaves and lie in ${}^{-1}\textrm{Per}(X/Y)$, by Lemma \ref{lem_ffE_in_Per} and Proposition \ref{prop_ff_is_perv}.
		
		By applying functor $Rf_*$ to $C^{\bcdot}$, we obtain, in view of the projection formula (Lemmas \ref{lem_projection_formula} and \ref{lem_Lif_in_A_f}), exact sequence (\ref{eqtn_fC}). Morphism $f$ has fibers of relative dimension bounded by one, hence $Rf_* C^\bcdot = 0$ implies that $Rf_* \hH_X^i(C^\bcdot) = 0$, for $i\in \mathbb{Z}$ (see Lemma \ref{lem_A_f_is_a_heart}). Functor $Rf_*|_{\Per{-1}(X/Y)}$ is exact, thus cohomology sheaves of $C^{\bcdot}$ in the $-1$-perverse \tr e are also objects in $\mathscr{A}_f$.
		
		Since $f^*$ is right exact,
		$\hH^{-1}_X(C^\bcdot)$ and $\hH^0_X(C^\bcdot)$ are both zero. Moreover, since $\hH^{-3}_X(C^\bcdot)$ is an object in $\aA_f$, it is torsion. Since $\mM$ is locally free, it has no torsion and we have $\hH_X^{-3}(C^\bcdot) \simeq 0$.
		Thus, $C^\bcdot$ is an object in $\cC_f^{\lle -2}\cap \cC_f^{\gge -2}$, which means that cohomology in $\Per{-1}(X/Y)$ can be non-zero in degree $-1$ only.
		By Proposition \ref{prop_ff_is_perv} and Remark \ref{rem_f_per=f_for_loc_free}, sequence
		$$
		f^*i^* \eE_{-1} \to f^*i^* \eE_0 \to f^*f_* \mM \to 0
		$$
		of objects in $\Per{-1}(X/Y)$ is isomorphic to
		\begin{equation}\label{eqtn_a}
		\hH^0_{\Per{-1}} Lf^* i^*\eE_{-1} \to \hH^0_{\Per{-1}} Lf^* i^* \eE_{0} \to \hH^0_{\Per{-1}} Lf^* f_* \mM \to 0.
		\end{equation}
		Functor $\hH^0_{\Per{-1}(X/Y)}(Lf^*(-)) \colon \Coh(Y) \to \Per{-1}(X/Y)$ is left adjoint to an exact functor $Rf_*|_{\Per{-1}(X/Y)} \colon \Per{-1}(X/Y) \to \Coh(Y)$, hence it is right exact. Thus, sequence (\ref{eqtn_a}) is exact. This shows that $\hH_{\Per{-1}(X/Y)}^{-1}(C^\bcdot) \simeq 0$, which finishes the proof.
	\end{proof}

	Proposition \ref{prop_Lgg_M} implies that triangle
	\begin{equation}\label{eqtn_LggM}
	\mM[1] \to Lg^*g_* \mM \to f^*f_* \mM \to \mM[2]
	\end{equation}
	is exact. By Lemma \ref{lem_loc_free_res_of_g_M}, object $Lg^*g_* \mM$ is quasi-isomorphic to a complex $f^* \fF_{-1} \to f^* \fF_0$, for some locally free sheaves $\fF_{-1}$, $\fF_0$ on $Y$. Hence, sequence
	\begin{equation}\label{eqtn_res_M}
	0 \to \mM \to f^* \fF_{-1} \to f^* \fF_0 \to f^*f_* \mM \to 0
	\end{equation}
	is a projective resolution for $f^*f_* \mM$ in $\Per{-1}(X/Y)$, see Lemma \ref{lem_ffE_in_Per}.

	Note that short exact sequence (\ref{eqtn_P_M_in_Coh}) and Lemma \ref{lem_ffE_in_Per} imply that
	\begin{equation}\label{eqtn_P_M_in_Per}
	0 \to f^*f_* \mM \to \mM \to \pP[1] \to 0
	\end{equation}
	is a short exact sequence in ${}^{-1}\textrm{Per}(X/Y)$.
	
	\subsection{2-periodicity}\label{ssec_periodicity}

	\begin{PROP}\label{prop_2_period_of_Lf}
		Let $f \colon X\to Y$ satisfy (r) and assume that $Y \subset \yY$ is a principal divisor in a smooth affine $\yY$. Then there exists a morphism $Lf^* \to Lf^*[2]$ of functors $\dD_{\textrm{qc}}(Y) \to \dD_{\textrm{qc}}(X)$, which, for any $E$ in $\dD^b(Y)$, is an isomorphism on almost all cohomology sheaves of $Lf^*(E)$.
	\end{PROP}
	\begin{proof}
		By applying $Lf^*$ to functorial exact triangle (\ref{eqtn_triangle_for_Lii}), we obtain a morphism $Lf^*E \to Lf^*E[2]$ with cone $Lf^*Li^*i_*E[1]$. Since $\yY$ is smooth, $\dD^b(\yY)$ coincides with the category of perfect complexes on $\yY$. It follows that the functor $Lf^*Li^*i_*$ takes $\dD^b(Y)$ to $\dD^b(X)$ (even to $\Perf(X)$), hence, for any $E\in \dD^b(Y)$, there exists $l\in \mathbb{Z}$ such that $\hH^i(Lf^*Li^*i_*E) = 0$, for $i < l$. Then $\hH^i Lf^*E \simeq \hH^{i-2} Lf^* E$, for any $i< l$.
	\end{proof}
	
	Let now $f\colon X\to Y$ be a crepant morphism of Gorenstein varieties satisfying (r) and assume that $Y \subset \yY$ is a principal divisor in a smooth $\yY$. Then $\iI_Y\simeq \oO_{\yY}$, hence composing  functorial exact triangle (\ref{eqtn_triangle_for_Lii}) with $Lf^*$ and precomposing with $Rf_*$ gives a functorial exact triangle
	\begin{equation}\label{eqtn_triangle_LfRf_LgRg}
	Lf^*Rf_*[1] \to Lg^*Rg_* \to Lf^*Rf_* \to Lf^*Rf_* [2].
	\end{equation}
	
	Object $Lg^*g_* \mM$ has only two non-zero cohomology sheaves described by Proposition \ref{prop_Lgg_M}, for any $\mM\in\mathscr{P}_{-1}$. Then the long exact sequence of cohomology sheaves for triangle (\ref{eqtn_triangle_LfRf_LgRg}) applied to $\mM$ implies short exact sequence
	$$
	0 \to L^2f^*f_* \mM \to f^*f_* \mM \to \mM \to 0
	$$
	and isomorphisms $L^lf^*f_* \mM \simeq L^{l+2} f^*f_* \mM$, for $l \geq 1$. Hence, we have 2-periodicity
	\begin{equation}\label{eqtn_L2ffM}
	L^lf^*f_* \mM \simeq \left\{ \begin{array}{ll}\pP, & \textrm{for even }\, l \geq 2,\\ 0, &\textrm{for odd }\, l. \end{array} \right.
	\end{equation}

	A statement similar to Proposition \ref{prop_Lgg_M} holds for projective object $\nN$ in $\Per{0}(X/Y)$. To establish this, we, first, need to be able to compare functors $g^!$ and $Lg^*$.
	\begin{LEM}\label{lem_g!_and_Lg}
		Let $f\colon X\to Y$ be a crepant morphism of Gorenstein varieties satisfying (r) and assume that $Y\subset \yY$ is a principal divisor in a smooth affine $\yY$. Then there exists a morphism
		$$
		\nu\colon Lg^*[-1] \rightarrow g^!
		$$
		of functors $\dD_{\textrm{qc}}(\yY) \to \dD_{\textrm{qc}}(X)$ that gives an isomorphism of functors $\dD^b(\yY) \to \dD^b(X)$.
	\end{LEM}
	\begin{proof}
		By Appendix \ref{sec_Grot-Ver-dual}, we have a morphism
		$$
		\nu \colon Lg^*(-) \otimes g^!(\oO_{\yY}) \rightarrow g^!(-).
		$$
		By Grothendieck duality (\ref{eqtn_G-Vdual}), we have:
		$$
		i_*i^! \oO_{\mathcal{Y}} \simeq i_* R\hH om(\oO_Y, i^!(\oO_{\mathcal{Y}})) \simeq R \hHom(i_* \oO_Y, \oO_{\mathcal{Y}}) \simeq i_*\oO_Y(Y)[-1].
		$$
		Since $Y\subset \mathcal{Y}$ is a principal divisor, sheaf $\oO_{\mathcal{Y}}(Y)$ is isomorphic to $\oO_{\mathcal{Y}}$. As $i_*$ is exact, we see that $i^! \oO_{\mathcal{Y}} \simeq \oO_Y[-1]$. Then, since $f$ is crepant, $g^!(\oO_{\yY}) \simeq f^!i^!(\oO_{\yY}) \simeq f^!(\oO_Y)[-1] \simeq \oO_X[-1]$, hence $\nu \colon Lg^*[-1]\to g^!$ is a morphism as in the formulation of the Lemma.
		
		By Lemma \ref{lem_two_def_of_f_for_perf_compl}, $\nu_{F^\bcdot}$ is an isomorphism for any $F^\bcdot \in \Perf(\yY)$. Since $\yY$ is smooth, $\Perf(\yY)$ is equivalent to $\dD^b(\yY)$, which finishes the proof.	
	\end{proof}
	
	\begin{PROP}\label{prop_LgRgN}
		Let $f\colon X\to Y$ be a crepant morphism of Gorenstein varieties satisfying (r) and assume that $Y\subset \yY$ is a principal divisor in a smooth affine $\yY$. Let further $\nN$ be an object in $\mathscr{P}_{0}$. Then object $Lg^*g_*\nN\simeq Lg^*Rg_* \nN$ has cohomology sheaves
		$$
		L^1g^*g_* \nN = \left\{ \begin{array}{ll}\nN,& \textrm{if }\, i =-1,\\ f^*f_* \nN, & \textrm{if }\, i = 0,\\ 0, & \textrm{otherwise.}	 \end{array} \right.
		$$
	\end{PROP}
	
	\begin{proof}
		Note, that since $\nN$ is an object in $\Per{0}(X/Y)$, the sheaf $R^1f_* \nN \simeq 0$. As $i_*$ is exact, we also get an isomorphism $Rg_* \nN\simeq g_* \nN$. By Lemma \ref{lem_loc_free_res_of_g_M}, sheaf $g_* \nN$ has a locally free resolution of length two, hence $Lg^*g_* \nN$ can have non-zero cohomology sheaves in degree $-1$ and $0$ only.
		
		By the definition of $g = i \circ f$, we have $\hH^0_X(Lg^*g_* \nN) \simeq g^*g_* \nN\simeq f^*f_* \nN$. Lemma \ref{lem_g!_and_Lg} implies that $ L^1g^*Rg_* \nN \simeq \hH^0_X g^! Rg_* \nN$. The unit morphism $\Id \to g^!Rg_*$ gives a map $\alpha \colon \nN \to L^1g^*g_* \nN$. By composing $Lf^*$ with the functorial exact triangle (\ref{eqtn_triangle_for_Lii}) and then applying to the object $f_*\nN$, we get that $L^1g^*g_* \nN$ is isomorphic to $\nN$ outside the exceptional locus of $f$. This shows that the kernel of $\alpha$ is a torsion sheaf, hence it is zero, as $\nN$ is locally free. It thus suffices to check that $\alpha$ is an epimorphism.
		
		Morphism $\alpha$ is defined globally. It is enough to check that it is an isomorphism in the completions at closed points of $Y$. Therefore, we assume that morphism $f$ satisfies (c) and $\nN$ is a direct sum of copies of $\nN_i$ and $\oO_X$. Statement for $\oO_X$ follows from the exact triangle
		$$
		\oO_Y[1] \to Li^*i_* \oO_Y \to \oO_Y \to \oO_Y[2],
		$$
		see Theorem \ref{thm_triangle_for_ii}, and the fact that $Lf^* \oO_Y \simeq f^* \oO_Y$. Hence, it remains to show that $L^1g^*g_* \nN_i \simeq \nN_i$.
		
		Functor $f_*$ applied to sequences (\ref{eqtn_ses_M_O_D}) and (\ref{eqtn_ses_O_N_D}) gives short exact sequences on $Y$,
		\begin{align*}
		& 0 \to \oO_Y^{r_i} \to  f_* \mM_i \to f_* \oO_{D_i} \to 0,& & 0 \to f_* \nN_i \to  \oO_Y^{r_i} \to f_* \oO_{D_i} \to 0,&
		\end{align*}
		as $R^1f_* \oO_X \simeq 0 \simeq R^1f_* \nN_i$. Since $Lf^* \oO_Y \simeq \oO_X$, by applying $Lf^*$ to these sequences and considering long exact sequence of cohomology, we see that $L^1f^*f_* \nN_i \simeq L^2f^*f_* \oO_{D_i} \simeq L^2f^*f_* \mM_i \simeq \pP_i$ and $L^2f^*f_* \nN_i \simeq L^3f^*f_* \oO_{D_i} \simeq L^3f^*f_* \mM_i\simeq 0$ (see (\ref{eqtn_L2ffM})).
		
		Considering the sequence of cohomology sheaves of triangle (\ref{eqtn_triangle_LfRf_LgRg}) applied to $\nN$ gives an exact sequence
		$$
		0 \to f^*f_* \nN_i \to L^1g^*g_* \nN_i \to \pP_i \to 0.
		$$
		On the other hand, Lemma \ref{lem_ses_P_N} and Remark \ref{rem_iota_N_i} imply that sequence
		$$
		0 \to f^*f_* \nN_i \to \nN_i \to \pP_i \to 0
		$$
		is exact. Hence, $\alpha$ is an isomorphism.
	\end{proof}
	
	\begin{REM}\label{rem_LiffN}
		From Propositions \ref{prop_2_period_of_Lf} and \ref{prop_LgRgN} it follows that
		$$
		L^lf^*f_* \nN \simeq \left\{\begin{array}{ll} \mathcal{Q},& \textrm{for odd }\, l\geq 1,\\ 0,& \textrm{for even }\,l\geq 2.\end{array} \right.
		$$
	\end{REM}
	
	\begin{REM}\label{rem_proj_res_N}
		Let $f\colon X\to Y$ be a crepant morphism of Gorenstein varieties satisfying (r) such that $Y \subset \yY$ is a principal divisor in a smooth affine $\yY$. Using Proposition \ref{prop_LgRgN}, we construct a projective resolution
		\begin{equation}\label{eqtn_res_N}
		0 \to \nN \to f^* \eE_{-1} \to f^* \eE_0 \to f^*f_* \nN \to 0
		\end{equation}
		of $f^*f_* \nN$ in $\Per{0}(X/Y)$. Here, $\eE_{-1}$ and $\eE_0$ are locally free sheaves on $Y$ such that $Li^*g_* \nN$ is quasi-isomorphic to $\eE_{-1}\to \eE_0$ (see Lemma \ref{lem_loc_free_res_of_g_M}).
	\end{REM}
	
	\begin{PROP}\label{prop_proj_res_of_P_i}
		Let $f\colon X\to Y$ satisfy (a), $\mM$ be in $\mathscr{P}_{-1}$ and $\nN$ in $\mathscr{P}_0$. Then there exist locally free sheaves $\fF_{-1}$, $\fF_0$, $\eE_{-1}$, $\eE_0$ on $Y$ such that
		\begin{align*}
		&0 \to \mM \to f^*\fF_{-1} \to f^* \fF_0 \to \mM \to \pP_{\mM}[1] \to 0&\\
		&0 \to \nN \to f^*\eE_{-1} \to f^* \eE_0 \to \nN \to \pP_{\nN} \to 0
		\end{align*}
		are projective resolutions of $\pP_{\mM}[1] = \hH^{-1}_X \iota_f^* \mM[1]$ and $\pP_{\nN} = \hH^0_X \iota_f^* \nN$ respectively in ${}^{-1}\textrm{Per}(X/Y)$ and $\Per{0}(X/Y)$.
	\end{PROP}
	\begin{proof}
		Composing resolution (\ref{eqtn_res_M}) with sequence (\ref{eqtn_P_M_in_Per}) gives the projective resolution for $\pP_{\mM}[1]$ Similarly, resolution (\ref{eqtn_res_N}) and sequence (\ref{eqtn_ses_P_N}) give projective resolution for $\pP_{\nN}$.
	\end{proof}

	\section{The flop functor and Van den Bergh's functor}\label{sec_flop_and_VdB_functor}
	
	Let now $X$, $X^+$ and $Y$ be quasi-projective Gorenstein varieties of dimension $n$ such that $X$ and $X^+$ are related by a flop over $Y$. We assume that $Y$ has canonical hypersurface singularities of multiplicity two. Note that this condition is satisfied if $Y$ has dimension three and terminal Gorenstein singularities. We assume that fibers of $f$ have relative dimension bounded by one and the exceptional locus of $f$ has codimension greater than one in $X$. Then morphism $f^+$ satisfies the same conditions.
	
	M. Van den Bergh in \cite{VdB} proved an equivalence of $\dD^b(X)$ with $\dD^b(X ^+)$ under these assumptions.
	
	We consider a diagram
	\begin{equation}\label{eqtn_flop_diagrams}
	\xymatrix{& X \times_Y X^+ \ar[dr]^{p^+} \ar[dl]_p &\\ X \ar[dr]_f && X^+ \ar[dl]^{f^+}\\ & Y &}
	\end{equation}
	and the flop functor
	\begin{equation}\label{eqtn_def_flop}
	F= Rp^+_* Lp^*\colon \dD_{\textrm{qc}}(X) \to \dD_{\textrm{qc}}(X^+).
	\end{equation}
	
	In this section we give an alternative description of the flop functor	and the functor $\Sigma$ considered in \cite{VdB} under an extra assumption that $Y$ is affine. We show that these are inverse to each other. Keeping the assumption that $Y$ is affine we also provide an alternative description of the flop-flop functor
	\begin{equation}\label{eqtn_def_FF}
	F^+F  = Rp_*\, Lp^{+*}\, Rp^+_*\, Lp^*\colon \dD_{\textrm{qc}}(X) \to \dD_{\textrm{qc}}(X).
	\end{equation}
	
	 We construct a functorial exact triangle relating $\Sigma$, $F$ and the derived push-forward to a smooth scheme $\yY$.   We conclude that the flop functor induces an equivalence $\dD^b(X) \xrightarrow{\simeq}\dD^b(X^+)$ and, following \cite{Chen}, show that $F$ is an equivalence also in the case when $Y$ is not affine.
	
	\begin{LEM}\label{lem_g_of_O_X}
		Let $g\colon X\to Z$ be a proper surjective morphism with connected fibers. If $Z$ is normal, then $g_* \oO_X \simeq \oO_Z$.
	\end{LEM}
	\begin{proof}
		For Stein decomposition $g \colon X\to \Spec_Z g_* \oO_X \xrightarrow{\f} Z$, morphism $\f$ is finite and has connected fibers, hence it is birational. As a finite birational morphism onto a normal variety is an isomorphism, $\f$ is an isomorphism. It follows that $g_* \oO_X \simeq \oO_Z$.
	\end{proof}
	
	\begin{REM}\label{rem_RpO=O}
		We assume that $X$ and $X^+$ are normal varieties. Since morphism $p$ is proper, surjective and with connected fibers, Lemma \ref{lem_g_of_O_X} above implies that $p_* \oO_{X\times_Y X^+}\simeq \oO_X$. Moreover, by Proposition \ref{prop_vanish_Rpi_Zpi_X_of_E}, the sheaf $R^1p_* \oO_{X\times_Y X^+}$ vanishes. It follows that $p$ satisfies (r).
	\end{REM}
	
	\vspace{0.3cm}
	\subsection{An alternative description of Van den Bergh's functor}\label{ssec_alt_desc_of_VdB}~\\
	
	Recall that a sheaf $F$ is \textit{normal} when, for any open sets $U \subset V$ such that codimension of $V\setminus U$ in $V$ is greater than one, the restriction morphism $F(V) \to F(U)$ is an isomorphism.
	
	\begin{LEM}\label{lem_charct_of_reflexive}\cite[Proposition 1.6]{Har_sta}
		Let $X$ be a normal Noetherian scheme. A sheaf $F$ on $X$ is reflexive if and only if it is torsion-free and normal.
	\end{LEM}

	For a scheme $X$, denote by $\textrm{Ref}\,(X)$ the category of reflexive sheaves.
	
	\begin{LEM}\label{lem_equiv_of_refl}(cf. \cite[Lemma 4.2.1]{VdB})
		Let $f \colon X\to Y$ be a projective birational morphism, $X,Y$ normal Noetherian and the exceptional locus of $f$ have codimension $>1$ in $X$. Then the following functors are mutually inverse equivalences:
		\begin{align*}
		&f_*\colon \textrm{Ref}\,(X) \to \textrm{Ref}\,(Y),& &(f^*(-))^{\vee \vee} \colon \textrm{Ref}\,(Y) \to \textrm{Ref}\,(X).&
		\end{align*}
	\end{LEM}
	
	\begin{proof}
		For a torsion free sheaf $F$ on $X$, its push-forward $f_* F$ is also torsion free. Thus, $f_* F$ will be reflexive if we check that it is normal.
		
		Let $G$ be a torsion-free sheaf on $Y$ which is not normal. There exist open sets $j \colon U \hookrightarrow V$ with the complement of $U$ in codimension greater than one in $V$ such that canonical morphism $G \to j_*j^* G$ is not an isomorphism. Then we have a non-trivial extension
		$$
		0 \to G \to j_*j^*G \to Q \to 0.
		$$
		The question is local, so we can assume $Y$ to be affine. Sheaf $Q$ is supported in codimension greater than one, hence there exists a closed subscheme $Z \subset Y$ of codimension greater than one such that sequence
		$$
		0 \to K \to \oO_{Z}^{\oplus k} \to Q \to 0
		$$
		is exact. Since $G$ is torsion-free, $\Hom(K, G)$ is zero. Hence, vanishing of $\Ext^1_Y(\oO_Z, G)$ implies that $\Ext^1_Y(Q, G)$ is also zero. Thus, in order to show that a torsion-free sheaf $G$ is normal it suffices to check that $\Ext^1_Y(\oO_Z, G)$ vanishes, for any closed $Z \subset Y$ of codimension greater than one.
		
		Group $\Ext^1_Y(\oO_Z, f_* F)$ is zero if and only if morphism $\Hom(\oO_Y, f_* F) \to \Hom(I_Z, f_* F)$ given by short exact sequence
		$$
		0 \to I_Z \to \oO_Y \to \oO_Z \to 0
		$$
		is an isomorphism. By adjunction, $\Hom_Y(I_Z, f_* F)$ is isomorphic to $\Hom_X(f^* I_Z, F)$. Exceptional locus of $f$ has codimension greater than one, hence $f^* I_Z$ is isomorphic to $\oO_X$ in codimension one. Since $F$ is reflexive, this implies $\Hom_X(f^*I_Z, F)= \Hom_X(\oO_X, F)$.
		
		Thus, functor $f_*$ maps reflexive sheaves on $X$ to reflexive sheaves on $Y$. Clearly, the same is true about $(f^*(-))^{\vee \vee}$.
		
		Let now $F$ be an object in $\textrm{Ref}\,(X)$. Adjunction
		implies a morphism $\alpha \colon (f^*f_*F)^{\vee \vee} \to F^{\vee \vee} \simeq F$. Both sheaves are reflexive and $\alpha$ is an isomorphism outside exceptional locus of $f$, which has codimension greater than one. Therefore $\alpha$ is an isomorphism.
		
		For $G \in \textrm{Ref}\,(Y)$, reflexification $f^*G \to (f^*G)^{\vee \vee}$ together with adjunction $G\to f_*f^*G$ give $\beta \colon G \to f_*((f^*G)^{\vee \vee})$. Again, morphism $\beta$ is an isomorphism in codimension one between reflexive sheaves, hence an isomorphism.
		
		Finally, a morphism between reflexive sheaves is determined by its restriction to any open set with the complement of codimension greater than one. Thus, functors $f_*$ and $(f^*(-))^{\vee \vee}$ give isomorphisms on the groups of morphisms between objects in $\textrm{Ref}\,(X)$ and $\textrm{Ref}\,(Y)$.
	\end{proof}

	Let $f\colon X\to Y$ satisfy (a). For $\mM$ an object in $  \mathscr{P}_{-1}$, put
	\begin{equation}\label{eqtn_def_N+}
	\nN^+ := (f^{+*}f_* \mM)^{\vee \vee}.
	\end{equation}
	By Lemma \ref{lem_equiv_of_refl}, sheaf $\nN^+$ belongs to $\textrm{Ref}\,(X^+)$.
	
	\begin{PROP}\label{prop_equiv_of_P_-1_P_0}
		Let $f\colon X\to Y$ satisfy (a) and $f^+ \colon X^+ \to Y$ its flop. Then functor
		$$
		(f^{+*}f_*(-))^{\vee \vee} \colon \textrm{Ref}\,(X) \to \textrm{Ref}\,(X^+)
		$$
		restricts to equivalences $\mathscr{P}_{-1} \xrightarrow{\simeq} \mathscr{P}_0^+$ and $\mathscr{P}_0 \xrightarrow{\simeq} \mathscr{P}_{-1}^+$.
	\end{PROP}
	
	\begin{proof}
		By \cite[Proposition 3.2.6]{VdB}, categories $\mathscr{P}_{-1}$ and $\mathscr{P}_0$ are subcategories of $\textrm{Ref}\,(X)$. By \cite[Proposition 4.3.1]{VdB}, subcategories $f_* \mathscr{P}_{-1}$ and $f^+_* \mathscr{P}_0^+$ of $\textrm{Ref}\,(Y)$ are equivalent. By Lemma \ref{lem_equiv_of_refl}, the subcategory $(f^{+*}f_* \mathscr{P}_{-1})^{\vee \vee}$ of $\textrm{Ref}\,(X^+)$ corresponds to $\mathscr{P}_{-1}$ under equivalences $\textrm{Ref}\,(X)\simeq \textrm{Ref}\,(Y) \simeq \textrm{Ref}\,(X^+)$, hence $\mathscr{P}_0^+ \simeq (f^{+*}f_* \mathscr{P}_{-1})^{\vee \vee}$.
		
Exchanging the roles of $f$ and $f^+$, we get equivaalence of $\mathscr{P}_{-1}^+$ with $\mathscr{P}_0$.
	\end{proof}

	\begin{LEM}\label{lem_Ext(M,M)=0}
		Let $f\colon X\to Y$ satisfy (r), $Y$ affine and $\mM_1$, $\mM_2$ objects in $\mathscr{P}_{p}$, for $p=-1$ or $0$. Then $\Ext_X^i(\mM_1, \mM_2) = 0$, for $i \geq 1$.
	\end{LEM}
	\begin{proof}
		By \cite[Proposition 3.2.6]{VdB}, objects in $\mathscr{P}_{p}$ are vector bundles on $X$. Hence, $\eE xt^i(\mM_1, \mM_2) \simeq 0$, for $i > 0$. Since morphism $f$ has fibers of relative dimension bounded by one, the local-to-global spectral sequence implies that $\Ext^i(\mM_1, \mM_2)\simeq H^0(Y, R^if_* \hH om(\mM_1, \mM_2))$ vanishes, for $i> 1$. Group $\Ext^1_X(\mM_1,\mM_2)$ is isomorphic to $\Ext^1_{\Per{p}(X/Y)}(\mM_1, \mM_2)$, hence also zero, as $\mM_1$ is an object of $\mathscr{P}_p$.
	\end{proof}
	
	Let $\mathscr{P} \subset \bB$ be an exact subcategory of an abelian  category. Denote by $\textrm{Hot}^{-,b}(\mathscr{P})$ the homotopy category of bounded above complexes of objects in $\mathscr{P}$ with bounded cohomology in $\bB$. Denote by $\textrm{Hot}^{-,b}(\mathscr{P})^{\lle l}$ the full subcategory of complexes with non-zero cohomology in degree up to $l$.  Denote by $\textrm{Hot}^-(\mathscr{P})$ the category of bounded above complexes of objects in $\mathscr{P}$ without any constraint on cohomology. Finally, denote by $\textrm{Hot}^b(\mathscr{P})$ the full subcategory of bounded complexes in $\textrm{Hot}^{-,b}(\mathscr{P})$. Categories $\textrm{Hot}^{-,b}(\mathscr{P})$, $\textrm{Hot}^-(\mathscr{P})$ and $\textrm{Hot}^b(\mathscr{P})$ are triangulated by \cite{Nee1}. If $\mathscr{P}$ is the category of projective objects in $\Per{p}(X/Y)$, for $p=-1,0$, then we use the above notation for the embedding $\mathscr{P}_p \subset \Coh(X)$.
	
	Assume that morphism $f\colon X\to Y$ satisfies (a).
	Since $\Per{p}(X/Y)$ has enough projective objects, for $p=-1$ or $0$, category $\dD^b(\Per{p}(X/Y))$ is equivalent to $\textrm{Hot}^{-,b}(\mathscr{P}_{p})$, see \cite[Theorem III.5.21]{GelMan}.
	
	\begin{PROP}\label{prop_D(X)=D(Per)}
		Let $f\colon X\to Y$ satisfy (r) and  $Y$ affine. Then categories $\dD^b(X)$ and $\dD^b(\Per{p}(X/Y))$ are equivalent, for $p=-1$ or $0$.
	\end{PROP}
	\begin{proof}
		Since objects in $\mathscr{P}_{p}$ are coherent sheaves on $X$, we have a functor $\Theta\colon \textrm{Hot}^{-,b}(\mathscr{P}_{p})\to\dD^b(X)$.  There exists $N_0$ such that $\Ext^q_{X}(P, C) = 0$, for any $P \in \mathscr{P}_{p}$ and $C \in  \Per{p}(X/Y)$ and $q\geq N_0$. Indeed, $C$ has non-zero cohomology sheaves in degrees $-1$ and $0$ only, and objects in $\mathscr{P}_p$ are locally free sheaves on $X$ \cite[Proposition 3.2.6]{VdB}. Then Proposition \ref{prop_exist_of_spect_seq} implies that cohomology groups of complex  $\prod_{j-i = p} \Hom_X(A^i, B^j)$ are isomorphic to $\Hom_X(A^\bcdot, B^\bcdot)$, for any $A^\bcdot, B^\bcdot \in \Hot^{-,b}(\mathscr{P}_p)$ (see Lemma \ref{lem_Ext(M,M)=0}). Hence, $\Hom^\bcdot_X(A^\bcdot, B^\bcdot) \simeq \Hom_{\Hot^{-,b}(\mathscr{P}_p)}(A^\bcdot, B^\bcdot)$, i.e. $\Theta$ is fully faithful. It is essentially surjective, because $\Per{p}(X/Y)$ is the heart of a bounded \tr e.
	\end{proof}
	
	By Proposition \ref{prop_equiv_of_P_-1_P_0}, functor
	\begin{equation}\label{eqtn_def_of_sigma}
	\Sigma_o := (f^{+*}f_*(-))^{\vee \vee} \colon \mathscr{P}_{-1} \to \mathscr{P}_0^+
	\end{equation}
	is fully faithful. Denote by
	\begin{equation}\label{eqtn_def_Sigma_minus}
	\Sigma \colon \textrm{Hot}^{-,b}(\mathscr{P}_{-1}) \to \textrm{Hot}^{-}(\mathscr{P}_{0}^+)
	\end{equation}
	the functor obtained by applying $\Sigma_o$ term-wise to complexes of objects of $\mathscr{P}_{-1}$.
	
	\begin{THM}\label{thm_VdB_equivalence}
		Let $f\colon X\to Y$ satisfy (a) and 
		$f^+ \colon X^+ \to Y$ be the flop of $f$. Then functor $\Sigma$ induces an equivalence between  $\dD^b(X)$ and $\dD^b(X^+)$.
	\end{THM}
	\begin{proof}
		Lemma \ref{lem_equiv_of_refl} implies that functor
		$$
		T_o := (f^*f^+_*(-))^{\vee \vee} \colon \mathscr{P}^+_0 \to \mathscr{P}_{-1}
		$$
		is the inverse of $\Sigma_o$. Let us extend it term-wise to a functor
		$$
		T \colon \textrm{Hot}^{-,b}(\mathscr{P}_0^+) \to \textrm{Hot}^{-}(\mathscr{P}_{-1}).
		$$
		
		Proposition \ref{prop_equiv_of_P_-1_P_0} implies that $\Sigma$ and $T$ are inverse equivalences between $\mathscr{P}_{-1}$ and $\mathscr{P}_0^+$. Hence, they are also inverse equivalences between $\Hot^-(\mathscr{P}_{-1})$ and $\Hot^-(\mathscr{P}_0)$. By Proposition \ref{prop_D(X)=D(Per)}, we have $\dD^b(X) \simeq \textrm{Hot}^{-,b}(\mathscr{P}_{-1})$ and $\dD^b(X^+) \simeq \textrm{Hot}^{-,b}(\mathscr{P}_0^+)$. Thus, it suffices to show that $\Sigma$ takes $\Hot^{-,b}(\mathscr{P}_{-1})$ to $\Hot^{-,b}(\mathscr{P}_0^+)$ and $T$ takes $\Hot^{-,b}(\mathscr{P}_{0}^+)$ to $\Hot^{-,b}(\mathscr{P}_{-1})$.
		
		Let $E$ be in $\Hot^{-,b}(\mathscr{P}_{-1})$ and assume that $\Sigma(E)$ has unbounded cohomology. If cohomology of $\Sigma(E)$ is unbounded with respect to the standard \tr e it is also unbounded with respect to the \tr e with heart $\Per{0}(X^+/Y)$, because these two \tr es are related by a tilt. Let $\nN^+$ be the projective generator of $\Per{0}(X^+/Y)$. Then $\dim_k \Hom_{X^+}^\bcdot(\nN^+, \Sigma(E))$ is infinite. By adjunction of inverse functors, we have that $\dim_k \Hom_X^\bcdot(T(\nN^+), E)$ is infinite. As $T_o(\nN^+)$ is in $\mathscr{P}_{-1}$, it contradicts boundedness of cohomology of $E$ with respect to the \tr e with heart $\Per{-1}(X/Y)$. Analogously, $T$ takes $\Hot^{-,b}(\mathscr{P}_{0}^+)$ to $\Hot^{-,b}(\mathscr{P}_{-1})$.
	\end{proof}
	
	Let us show that Theorem \ref{thm_VdB_equivalence} may be considered as a write-up of the Van den Bergh's theorem.
	
	Let $f\colon X\to Y$ satisfy (a). Since $f$ has fibers of relative dimension bounded by 1, \cite[Lemma 3.2.2]{VdB} implies that for an $f$-ample line bundle $\lL$ both $\lL \oplus \oO_X$ and $\oO_X \oplus \lL^{-1}$ are compact generators of $\dD_{\textrm{qc}}(X)$. It follows 
	that both the projective generator $\mM = \bigoplus \mM_i$ of $\Per{-1}(X/Y)$ (defined as in Remark \ref{rem_divisors_D_i}) and $\nN= \hHom(\mM, \oO_X)$ of $\Per{0}(X/Y)$ (see Remark \ref{rem_pro_gen_from_N}) are compact generators of the category $\dD_{\textrm{qc}}(X)$. Analogous result holds for projective generators for $\Per{-1}(X^+/Y)$ and $\Per{0}(X^+/Y)$.
	
	As $(f^{+*}f_*(-))^{\vee \vee}$ restricts to an equivalence $\mathscr{P}_{-1} \to \mathscr{P}_0^+$, Proposition \ref{prop_equiv_of_P_-1_P_0}, $\nN^+$ (as in (\ref{eqtn_def_N+})) is a projective generator for $\Per{0}(X^+/Y)$ and the endomorphisms algebras of $\mM$ and $\nN^+$ are isomorphic.
	 
	Moreover, by Lemma \ref{lem_Ext(M,M)=0}, both $\mM$ and $\nN^+$ have no higher self Ext groups. Hence, by the result of B. Keller \cite{Kel1} both $\dD_{\textrm{qc}}(X)$ and $\dD_{\textrm{qc}}(X^+)$ are equivalent to the derived category of DG modules over the algebra $A = \Hom_X(\mM,\mM)$, cf. \cite{BvdB}. 
Denote by
$$
\Sigma_{\textrm{qc}} \colon \dD_{\textrm{qc}}(X) \xrightarrow{\simeq} \dD_{\textrm{qc}}(X^+)
$$ 
	the resulted equivalence.
	\begin{PROP}\label{prop_DG_enha_of_sigma}
		Let $f\colon X\to Y$ satisfy (a). 
	Functor $ \Sigma_{\textrm{qc}}|_{\dD^b(X)}:\dD^b(X)\to \dD^b(X^+)$ is an equivalence   that takes $\Per{-1}(X/Y)$ to $\Per{0}(X^+/Y)$, and it 
	coincides with Van den Bergh's equivalence in \cite[Theorem 4.4.2]{VdB}.
	\end{PROP}
	\begin{proof}
		Category $\Hot^{-,b}(\mathscr{P}_{-1})$ is a subcategory of the homotopy category of complexes of projective $A$-modules; it consists of bounded above complexes.
		As $\Sigma_o$ as in (\ref{eqtn_def_of_sigma}) maps projective generator for $\Per{-1}(X/Y)$ to the projective generator for $\Per{0}(X^+/Y)$, we have $\Sigma_{\textrm{qc}}|_{\mathscr{P}_{-1}} = \Sigma_o$. We conclude that $\Sigma = \Sigma_{\textrm{qc}}|_{\Hot^{-,b}(\mathscr{P}_{-1})}$ is the equivalence defined by M. Van den Bergh in \cite[Theorem 4.4.2]{VdB}.		
	\end{proof}
	Categories $\mathscr{P}_{0}$ and $\mathscr{P}_{-1}^+$ are equivalent, too. We can thus analogously construct an equivalence $\Sigma_{\textrm{qc}}^{-1} \colon \dD^b(X) \to \dD^b(X^+)$ that takes $\Per{0}(X/Y)$ to $\Per{-1}(X^+/Y)$.
	
	\vspace{0.3cm}
	\subsection{An alternative description of the flop functor}\label{ssec_alt_descr_of_flop}~\\

	Consider diagram (\ref{eqtn_flop_diagrams}). The counit of $f^* \dashv f_*$ adjunction
	$$
	f^*f_* \to \Id
	$$
	gives a morphism
	$$
	p^*f^*f_* \to p^*.
	$$	
	Isomorphism $p^*f^* \simeq p^{+*}f^{+*}$ and $p^{+*}\dashv p^+_*$ adjunction lead to a base-change morphism
	\begin{equation}\label{eqtn_def_of_omega}
	\epsilon \colon f^{+*}f_* \to p^+_*p^*.
	\end{equation}
	Note that the morphism $\epsilon$ is by definition the composition
	\begin{equation}\label{eqtn_decomp_of_omega}
	\epsilon \colon f^{+*}f_* \rightarrow p^+_*p^{+*}f^{+*}f_* \rightarrow p^{+}_*p^*.
	\end{equation}
	
	Analogous argument gives a derived base-change
	\begin{equation}\label{eqtn_def_of_der_base_change}
	\wt{\epsilon} \colon Lf^{+*}Rf_* \to Rp^+_* Lp^*.
	\end{equation}
	
	\begin{LEM}\label{lem_omegaM_is_iso}
		Let $f$ satisfy (a) and $\mM$ be an object in $\mathscr{P}_{-1}$. Then $\epsilon_{\mM} \colon f^{+*}f_* \mM \to p^+_*p^* \mM$ is an isomorphism.
	\end{LEM}
	\begin{proof}
		Since $\mM$ is locally free, \cite[Propostion 3.2.6]{VdB}, applying $p^*$ to sequence (\ref{eqtn_P_M_in_Coh}) we get an exact sequence
		$$
		0 \to p^* \mathcal{P} \to p^*f^*f_* \mM \to p^*\mM \to 0
		$$
		on $X\times_Y X^+$. Sheaf $\pP$ is in $\mathscr{A}_f$, hence Proposition \ref{prop_vanish_Rpi_Zpi_X_of_E} implies an isomorphism
		$$
		p^+_*p^*f^*f_*\mM \xrightarrow{\simeq} p^+_*p^* \mM.
		$$
		Since diagram (\ref{eqtn_flop_diagrams}) is commutative, $p^+_*p^*f^*f_* \mM \simeq p^+_*p^{+*}f^{+*}f_* \mM$. Isomorphism $f^{+*}f_*\mM \simeq f^{+*}f^+_* \nN^+$ (see Proposition \ref{prop_equiv_of_P_-1_P_0}) together with Lemma \ref{lem_ses_P_N} implies that $f^{+*}f_*\mM$ is torsion-free. Moreover, by Remark \ref{rem_RpO=O}, $Rp^+_* \oO_{X\times_Y X^+} = \oO_{X^+}$. Thus, conditions of Lemma \ref{lem_projection_formula} are satisfied and we have an isomorphism $f^{+*}f_* \mM \xrightarrow{\simeq} p^+_*p^{+*}f^{+*}f_* \mM$. Hence, $\epsilon_{\mM}$ is a composite of two isomorphisms in (\ref{eqtn_decomp_of_omega}).
	\end{proof}
	
	\begin{PROP}\label{prop_flop=non_der_flop}
		Let $f$ satisfy (a). Then the flop functor $Rp^+_* Lp^*$ on $\mathscr{P}_{-1}$ is isomorphic to the non-derived flop functor $p^+_*p^*$.
	\end{PROP}
	
	\begin{proof}
		Every object in $\mathscr{P}_{-1}$ locally free, \cite[Propostion 3.2.6]{VdB}, hence $Lp^*|_{ \mathscr{P}_{-1}} \simeq p^*|_{\mathscr{P}_{-1}}$. Moreover, since $R^1f_* \mathscr{P}_{-1} \simeq 0$, \cite[Lemma 3.1.2]{VdB} Proposition \ref{prop_vanish_Rpi_Zpi_X_of_E} implies that $R^1p^+_*p^* \mathscr{P}_{-1} \simeq 0$.
	\end{proof}
	
	We define functor
	$$
	f^{+*}f_* \colon \textrm{Hot}^{-,b}(\mathscr{P}_{-1}) \to \textrm{Hot}^{-}(f^{+*}f_* \mathscr{P}_{-1})
	$$
	term-wise.
	We use the same notation for the composite of the above functor with the canonical functor $\textrm{Hot}^{-}(f^{+*}f_* \mathscr{P}_{-1}) \to \dD^-(X^+)$.
	Also term-wise, we define a functor
	$$
	p^+_*p^* \colon \textrm{Hot}^{-,b}(\mathscr{P}_{-1}) \to \dD^-(X^+).
	$$
	
	\begin{PROP}\label{prop_iso_of_ff_and_pp_for_flop}
		Let $f$ satisfy (a). Then functors $f^{+*} f_*$ and $p^+_*p^*$ are isomorphic on the category $\textrm{Hot}^{-,b}(\mathscr{P}_{-1})$. They are also isomorphic to the flop functor restricted to $\dD^b(X)$ under the equivalence $\dD^b(X) \simeq \textrm{Hot}^{-,b}(\mathscr{P}_{-1})$.
	\end{PROP}
	\begin{proof}
		Proposition \ref{prop_flop=non_der_flop} implies that the canonical morphism $p^+_*p^*|_{\mathscr{P}_{-1}} \to Rp^+_* p^*|_{\mathscr{P}_{-1}}$ is an isomorphism. All objects in $\mathscr{P}_{-1}$ are locally free, \cite[Propostion 3.2.6]{VdB}, hence canonical morphism $Rp^+_*Lp^*|_{\mathscr{P}_{-1}} \to Rp^+_*p^*|_{\mathscr{P}_{-1}}$ is also an isomorphism. This implies isomorphism of $p^+_*p^*|_{\mathscr{P}_{-1}}$ and $Rp^+_*Lp^*|_{\mathscr{P}_{-1}}$.
		
		Lemma \ref{lem_omegaM_is_iso} assures that the base-change $\epsilon_{\mM} \colon f^{+*}f_* \mM \to p^+_*p^* \mM$  is an isomorphism, for all $\mM\in \mathscr{P}_{-1}$.
		
		Thus all three functors are isomorphic on $\mathscr{P}_{-1}$. Induction on triangles shows that the isomorphism extends to $\Hot^b(\mathscr{P}_{-1})$. By Lemma \ref{lem_iso_on_perf_is_iso_on_Db} below they are also isomorphic on $\textrm{Hot}^{-,b}(\mathscr{P}_{-1})$.
	\end{proof}
	In the proof of Proposition \ref{prop_iso_of_ff_and_pp_for_flop} we use the fact that in order to check that a natural transformation gives an isomorphism of two functors defined on $\textrm{Hot}^{-,b}(\mathscr{P}_{-1})$, it suffices to check that it gives an isomorphism of the functors restricted to $\Hot^b(\mathscr{P}_{-1})$, which we prove now.
	
	\begin{LEM}\label{lem_iso_on_perf_is_iso_on_Db}
		Let $\aA$ be an abelian category, $\mathscr{P}\subset \bB$ an embedding of an exact subcategory and $F, G \colon \textrm{Hot}^{-,b}(\mathscr{P}) \to \dD^-(\aA)$ exact functors. Assume there exists $n_0$ such that both $F$ and $G$ map $\textrm{Hot}^{-,b}(\mathscr{P})^{\lle 0}$ to $\dD^-(\aA)^{\lle n_0}$ and that the image of either $F$ or $G$ is contained in $\dD^b(\aA)$. Let $\kappa \colon F|_{\textrm{Hot}^b(\mathscr{P})} \to G|_{\textrm{Hot}^b(\mathscr{P})}$ be an isomorphism of functors. Then $\kappa$ admits an extension to a functorial isomorphism $\wt{\kappa} \colon F\to G$.
	\end{LEM}
	
	\begin{proof}
		
		Let $E$ be an object in $\textrm{Hot}^{-,b}(\mathscr{P})$. We denote by $\sigma_{\gge k}E$ the ``stupid'' truncation:
		$$
		\sigma_{\gge k}(E)^i = \left\{\begin{array}{cl}E^i, & \textrm{if }\, i \geq k,\\ 0, & \textrm{otherwise.} \end{array} \right.
		$$
		Morphism $\sigma_{\geq l}E \to E$ induces isomorphisms $\tau_{\geq m} F(\sigma_{\geq l}E) \xrightarrow{\simeq} \tau_{\geq m} F(E)$, $\tau_{\geq m} G(\sigma_{\geq l}E) \xrightarrow{\simeq} \tau_{\geq m} G(E)$, for any $m\in \mathbb{Z}$, and $l\leq m-n_0$.
		We have
		\begin{equation}\label{eqtn_isos}
		\tau_{\geq m} F(E) \simeq \tau_{\geq m} F(\sigma_{\geq l} E) \simeq \tau_{\geq m} G(\sigma_{\geq l} E) \simeq \tau_{\geq m} G(E),
		\end{equation}
		for any $l\leq m-n_0$. Let us assume that $F(E)$ is an object in $\dD^b(\aA)$. There exists $N$ such that morphism $F(E) \to \tau_{\geq m} F(E)$ is an isomorphism, for any $m\leq N$. Thus, isomorphisms (\ref{eqtn_isos}) imply that $G(E)$ is also an object in $\dD^b(\aA)$. We define $\wt{\kappa}$ as the composite of isomorphisms
		$$
		F(E) \rightarrow \tau_{\geq m} F(E) \rightarrow \tau_{\geq m} F(\sigma_{\geq l} E) \xrightarrow{\tau_{\geq m}\kappa_{\sigma_{\geq l}E} }\tau_{\geq m} G(\sigma_{\geq l} E) \rightarrow \tau_{\geq m} G(E) \rightarrow G(E),
		$$
		for any $m\leq N$ and $l\leq m-n_0$. Since $\kappa$ is a natural transformation, $\tau_{\geq m} \kappa_{\sigma_{\geq l}E} \simeq \tau_{\geq m} \kappa_{\sigma_{\geq l'} E}$, for any $l,l'\leq m-n_0$ . Hence, $\wt{\kappa}$ does not depend on the choice of $m\leq N$ and $l\leq m-n_0$.
	\end{proof}
	
	\vspace{0.3cm}
	\subsection{A functorial exact triangle of functors $\dD_{\textrm{qc}}(X) \to \dD_{\textrm{qc}}(X^+)$}\label{ssec_dist_triang_func_D(X)_D(X+)}~\\

	Assume that $f\colon X\to Y$ satisfies (a), i.e. there exists a closed 	embedding $i\colon Y \to \yY$, for a smooth $\yY$  of dimension $n+1$. Then, $X\times_Y X^+\simeq X\times_{\yY} X^+$, i.e. diagram
	\begin{equation}\label{eqtn_diagram_flop_2}
	\xymatrix{& X\times_Y X^+ \ar[dl]_{p} \ar[dr]^{p^+} & \\ X \ar[dr]_{g} & & X^+ \ar[dl]^{g_+} \\ & \yY&}
	\end{equation}
	is fibered.

	Recall, that we denote by $F= Rp^+_*Lp^*$ the flop functor and by $\Sigma\colon \Hot^{-,b}(\mathscr{P}_{-1})\to \dD^b(X^+)$ the term-wise extension of $\Sigma_o = (f^{+*}f_*(-))^{\vee \vee}\colon \mathscr{P}_{-1} \to \mathscr{P}_0^+$.

	\begin{LEM}\label{lem_Sigma_commutes_with_Rf}
		Let $f\colon X\to Y$ satisfy (a). Functors $Rf^+_* \Sigma,Rf_*\colon \Hot^{-,b}(\mathscr{P}_{-1}) \to \dD^b(Y)$ are isomorphic.
	\end{LEM}
	
	\begin{proof}
		Consider morphism $\alpha\colon f_*|_{\mathscr{P}_{-1}}\to f^+_* \Sigma_o|_{\mathscr{P}_{-1}}$ defined as the composite
		$$
		f_* \xrightarrow{\eta f_*} f^+_*f^{+*}f_* \xrightarrow{f^+_* \beta f^{+*}f_*} f^+_*(f^{+*}f_*(-)^{\vee \vee}) = f^+_* \Sigma_o,
		$$
		for the unit $\eta$ of $f^{+*}\dashv f^+_*$ adjunction and the reflexification $\beta\colon (-) \to (-)^{\vee \vee}$. Lemma \ref{lem_equiv_of_refl} implies that $\alpha$ is an isomorphism of functors (see Lemma \ref{lem_iso_on_perf_is_iso_on_Db}).
		
		Morphism $\alpha$ yields morphism $\alpha^{-} \colon  f_*^-\to (f^{+}_* \Sigma)$ of the term-wise extension of functors $f_*|_{\mathscr{P}_{-1}}$ and $f^{+}_* \Sigma_o|_{\mathscr{P}_{-1}}$ to the category $\Hot^{-,b}(\mathscr{P}_{-1})$. Since $\alpha$ is an isomorphism, the same is true about $\alpha^-$.
		
		Let now $E$ be an object in $\dD^b(X)$. It is isomorphic to a complex $\pP_{\bcdot}$ in $\Hot^{-,b}(\mathscr{P}_{-1})$. The first layer of spectral sequence
		$$
		E^1_{p,q} = R^pf_* \pP_q \Rightarrow R^{p+q}f_* E
		$$  	
		has only one non-zero row. Hence, complex $f_* \pP_{\bcdot}$ is isomorphic to $Rf_* E$. Analogously, since $\Sigma_o(\mathscr{P}_{-1}) \simeq \mathscr{P}_0^+$, complex $f^+_* \Sigma(\pP_\bcdot)$ is isomorphic to $Rf^+_* \Sigma\, E$. Thus, $\alpha^-$ induces an isomorphism of $Rf_* $ and $Rf^+_* \Sigma$.	
	\end{proof}
	
	\begin{PROP}\label{prop_triangle_for_flop}
		Let $f\colon X\to Y$ satisfy (a) and let $\Sigma_{\textrm{qc}}$ be as in Proposition \ref{prop_DG_enha_of_sigma}. The derived base change $\wt{\epsilon}$ induces a functorial exact triangle of  functors $\dD_{\textrm{qc}}(X) \to \dD_{\textrm{qc}}(X^+)$:
		$$
		\Sigma_{\textrm{qc}}[1] \to Lg^{+*}Rg_* \xrightarrow{\wt{\epsilon}} F \to \Sigma_{\textrm{qc}}[2].
		$$
	\end{PROP}
	
	\begin{proof}
		We have already seen in Section \ref{ssec_alt_desc_of_VdB} that there exists a choice of compact generators $\mM$ and $\nN^+$ of $\dD_{\textrm{qc}}(X)$ and $\dD_{\textrm{qc}}(X^+)$ respectively, such that $\dD_{\textrm{qc}}(X) \simeq \dD(\textrm{Mod--}A) \simeq \dD_{\textrm{qc}}(X^+)$, for an algebra $A =\Hom_X(\mM, \mM)$ (see Lemma \ref{lem_Ext(M,M)=0}). Under these equivalences, functor $\Sigma_{\textrm{qc}}$ is given by algebra $A$ considered as an $A^{\opp}\otimes A$ bimodule.
		
		Appendix \ref{sec_sph_funct_and_enh}, allows us to lift the base change $Lg^{+*}Rg_* \to Rp^+_* Lp^*$ to a 1-morphism in $\Bimod$. It induces a functorial exact triangle
		\begin{align}\label{eqtn_sigma_prime}
		\Sigma' \to Lg^{+*}Rg_* \to Rp^+_* Lp^* \to \Sigma'[1].
		\end{align}

		Let us show that functors $\Sigma'|_{\mathscr{P}_{-1}}$ and $\Sigma_o[1]|_{\mathscr{P}_{-1}}$ are isomorphic. To this end, we consider an object $\mM$ in $\mathscr{P}_{-1}$. By Proposition \ref{prop_equiv_of_P_-1_P_0}, object $\nN^+ := (f^{+*}f_*(\mM))^{\vee\vee}$ lies in $\mathscr{P}_0^+$. Thus, Proposition \ref{prop_LgRgN} yields an exact triangle
		\begin{equation}\label{eqtn_5}
		\nN^+[1] \to Lg^{+*}Rg_*^+ \nN^+ \to f^{+*}f^+_* \nN^+ \to \nN^+[2].
		\end{equation}
		By Lemma \ref{lem_equiv_of_refl}, we have $f^+_*\nN^+\simeq f_* \mM$, hence $g^+_* \nN^+ \simeq g_* \mM$. Proposition \ref{prop_iso_of_ff_and_pp_for_flop} assures that $f^{+*}f_*\mM \simeq F(\mM)$, hence triangle (\ref{eqtn_5}) reads
		\begin{equation}\label{eqtn_11}
		\Sigma_o(\mM)[1] \to Lg^{+*}Rg_* \mM \to F(\mM) \to \Sigma_o(\mM)[2].
		\end{equation}
		
		There exists a morphism from triangle (\ref{eqtn_11}) to triangle (\ref{eqtn_sigma_prime}) applied to $\mM$ which is equal to the identity morphism on $Lg^{+*}g_*\, \mM$ and $F(\mM)$. Thus, for any $\mM$ in $\mathscr{P}_{-1}$, there exists an isomorphism $\alpha_{\mM} \colon \Sigma_o(\mM)[1] \to \Sigma'(\mM)$. It is unique, as $\Hom_{X^+}(\Sigma_o(\mM)[1], F(\mM)[-1]) = \Ext^{-2}_{X^+}(\nN^+, f^{+*}f^+_* \nN^+) = 0$ (for degree reason). $\alpha_{\mM}$ when composed with $\Sigma'(\mM) \to Lg^{+*}Rg_*(\mM)$ is equal to $\beta_{\mM}$, for a morphism $\beta = \eta \Sigma_o \colon \Sigma_o[1] \to Lg^{+*} Rg_*$. Here, $\eta \colon  \Id[1] \to g^{+!}Rg^+_*[1] \simeq Lg^{+*}Rg^+_*$ is the adjunction unit (we use $Rg^+_*\Sigma_o \simeq Rg_*$ which follows from Lemma \ref{lem_Sigma_commutes_with_Rf}).
		
		For any $\mM, \mM_1 \in \mathscr{P}_{-1}$ and $\f\colon \mM \mapsto \mM_1$, both compositions
		\begin{align*}
		&\Sigma_o(\mM)[1] \xrightarrow{\Sigma_o(\f)[1]} \Sigma_o(\mM_1)[1] \xrightarrow{\alpha_{\mM_1}} \Sigma'(\mM_1),& &\Sigma_o(\mM)[1] \xrightarrow{\alpha_{\mM}} \Sigma'(\mM) \xrightarrow{\Sigma'(\f)} \Sigma'(\mM_1)&
		\end{align*}
		fit into a commuting diagram
		\[
		\xymatrix{F(\mM_1)[-1] \ar[r] & \Sigma'(\mM_1) \ar[r] & Lg^{+*}Rg_*(\mM_1)\\
			F(\mM)[-1] \ar[u]^{F(\f)[-1]} \ar[r]& \Sigma_o(\mM)[1] \ar[r] \ar[u] & Lg^{+*}Rg_*(\mM). \ar[u]_{Lg^{+*}Rg_*(\f)} }
		\]
		Since both $\Hom_{X^+}(\Sigma_o(\mM)[1], F(\mM_1))$ and $\Hom_{X^+}(\Sigma_o(\mM)[1], F(\mM_1)[-1])$ vanish (also for degree reason), we have $\Sigma'(\f) \circ \alpha_{\mM} = \alpha_{\mM_1} \circ \Sigma_o(\f)[1]$, i.e. $\alpha$ extends to a morphism of functors $\alpha\colon \Sigma_o[1]|_{\mathscr{P}_{-1}} \to \Sigma'|_{\mathscr{P}_{-1}}$. Moreover, any $\psi$ in $\Hom_{X^+}(\Sigma_o(\mM)[1], \Sigma_o(\mM_1)[1]) \simeq \Hom_X(\mM, \mM_1)$ uniquely determines $\wt{\psi} \colon \Sigma'(\mM) \to \Sigma'(\mM_1)$. Similarly, as $\Hom_{X^+}(\Sigma'(\mM), F(\mM_1)) = 0  =\Hom_{X^+}(\Sigma'(\mM), F(\mM_1)[-1])$, any $\tau \in \Hom_{X^+}(\Sigma'(\mM), \Sigma'(\mM_1))$ determines $\wt{\tau} \colon \Sigma_o(\mM)[1] \to \Sigma_o(\mM_1)[1]$. Thus, morphism $\alpha\colon \Hom_{X^+}(\Sigma_o(\mM)[1], \Sigma_o(\mM_1)[1])\to \Hom_{X^+}(\Sigma'(\mM), \Sigma'(\mM_1))$ is a bijection. It follows that $\alpha$ is an isomorphism of functors  $\alpha\colon \Sigma_o[1]|_{\mathscr{P}_{-1}} \xrightarrow{\simeq} \Sigma'|_{\mathscr{P}_{-1}}$.
		
		Lemma \ref{lem_Ext(M,M)=0} implies that $\dD_{qc}(X) \simeq \dD(\textrm{Mod--}A)$, for an algebra $A$. Furthermore, by definition, functor $\Sigma_{\textrm{qc}}$ is isomorphic to the bimodule functor $\Phi_A$, for $A$ considered as $A^{\opp}\otimes A$ bimodule. By construction (via an exact triangle), $\Sigma'$ is also a bimodule functor.
		Thus, assumptions of Lemma \ref{lem_iso_on_A_mod} are satisfied and we conclude that $\Sigma'[-1] \simeq \Sigma_{\textrm{qc}}$. Hence, functorial exact triangle (\ref{eqtn_sigma_prime}) reads
		\begin{equation*}
		\Sigma_{\textrm{qc}}[1] \to Lg^{+*}Rg_* \to F \to \Sigma_{\textrm{qc}}[2].
		\end{equation*}
	\end{proof}
	
	\begin{COR}\label{cor_flop_is_bounded}
		Let $f\colon X\to Y$ satisfy (p). The flop functor $F$ takes $\dD^b(X)$ to $\dD^b(X^+)$.
	\end{COR}
	\begin{proof}
		The statement is local in $Y$, therefore we can assume that morphism $f$ satisfies (a). Since both $Lg^{+*}Rg_*$ and $\Sigma_{\textrm{qc}}$ take $\dD^b(X)$ to $\dD^b(X^+)$ (see Theorem \ref{thm_VdB_equivalence}), the functorial exact triangle of Proposition \ref{prop_triangle_for_flop} implies that the flop functor $F$ takes $\dD^b(X)$ to $\dD^b(X^+)$.
	\end{proof}
	
	\vspace{0.3cm}
	\subsection{An alternative description of the flop-flop functor}\label{ssec_alt_descr_of_flop_flop}~\\
	
	We show that $F^+F$ on category $\textrm{Hot}^{-,b}(\mathscr{P}_{-1})$ is isomorphic to $f^*f_*$. First, we show vanishing of higher inverse images of projective objects in $\mathscr{A}_f$ and then extend Proposition \ref{prop_flop=non_der_flop} to the sheaf $f^*f_* \nN$.
	
	\begin{LEM}\label{lem_vanishing_of_LpP}
		Let $f\colon X\to Y$ satisfy (a) and $\pP$ be projective in $\mathscr{A}_f$. Then $L^jp^* \pP = 0$, for $j\geq 2$.
	\end{LEM}
	
	\begin{proof}
		The statement is local in $X$, in particular in $Y$. Therefore, we can assume that $f$ satisfy (c) and $\pP$ is a direct sum of copies of $\pP_i$ as in (\ref{eqtn_def_P_i}).
		
		Consider the counit $Lg^*g_* \mM_i \to \mM_i$ of the $Lg^* \dashv Rg_*$ adjunction. It gives an exact triangle
		\begin{equation}\label{eqtn_S_M}
		Lg^*g_* \mM_i \to \mM_i \to S_{\mM_i} \to Lg^*g_* \mM_i[1].
		\end{equation}
		Applying functor $Lp^*$ to it yields an exact triangle on $X\times_Y X^+$:
		$$
		L(p^*g^*)g_*\mM_i \to p^* \mM_i \to Lp^* S_{\mM_i} \to L(g^*p^*)g_* \mM_i[1]
		$$
		Lemma \ref{lem_loc_free_res_of_g_M} implies that $L^j(p^*g^*)g_* \mM_i = 0$, for $j\geq 2$. Thus, $L^j p^* S_{\mM_i} = 0$, for $j>2$.
		
		Proposition \ref{prop_Lgg_M}, long exact sequence of cohomology associated to triangle (\ref{eqtn_S_M}) and sequence (\ref{eqtn_P_M_in_Coh}) imply that $S_{\mM_i}$ has two non-zero cohomology sheaves:
		$$
		\mM_i[2] \to S_{\mM_i} \to \pP_i[1] \to \mM_i[3].
		$$
		By applying $Lp^*$ and looking at the cohomology sheaves of the obtained triangle, we get an isomorphism $L^jp^*\pP_i \simeq L^{j+1}p^* S_{\mM_i} \simeq 0$, for $j\geq 3$, and an exact sequence
		$$
		0 \to L^2p^* \pP_i \to p^* \mM_i \to L^2 p^* S_{\mM_i} \to L^1p^* \pP_i \to 0.
		$$
		Sheaf $L^2p^* \pP_i$ is supported on the exceptional divisor of $p$, hence it is torsion. Since $p^* \mM_i$ is locally free, $L^2 p^* \pP_i = 0$.
	\end{proof}

	\begin{PROP}\label{prop_flop=non_der_flop_on_ffN}
		Let $f$ satisfy (a). Then flop functor $Rp^+_* Lp^*$ on category $f^*f_* \mathscr{P}_0$ is isomorphic to the non-derived flop functor $p^+_*p^*$.
	\end{PROP}
	\begin{proof}
		In view of Lemma \ref{lem_vanishing_of_LpP}, $L^jp^* \mathcal{Q} \simeq 0$, for any $\mathcal{Q}$ projective in $\mathscr{A}_f$ and any $j\geq 2$. Since $\nN$ is locally free, by applying $Lp^*$ to sequence (\ref{eqtn_ses_P_N}), we get an isomorphism $Lp^* f^*f_*\nN \simeq p^*f^*f_* \nN$. As diagram (\ref{eqtn_flop_diagrams}) commutes, the latter sheaf is isomorphic to $p^{+*}f^{+*}f_* \nN$. Lemma \ref{lem_projection_formula} implies that $R^1p^+_* p^{+*}(E) \simeq 0$, for any sheaf $E$ on $X^+$. Thus, $Rp^+_* p^{+*}f^{+*}f_* \nN\simeq p^+_* p^{+*}f^{+*}f_* \nN \simeq p^+_* p^{*}f^{*}f_* \nN$.
	\end{proof}
	
	Base change $\epsilon$ given by (\ref{eqtn_def_of_omega}) and an analogous $\epsilon^+ \colon f^*f^+_* \to p_* p^{+*}$ give $\wt{\epsilon} \colon f^*f_* \to p_*p^{+*}p^+_*p^*$ which is the composite
	\begin{equation}
	\wt{\epsilon} \colon f^*f_* \rightarrow f^*f^+_*f^{+*}f_* \xrightarrow{\epsilon^+\circ \epsilon} p_*p^{+*}p^+_*p^*.
	\end{equation}
	\begin{PROP}\label{prop_wt_omega_iso_on_M_and_O}
		Let $f$ satisfy (a). Morphism $\wt{\epsilon}$ is an isomorphism on the category $\textrm{Hot}^{-,b}(\mathscr{P}_{-1})$.
	\end{PROP}
	
	\begin{proof}
		Let $\mM$ be in $\mathscr{P}_{-1}$. By definition (\ref{eqtn_def_N+}) of $\nN^+$, we have $f_* \mM \simeq f^+_* \nN^+$, hence Lemma \ref{lem_Lif_in_A_f} implies that morphism $f^*f_* \mM \to f^*f^+_*f^{+*}f_* \mM$ is an isomorphism. Lemma \ref{lem_omegaM_is_iso} guarantees that $f^*f^+_*\epsilon_{\mM} \colon f^*f^+_*f^{+*}f_* \mM \to f^*f^+_*p^+_*p^*\mM$ is also an isomorphism. Finally, the counit morphism $f^*f^+_*p^+_*p^*\mM \to p_*p^* f^*f^+_* p^+_* p^* \mM$ yields an isomorphism $f^*f^+_*p^+_*p^*\mM \to p_*p^{+*}p^+_*p^* \mM$. Indeed, we have
		$$
		p_*p^* f^*f^+_* p^+_* p^* \mM \simeq p_*p^* f^*f^+_* f^{+*}f_* \mM \simeq p_*p^* f^*f_* \mM \simeq p_*p^{+*}f^{+*}f_* \mM \simeq p_*p^{+*}p^+_*p^* \mM.
		$$
		Above, the first and the isomorphisms are given by Lemma \ref{lem_omegaM_is_iso} and the second one follows from Lemma \ref{lem_Lif_in_A_f} together with the isomorphism $f_* \mM\simeq f^+_*\nN^+$ (see Proposition \ref{prop_equiv_of_P_-1_P_0}). Finally, the third isomorphism is given by the commutativity of diagram (\ref{eqtn_flop_diagrams}).
		
		Thus, for any $\mM \in \mathscr{P}_{-1}$, we have an isomorphism $f^*f_* \mM  \to p_*p^{+*}p^+_*p^* \mM$.
		By induction on triangles, the isomorphism extends to $\Hot^b(\mathscr{P}_{-1})$. Lemma \ref{lem_iso_on_perf_is_iso_on_Db} implies that base change $\wt{\epsilon}$ is an isomorphism on $\textrm{Hot}^{-,b}(\mathscr{P}_{-1})$.
	\end{proof}
	
	\begin{PROP}\label{prop_flop_flop_as_ff}
		Let $f\colon X \to Y$ satisfy (a). The flop-flop functor $F^+F$ restricted to $\dD^b(X)$ is isomorphic to term-wise functor $f^*f_*$ on the category $\textrm{Hot}^{-,b}(\mathscr{P}_{-1})$ under the equivalence $\textrm{Hot}^{-,b}(\mathscr{P}_{-1})\simeq \dD^b(X)$.
	\end{PROP}
	\begin{proof}
		In view of Proposition \ref{prop_wt_omega_iso_on_M_and_O}, it suffices to check that functors $p_*p^{+*}p^+_*p^*$ and $Rp_*Lp^{+*}Rp^+_*Lp^*$ are isomorphic on $\textrm{Hot}^{-,b}(\mathscr{P}_{-1})$. By Proposition \ref{prop_flop=non_der_flop}, we have an isomorphism $Rp^+_* Lp^* \mM \simeq p^+_* p^* \mM$, for any $\mM \in \mathscr{P}_{-1}$. Proposition \ref{prop_iso_of_ff_and_pp_for_flop} implies that $p^+_*p^* \mM\simeq f^{+*}f_* \mM$, hence $Rp_*Lp^{+*}p^+_*p^* \mM \simeq p_* p^{+*}p^+_*p^* \mM$ by Proposition \ref{prop_flop=non_der_flop}. By induction on triangles, we have an isomorphism on $\Hot^b(\mathscr{P}_{-1})$. We conclude by Lemma \ref{lem_iso_on_perf_is_iso_on_Db}.
	\end{proof}

	\vspace{0.3cm}
	\subsection{The flop functor as the inverse of Van den Bergh's functor}\label{ssec_flop_is_inv_of_VdB}~\\
	
	We show that the flop functor $F^+\colon \dD^b(X^+)\to \dD^b(X)$ and functor $\Sigma$ are inverse to each other. In view of Theorem \ref{thm_VdB_equivalence}, this statement is equivalent to the flop functor $F^+$ being isomorphic to the term-wise functor $T = (f^*f^+_*(-))^{\vee \vee}$.
	
	\begin{PROP}\label{prop_flop_on_P_0}
		Let $f\colon X \to Y$ satisfy (a). The flop functor is isomorphic to $(f^{+*}f_*(-))^{\vee \vee}$ a a functor $\mathscr{P}_{0}\to \mathscr{P}_{-1}^+$.
	\end{PROP}
	
	\begin{proof}
		By Proposition \ref{prop_LgRgN}, for any $\nN\in \mathscr{P}_0$, triangle
		\begin{equation}\label{eqtn_F1}
		Lp^* \nN[1] \to L(p^*g^*)g_* \nN\to Lp^*(f^*f_* \nN) \to Lp^*\nN[2]
		\end{equation}
		is exact. Since $g_* \nN$ has locally free resolution of length two (see Lemma \ref{lem_loc_free_res_of_g_M}), by considering exact sequence of cohomology of triangle (\ref{eqtn_F1}), we get an exact sequence
		\begin{equation}\label{eqtn_F2}
		0 \to L^2p^*(f^*f_* \nN) \to p^* \nN\to L^1(p^*g^*)g_* \nN \to L^1p^*(f^*f_* \nN) \to 0.
		\end{equation}
		By applying $Lp^*$ to sequence (\ref{eqtn_ses_P_N}), we get isomorphisms $L^3p^*\mathcal{Q} \simeq L^2p^*(f^*f_* \nN)$ and $L^2 p^* \mathcal{Q} \simeq L^1 p^*(f^*f_* \nN)$, where $\mathcal{Q} = \hH_X^0(\iota_f^* \nN)$ is a projective object in $\mathscr{A}_f$. In view of Lemma \ref{lem_vanishing_of_LpP}, sequence (\ref{eqtn_F2}) reduces to an isomorphism $p^* \nN\simeq L^1(p^*g^*)g_* \nN$.
		
		Let $\mM^+ = (f^{+*}f_* \nN)^{\vee \vee}$ be an object in $\mathscr{P}_{-1}^+$. Proposition \ref{prop_equiv_of_P_-1_P_0} implies that $f^+_* \mM^+ \simeq f_* \nN$. Since diagram (\ref{eqtn_diagram_flop_2}) commutes, we have
		$$
		p^* \nN \simeq L^1(p^*g^*)g_* \nN \simeq L^1(p^{+*}g^{+*})g^+_* \mM^+.
		$$
		By applying $Lp^{+*}$ to sequence (\ref{eqtn_LggM}), we get an exact sequence
		\begin{equation}\label{eqtn_F3}
		0 \to L^2 p^{+*}(f^{+*}f^+_* \mM^+) \to p^{+*}\mM^+ \to L^1(p^{+*}g^{+*})g^+_* \mM^+ \to L^1p^{+*}(f^{+*}f^+_*\mM^+) \to 0.
		\end{equation}
		Sequence (\ref{eqtn_P_M_in_Coh}) together with Lemma \ref{lem_vanishing_of_LpP} imply that $L^2p^{+*}(f^{+*}f^+_* \mM^+) \simeq L^2p^{+*} \pP^+ \simeq 0$. Hence, sequence (\ref{eqtn_F3}) can be rewritten as
		\begin{equation}\label{eqtn_F4}
		0 \to  p^{+*}\mM^+ \to L^1(p^{+*}g^{+*})g^+_* \mM^+ \to L^1p^{+*}(f^{+*}f^+_*\mM^+) \to 0.
		\end{equation}
		Lemma \ref{lem_Lif_in_A_f} and Proposition \ref{prop_wt_omega_iso_on_M_and_O} give isomorphisms
		$$
		p^+_*p^{+*}f^{+*}f^+_* \mM^+ \simeq p^+_*p^{+*}p^+_*p^{*}p_*p^{+*} \mM^+ \simeq p^+_*p^{*}p_*p^{+*} \mM^+ \simeq f^{+*}f^+_* \mM^+.
		$$
		Thus, sequence (\ref{eqtn_E_to_ffE}) implies that $Rp^+_*(L^1p^{+*}f^{+*}f^+_* \mM^+) \simeq 0$. Moreover, we have $Rp^+_* p^{+*}\mM^+ \simeq \mM^+$. Hence, by applying $Rp^+_*$ to sequence (\ref{eqtn_F4}), we get
		$$
		Rp^+_* p ^* \nN\simeq Rp^+_*(L^1(p^{+*}g^{+*})g^+_* \mM^+) \simeq Rp^+_* p^{+*}\mM^+ \simeq \mM^+.
		$$
		Hence, flop functor takes $\nN$ to $\mM^+$.
		
		Let $\epsilon_{\nN} \colon f^{+*}f_* \nN \to p^+_*p^* \nN\simeq \mM^+$ be the base change. Since $f_*\nN \simeq f^+_* \mM^+$, we have $f^{+*}f_* \nN \simeq f^{+*}f^+_* \mM^+$. The kernel of reflexification $f^{+*}f^+_* \mM^+ \to (f^{+*}f^+_*\mM^+)^{\vee \vee} \simeq \mM^+$ is torsion and $p^{+*}p_* \nN \simeq \mM^+$ is torsion-free, hence, by applying functor $\Hom_{X^+}(-, p^{+*}p_* \nN)$ to sequence (\ref{eqtn_P_M_in_Coh}) with $f$ and $\mM$ replaced by $f^+$ and $\mM^+$, we get an isomorphism $\Hom_{X^+}(\mM^+, p^{+*}p_* \nN) \simeq \Hom_{X^+}(f^{+*}f^+_*\mM^+, p^{+*}p_* \nN)$. It follows that $\epsilon_{\nN}$ factors uniquely via a morphism $\wt{\epsilon}_{\nN}\colon \mM^+ \simeq (f^{+*}f_* \nN)^{\vee \vee} \to p^+_*p^*\nN$ which is easily checked to be functorial. Thus, there exists a morphism $\wt{\epsilon}\colon  (f^{+*}f_*(-))^{\vee \vee} \to p^+_*p^*(-)$ of functors.
		
		Objects of $\mathscr{P}_0$ and $\mathscr{P}_{-1}^+$ are reflexive on $X$ and $X^+$, hence any morphism $\nN_1 \to \nN_2$ (and $\mM^+_1 \to \mM_2^+$) is determined by its restrictions to any open set with the complement of codimension greater than one. In particular, by its restriction to the complement $U \subset X$ of the exceptional set of $f$ (and $f^+$). Flop functor is $Rp^+_* Lp^*$ and the other functor is $(f^{+*}f_*(-))^{\vee \vee}$. Then, taking into account that maps $f$, $f^+$, $p$, and $p^+$ are isomorphisms outside the exceptional sets and $(f^{+})^{-1}\circ f=p^+\circ p^{-1}$ on $U$,
		the morphism $\Hom_X(\nN_1, \nN_2) \to \Hom_{X^+}(\mM^+_1, \mM^+_2)$ induced by the flop functor $F$ coincides with the morphism given by $(f^{+*}f_*(-))^{\vee \vee}$. It follows that $\wt{\epsilon}$ is an isomorphism of functors.
	\end{proof}
	
	\begin{THM}\label{thm_flop_is_inv_of_vdB}
		Let $f\colon X\to Y$ satisfy (a). The flop functor $F$ takes $\Per{0}(X/Y)$ to $\Per{-1}(X^+/Y)$. It is isomorphic $\Sigma_{\textrm{qc}}^{-1}$ in the category of bimodule functors $\dD_{\textrm{qc}}(X)\to \dD_{\textrm{qc}}(X^+)$.
	\end{THM}
	\begin{proof}
		The same argument as in the proof of Proposition \ref{prop_DG_enha_of_sigma} shows that $\Sigma_{\textrm{qc}}^{-1}$ is the term-wise extension of $(f^{+*}f_*(-))^{\vee \vee}$ to the category $\Hot^{-,b}(\mathscr{P}_0)$. Proposition \ref{prop_flop_on_P_0} implies that $F|_{\mathscr{P}_0}\simeq \Sigma_{\textrm{qc}}^{-1}|_{\mathscr{P}_0}$.
		
		Categories $\dD_{\textrm{qc}}(X)$ and $\dD_{\textrm{qc}}(X^+)$ have compact generators $\nN\in \mathscr{P}_0$ and $\mM^+ = (f^{+*}f_*\nN)^{\vee \vee}\in \mathscr{P}_{-1}^+$. For algebras $A_X = \Hom_X(\nN, \nN)$ and $A_{X^+} = \Hom_{X^+}(\mM^+, \mM^+)$, we have equivalences $\dD_{\textrm{qc}}(X) \simeq \dD(\textrm{Mod--}A_X)$, $\dD_{\textrm{qc}}(X^+) \simeq \dD(\textrm{Mod--}A_{X^+})$. By construction, functor $\Sigma_{\textrm{qc}}^{-1}$ is isomorphic to a bimodule functor $\Phi_{M_1}$, for some $A_X^{\opp} \otimes A_{X^+}$ DG bimodule $M_1$, see Appendix \ref{sec_sph_funct_and_enh}. Similarly, the flop functor $F$ is isomorphic to $\Phi_{M_2}$, for some DG bimodule $M_2$. Moreover, from the construction of $\Sigma_{\textrm{qc}}^{-1}$, it follows that $\hH^i(M_1)= 0$, for $i \neq 0$ (see the discussion before Proposition \ref{prop_DG_enha_of_sigma}). By Lemma \ref{lem_iso_on_A_mod}, isomorphism $F|_{\mathscr{P}_0} \simeq \Sigma_{\textrm{qc}}^{-1}|_{\mathscr{P}_0}$ implies an isomorphism $F\simeq \Sigma_{\textrm{qc}}^{-1}$.
	\end{proof}
	As an immediate consequence of Theorem \ref{thm_VdB_equivalence} and Theorem \ref{thm_flop_is_inv_of_vdB} we get
	\begin{COR}\label{cor_flop_is_equiv}
		Let $f\colon X\to Y$ satisfy (a). Then the flop functor $F$ induces an equivalence of $\dD^b(X)$ with $\dD^b(X^+)$.
	\end{COR}
	
	In view of Proposition \ref{prop_triangle_for_flop} and Theorem \ref{thm_flop_is_inv_of_vdB} we have a functorial exact triangle of functors $\dD_{\textrm{qc}}(X) \to \dD_{\textrm{qc}}(X^+)$
	\begin{equation}\label{eqtn_flop_LgRg_flop}
	(F^+)^{-1}[1] \to Lg^{+*}Rg_* \to F \to (F^+)^{-1}[2]
	\end{equation}
	Moreover, it restricts to a functorial exact triangle of functors $\dD^b(X)\to \dD^b(X^+)$.

	\begin{COR}\label{cor_flop_t-exact_and_equiv_on_A_f}
		Let $f\colon X\to Y$ satisfy (a). The flop functor $F\colon \dD^b(X) \to \dD^b(X^+)$ is $t$-exact when $\dD^b(X)$ is endowed with the \tr e with heart $\Per{p}(X/Y)$ and $\dD^b(X^+)$ with the \tr e with heart $\Per{p-1}(X^+/Y)$. Moreover, $F$ induces an equivalence of $\mathscr{A}_f$ with $\mathscr{A}_{f^+}[1]$.
	\end{COR}
	\begin{proof}
		Since $Rf^+_* \circ F \simeq Rf_*$ and, by Theorem \ref{thm_flop_is_inv_of_vdB}, functor $F$ takes $\Per{0}(X/Y)$ to $\Per{-1}(X^+/Y)$, the flop $F$ induces an equivalence of $\mathscr{A}_f\subset \Per{0}(X/Y)$ with $\mathscr{A}_{f^+}[1]\subset \Per{-1}(X^+/Y)$.
		
		Category $\Per{p}(X/Y)$ is defined by two conditions; we require that $Rf_* E$ is a pure sheaf, for any $E \in \Per{p}(X/Y)$, and that $\mathscr{A}_f[-p] \subset \Per{p}(X/Y)$. It follows from the properties of $F$ that $F(\Per{p}(X/Y)) \subset \Per{p-1}(X^+/Y)$. Since $F$ is an equivalence, both categories are hearts of bounded \tr es on $\dD^b(X^+)$. It follows that $F(\Per{p}(X/Y)) = \Per{p-1}(X^+/Y)$.
	\end{proof}

	\begin{COR}\label{cor_D(X)=D(Per)}
		Let morphism $f\colon X\to Y$ satisfy (a). Then categories $\dD^b(X)$ and $\dD^b(\Per{p}(X/Y))$ are equivalent, for $p\in \mathbb{Z}$.
	\end{COR}
	\begin{proof}
		Corollary \ref{cor_flop_t-exact_and_equiv_on_A_f} implies that, for any $p \in \mathbb{Z}$, category $\Per{p}(X/Y)$ is equivalent either to $\Per{0}(X/Y)$ or to $\Per{-1}(X/Y)$. Hence, $\dD^b(\Per{p}(X/Y)) \simeq \dD^b(\Per{q}(X/Y))$, for $q \in \{-1,0\}$, $q \cong p \, \mod 2$. By Proposition \ref{prop_D(X)=D(Per)}, the latter category is equivalent to $\dD^b(X)$, which finishes the proof.
	\end{proof}
	
	Let $f\colon X\to Y$ satisfy (a). Proposition \ref{prop_vanish_Rpi_Zpi_X_of_E} implies that $Rp_* \oO_{X\times_Y X^+} = \oO_X$. Thus, commutativity of diagram (\ref{eqtn_flop_diagrams}) yields isomorphisms $F^+ \circ Lf^{+*} \simeq Rp_*Lp^{+*}Lf^{+*} \simeq Rp_*Lp^*Lf^* \simeq Lf^*$. Hence, also $F^+\circ Lg^{+*} \simeq Lg^*$. Composing triangle (\ref{eqtn_flop_LgRg_flop}) with $F^+$ leads to a functorial exact triangle of functors $\dD_{\textrm{qc}}(X) \to \dD_{\textrm{qc}}(X)$:
	\begin{equation}\label{eqtn_triangle_for_ff_from_f}
	\Id_{\dD_{\textrm{qc}}(X)}[1] \xrightarrow{\eta} Lg^*Rg_* \to F^+F \to \Id_{\dD_{\textrm{qc}}(X)}[2].
	\end{equation}

	\vspace{0.3cm}
	\subsection{Reduction to the affine case}\label{ssec_red_to_affine}~\\
	
	We have shown with Corollary \ref{cor_flop_is_equiv} that, if a morphism $f\colon X\to Y$ satisfies (a), then the flop functor $F$
	is an equivalence of $\dD^b(X)$ with $\dD^b(X^+)$. Now, following \cite{Chen}, we show that the flop functor is also an equivalence, provided morphism $f$ satisfies (p).
	
	Note that the condition on $Y$ to have hypersurface singularities means that completions of the local rings at all closed points of $Y$ are defined by one equation in regular complete local rings.
	\begin{LEM}\label{lem_affine_cover}
		Let $Y$ be an irreducible variety of dimension $n$ with hypersurface singularities. Then it has a finite open covering by affine subvarieties $Y_i$ such that every $Y_i$ admits a closed embedding as a principal divisor into a smooth affine $\yY_i$.
	\end{LEM}
	For the reader's convenience, we will give the proof of this statement.
	\begin{proof}
		Let $y\in Y$ be a closed point. Since Zariski topology is quasi-compact, it is enough to show existence of an affine open neighbourhood of $y$ which admits a closed embedding into a smooth affine $\yY$ of dimension $n+1$. Since the statement is local, we can assume that $Y$ is affine. Then $Y$ is realized as a closed subscheme in ${\mathbb A}^m$. We can assume that $m>n+1$, because otherwise the proof is obvious. Let ${\mathbb A}^m\subset {\mathbb P}^m$ be the compactification to a projective space and ${\bar Y}\subset {\mathbb P}^m$ the closure of $Y$ in ${\mathbb P}^m$.
		
		Since $Y$ has hypersurface singularities, the Zariski tangent space at point $y\in Y$ has dimension $\le n+1$. It follows that we can choose a point $z$ in ${\mathbb P}^m\setminus {\bar Y}$ such that the line $(zy)\subset {\mathbb P}^m$ does not contain any point of ${\bar Y}$ except for $y$ and is not tangent to $Y$ at $y$.
		
		Denote by $Y_1$ the image of ${\bar Y}$ under the projection $p:{\mathbb P}^m\to {\mathbb P}^{m-1}$ with center in $z$. By the choice of $z$, the morphism ${\bar Y}\to Y_1$ is an isomorphism for some Zariski neighbourhoods of $y$ and $p(y)$. We can keep projecting in the same manner, if necessary, by replacing ${\bar Y}\subset {\mathbb P}^m$ with $Y_1\subset {\mathbb P}^{m-1}$, until we come to the projective space of dimension $n+1$ and a closed irreducible subvariety $Y'$ in it of dimension $n$ together with a morphism $\pi :{\bar Y}\to Y'$ which is an isomorphism in neighbourhoods of $y$ and $\pi (y)$.
		
		We can find an affine open neighbourhood $\yY \subset {\mathbb P}^{n+1}$ of $\pi (y)$ such that $\pi$ gives an isomorphism of $Y'\cap \yY$ with a neighbourhood of $y\in Y$. By choosing a smaller affine $\yY$, if necessary, we can assume that $Y'\cap \yY \subset \yY$ is cut out by a single equation. Then $Y'\cap \yY$ is affine and embedded as a principal divisor in $\yY$, so we are done.
	\end{proof}
	
	We consider $f\colon X\to Y$ which satisfies (p).
	Lemma \ref{lem_affine_cover} implies existence of a finite affine covering $Y = \bigcup Y_i$ such that, for every $i$, space $Y_i$ admits an embedding as a principal Cartier divisor to a smooth affine $\yY_i$.
	
	For every $i$, we pull back the whole diagram to $Y_i$:
	\[
	\xymatrix{X_i \ar[dr]_{f_i} & & X_i^+ \ar[dl]^{f^+_i} & & X \ar[dr]_{f} & & X^+ \ar[dl]^{f^+} \\ & Y_i & \ar[rr]& & & Y & }
	\]
	Note that, for any $i$, morphism $f_i$ satisfies (a).
	\begin{PROP}\label{prop_global_to_affine}
		Let $f\colon X\to Y$ and $f_i \colon X_i \to Y_i$ be as above. Assume that flop functors $F_i \colon \dD^b(X_i) \to \dD^b(X_i^+)$ are equivalences, for every $i$. Then the flop functor $F \colon \dD^b(X) \to \dD^b(X^+)$ is an equivalence.
	\end{PROP}
	
	\begin{proof}
		Since $X$ is Gorenstein, category $\dD^b(X)$ has a spanning class $\Omega = \{\oO_{x} \}$, where $x$ runs over all closed points in $X$ \cite[Lemma 1.26]{HerLopSal}. With such a choice of a spanning class the same argument as in \cite[Proposition 3.2]{Chen} proves that $F$ is an equivalence if all $F_i$ are.
	\end{proof}

	\section{Spherical pairs}\label{sec_schobers}
	
	With the notation of (\ref{eqtn_flop_diagrams}), we consider the category $\dD^b(X\times_Y X^+)$ and its quotient by the intersection $\kK^b$  of kernels of $Rp_*$ and $Rp^+_*$. We prove that $\dD^b(\mathscr{A}_f)$ is equivalent to the full subcategory of $\dD^b(X\times_Y X^+)/\kK^b$ defined as the kernel of $Rp^+_*$. We show that $\dD^b(X\times_Y X^+)/\kK^b$ admits SODs $\langle \dD^b(\mathscr{A}_f), \dD^b(X^+)\rangle = \langle \dD^b(X), \dD^b(\mathscr{A}_f) \rangle$ and similarly for $\mathscr{A}_{f^+}$. We conclude that $(\dD^b(X), \dD^b(X^+))$ and $(\dD^b(\mathscr{A}_f), \dD^b(\mathscr{A}_{f^+}))$ are spherical pairs, which yields a geometric incarnation of the schober \cite{KapSche} related to the flop.
	
	\vspace{0.3cm}
	\subsection{Category $\dD(\mathscr{A}_f)$ as a full subcategory of $\dD_{\textrm{qc}}(X \times_Y X^+)$}~\\

	\begin{LEM}\label{lem_proj_form_for_P}
		Let $f\colon X\to Y$ satisfy (a) and let $\pP$ be a projective object in $\mathscr{A}_f$. Then $p_*p^* \pP \simeq \pP$, $Rp_* L^1p^* \pP = 0$ and $Rp_* p^* \pP \simeq \pP$.
	\end{LEM}
	\begin{proof}
		By Corollary \ref{cor_flop_t-exact_and_equiv_on_A_f}, the flop functor $F^+$ induces an equivalence of $\mathscr{A}_{f^+}$ and $\mathscr{A}_f[1]$. Hence, there exists projective $\pP^+\in \mathscr{A}_{f^+}$ such that $\pP[1] \simeq Rp_* Lp^{+*}\pP^+$. Since $L^ip^{+*} \pP^+ = 0$, for $i>1$, by Lemma \ref{lem_vanishing_of_LpP}, and $Rp_*p^{+*}\pP^+  = 0$, by Corollary \ref{cor_pull-back_of_A_f}, we have $\pP \simeq p_* L^1p^{+*}\pP^+$. Then, according to Lemma \ref{lem_Lif_in_A_f}, we have $p_*p^* \pP \simeq \pP$ and $Rp_* L^1p^* \pP = 0$. Since $Rp_* Lp^* \pP \simeq \pP$, it follows that $Rp_* p^* \pP\simeq \pP$.
	\end{proof}
	
	\begin{LEM}\label{lem_triangle_for_pP}
		Let $f\colon X \to Y$ satisfy (a) and let $\mM$ be a projective object in $\Per{-1}(X/Y)$. There exists an exact triangle
		\begin{equation}\label{eqtn_triangle_for_pP}
		p^*\pP \to Lp^{+*}(f^{+*}f^+_* \nN^+) \to Lp^* \mM \to p^*\pP[1],
		\end{equation}
		for $\pP :=\hH_X^{-1}(\iota_f^* \mM)$ projective in $\mathscr{A}_f$ and $\nN^+$ given by formula (\ref{eqtn_def_N+}).
	\end{LEM}
	\begin{proof}
		By Lemma \ref{lem_vanishing_of_LpP}, we have $L^jp^* \pP = 0$, for $j>1$ and any projective object $\pP$ in $\mathscr{A}_f$. Thus, applying $Lp^*$ to sequence (\ref{eqtn_P_M_in_Coh}) and using the fact that $\mM$ is locally free, we get an isomorphism $L^jp^* \pP \simeq L^jp^*(f^*f_* \mM)$, for $j>0$. Octahedron
		\[
		\xymatrix{p^* \pP\ar[r] & p^* f^*f_* \mM \ar[r] & Lp^* \mM \\
			Lp^* \pP \ar[r] \ar[u] & Lp^*(f^*f_* \mM) \ar[r] \ar[u] & Lp^* \mM \ar[u]_{\simeq}\\ L^1p^* \pP [1] \ar[r]^{\simeq} \ar[u] & L^1p^*(f^*f_* \mM) \ar[r] \ar[u] & 0 \ar[u]}
		\]
		implies that triangle
		\begin{equation}\label{eqtn_almost_triangle_for_pP}
		p^* \pP \to p^*f^*f_* \mM \to Lp^* \mM\to p^*\pP[1]
		\end{equation}
		is exact. We also have
		\begin{equation*}\label{eqtn_iso_pffM_LpffN}
		p^*f^*f_* \mM\simeq Lp^{+*}f^{+*}f^+_*\nN^+,
		\end{equation*}
		for $\nN^+ \in \mathscr{P}^+_0$ given by (\ref{eqtn_def_N+}). Indeed, commutativity of diagram (\ref{eqtn_flop_diagrams}) implies that $p^*f^*f_* \mM \simeq p^{+*}f^{+*}f_* \mM$. Further, Lemma \ref{lem_equiv_of_refl} yields $f_* \mM \simeq f^+_* \nN^+$. Finally, $\pP^+ = \hH^0_X(\iota_{f^+}^*\nN^+)$ is a projective object in $\mathscr{A}_{f^+}$, according to Lemma \ref{lem_ses_P_N}, hence $L^jp^{+*} \pP^+ = 0$, for $j>1$ (by Lemma \ref{lem_vanishing_of_LpP}). Thus, applying $p^{+*}$ to sequence (\ref{eqtn_ses_P_N}) on $X^+$ yields $L^jp^{+*}(f^{+*}f^+_* \nN^+) = 0$, for $j>0$.
		
		It follows that triangle (\ref{eqtn_almost_triangle_for_pP}) is isomorphic to triangle (\ref{eqtn_triangle_for_pP}).
	\end{proof}
	
	\begin{PROP}\label{prop_pP_no_higher_ext}
		Let $f\colon X\to Y$ satisfy (a) and let $\mM$ be a projective object in $\Per{-1}(X/Y)$. Then $p^* \pP\in \Perf(X\times_Y X^+)$, $\Hom_{X\times_Y X^+}(p^* \pP, p^* \pP) \simeq \Hom_X(\pP, \pP)$ and $\Ext^i_{X\times_Y X^+}(p^*\pP, p^* \pP) = 0$, for $i>0$, where $\pP = \hH_X^{-1}(\iota_f^* \mM)$ is projective in $\mathscr{A}_f$.
	\end{PROP}
	
	\begin{proof}
		Sheaf $\mM$ is locally free on $X$, hence $Lp^* \mM \in \Perf(X\times_Y X^+)$. Sequence (\ref{eqtn_res_N}) implies that $f^{+*}f^+_* \nN^+ \in \Perf(X^+)$, hence $Lp^{+*}f^{+*}f_* \nN$ is also a perfect complex on $X\times_Y X^+$. Triangle (\ref{eqtn_triangle_for_pP}) implies that $p^* \pP \in \Perf(X\times_Y X^+)$.
		
		By adjunction and Lemma \ref{lem_proj_form_for_P},  we have $\Hom_{X\times_Y X^+}(p^* \pP,p^* \pP) \simeq \Hom_X(\pP, \pP)$. Thus, it suffices to check that $\Ext^i_{X\times_Y X^+}(p^* \pP, p^* \pP) = 0$, for $i>0$. Applying $\Hom(-, p^* \pP)$ to triangle (\ref{eqtn_triangle_for_pP}) yields, for any $i\in \mathbb{Z}$, an isomorphism
		$$
		\Ext^{i+1}_{X\times_Y X^+}(Lp^* \mM, p^* \pP) \simeq \Ext^i_{X\times_Y X^+}(p^*\pP, p^* \pP),
		$$
		because $\Ext^j_{X\times_Y X^+}(Lp^{+*}(f^{+*}f^+_* \nN^+), p^* \pP) \simeq \Ext^j_{X^+}(f^{+*}f^+_* \nN^+, Rp^+_* p^*\pP)= 0$ (see Corollary \ref{cor_pull-back_of_A_f}). Sheaf $\mM$ is projective in $\Per{-1}(X/Y)$ and $\pP[1] \in \Per{-1}(X/Y)$. Thus, by Lemma \ref{lem_proj_form_for_P} again, we conclude that
		$$
		\Ext^{i+1}_{X\times_Y X^+}(Lp^* \mM, p^* \pP) \simeq \Ext^{i+1}_X(\mM, Rp_*p^*\pP) \simeq \Ext^i_X(\mM, \pP[1]) = 0,
		$$
		for $i\neq 0$, which finishes the proof.
	\end{proof}
	
	Let now $\mM$ be a projective generator of $\Per{-1}(X/Y)$ and $\pP= \hH_X^{-1}(\iota_f^* \mM)$. Denote by
	\begin{equation*}
	A_P := \Hom_{X\times_Y X^+}(p^*\pP, p^*\pP)
	\end{equation*}
	the endomorphism algebra of $p^* \pP$. Functor $\Hom_{X\times_YX^+}(p^*\pP, -)\colon \textrm{QCoh}(X\times_YX^+) \to \textrm{Mod--}A_P$ has left adjoint $(-)\otimes_{A_P}E\colon \textrm{Mod--}A_P \to \textrm{QCoh}(X\times_Y X^+)$, \cite[Theorem 3.6.3]{Pop}. We denote by $\bar{p}^* \colon \dD(\mathscr{A}_f) \to\dD_{\textrm{qc}}(X\times_Y X^+)$ its derived functor.
	
	Since $\pP$ is a projective generator of $\mathscr{A}_f$, categories $\dD(\mathscr{A}_f)$ and $\dD(\textrm{Mod--}A_P)$ are equivalent. As algebras $A_P$ and $\wt{A}_P$ are quasi-isomorphic, we also have $\dD(\textrm{Mod--}A_P) \simeq [\textit{SF--}\wt{A}_P]$, for the category $\textit{SF--}\wt{A}_P \subset \textrm{DGMod--}\wt{A}_P$ of semi-free $A_X$ DG modules, i.e. DG modules that admit a filtration with direct sums of shifts of representable DG modules as the graded factors of the filtration \cite{Dri}.
	
	\begin{PROP}\label{prop_fully_faith_D(A_f)_to_D(W)}
		Functor $\bar{p}^* \colon \dD(\mathscr{A}_f) \to \dD_{\textrm{qc}}(X\times_Y X^+)$ is  fully faithful.
	\end{PROP}
	\begin{proof}
		By \cite[Lemma 4.2]{Kel2} in order to prove that functor $\dD(\textrm{Mod--}A_P) \to \dD_{qc}(X\times_Y X^+)$ is fully faithful, it suffices to show that it induces a bijection $\Hom_{\dD(A_f)}(\pP, \pP[n]) \to \Hom_{X\times_Y X^+}(p^*\pP, p^* \pP[n])$, for any $n\in \mathbb{Z}$ and that $p^* \pP \in \dD_{\textrm{qc}}(X\times_Y X^+)$ is compact. Both statements follow from Proposition \ref{prop_pP_no_higher_ext}.
	\end{proof}
	
	Functor $\overline{p}^*$ restricts to a fully faithful functor $\dD^b(\mathscr{A}_f)\to \dD^-(X\times_Y X^+)$. We shall show that $\dD^b(\mathscr{A}_f)$ is in fact a subquotient of $\dD^b(X\times_Y X^+)$. Following an idea of M. Kapranov, we show that the flop-flop functor is a spherical cotwist arising from a spherical pair.
	
	\vspace{0.3cm}
	\subsection{Spherical pairs $(\dD^b(X)$, $\dD^b(X^+))$, $(\dD^b(\mathscr{A}_f), \dD^b(\mathscr{A}_{f^+}))$ }\label{ssec_spherical_pair}~\\

	For a morphism $f\colon X\to Y$ satisfying (a) and its flop $f^+\colon X^+ \to Y$, we consider
	\begin{equation}\label{eqtn_def_intersec_of_kern}
	\kK = \{ E \in \dD^-(X\times_Y X^+)\,|\, Rp_*(E) = 0,\, Rp^+_*(E) = 0\}
	\end{equation}
	and quotient category $\dD^-(X\times_Y X^+)/\kK$. Above, as in diagram (\ref{eqtn_flop_diagrams}), morphisms $p$ and $p^+$ denote the projections for $X\times_Y X^+$. Since both $p$ and $p^+$ have fibers of relative dimension bounded by one, category $\kK$ inherits the standard \tr e from $\dD^-(X\times_Y X^+)$ (Lemma \ref{lem_A_f_is_a_heart}). In particular, $K\in \kK$ if and only if $\hH^i(K) \in \kK$, for all $i\in \mathbb{Z}$.
	
	Let $\textrm{Mor}_{\kK}\subset \dD^-(X\times_Y X^+)$ denote the category with objects as in $\dD^-(X\times_Y X^+)$ and morphisms $f\colon E \to F$ such that the cone of $f$ lies in $\kK$ \cite{V2, Nee2}. The quotient category $\dD^-(X\times_Y X^+)/\kK$ has the same objects as $\dD^-(X\times_Y X^+)$. For any pair of objects $E, F$, morphisms $E\to F$ in $\dD^-(X\times_Y X^+)/\kK$ are equivalence classes of triples $(f, Z, g)$, where $ Z\xrightarrow{f} E$ is a morphism in $\textrm{Mor}_{\kK}$ and $Z\xrightarrow{g} F$ is a morphism in $\dD^-(X\times_Y X^+)$. We have $(f,Z,g) \sim (f', Z', g')$ if there exists a triple $(f'', Z'', g'')$ and $Z''\xrightarrow{u} Z$, $Z''\xrightarrow{v} Z'$ for which diagram
	\begin{equation}\label{eqtn_diagram_for_equivalence_classes}
	\xymatrix{& Z \ar[dl]_f \ar[dr]^{g} & \\ E & Z'' \ar[l]|{f''} \ar[r]|{g''} \ar[u]_{u} \ar[d]^{v} & F \\ & Z'\ar[ul]^{f'} \ar[ur]_{g'}}
	\end{equation}
	commutes, cf. \cite{Nee2}.
	
	We denote by $\kK^b$ the intersection $\kK \cap \dD^b(X\times_Y X^+)$. The embedding $\dD^b(X\times_Y X^+) \to \dD^-(X\times_Y X^+)$ induces a functor
	\begin{equation}\label{eqtn_def_Theta}
	\chi \colon \dD^b(X \times_Y X^+)/ \kK^b \to \dD^-(X\times_Y X^+)/\kK.
	\end{equation}
	\begin{LEM}\label{lem_theta_fullly_faith}
		Functor $\chi$ is fully faithful.
	\end{LEM}
	\begin{proof}
		It is clear from the above definitions that $\chi$ is well-defined.
		Let us check that $\chi$ is faithful. To this end, we assume that $(f,Z,g)$, $(f',Z',g')$ are morphisms $E_1 \to E_2$ in $\dD^b(X\times_Y X^+)/\kK^b$ which are equivalent as morphisms in $\dD^-(X\times_Y X^+)/\kK$. Let $(f'', Z'', g'')$ with $Z''\in \dD^-(X\times_Y X^+)$ be the triple defining the equivalence as in (\ref{eqtn_diagram_for_equivalence_classes}). Choose $l\in \mathbb{Z}$ such that $E_1$, $E_2$, $Z$ and $Z'$ belong to $\dD^b(X\times_Y X^+)^{\gge l}$. Then
		\begin{equation}\label{eqtn_cone_of_trunc}
		\hH^i(\textrm{Cone}(\tau_{\gge l} f'')) = \left\{\begin{array}{ll} \hH^i(\textrm{Cone}(f'')),& \textrm{for }\, i\geq l-1,\\
		0,&\textrm{for }\, i<l-1,\end{array} \right.
		\end{equation}
		implies that all cohomology sheaves of $\textrm{Cone}(\tau_{\gge l} f'')$ belong to $\kK$, i.e. $\textrm{Cone}(\tau_{\gge l}f'')$ is an object in $\kK^b$. The truncation $\tau_{\gge l}$ of diagram (\ref{eqtn_diagram_for_equivalence_classes}) is isomorphic to the original diagram on the boundary diamond, while the middle line $(\tau_{\gge l}f'', \tau_{\gge l} Z'', \tau_{\gge l} g'')$ defines a morphism in $\dD^b(X\times_Y X^+)/\kK^b$. Hence, $(f,Z,g) \sim (f', Z', g')$ holds in $\dD^b(X\times_Y X^+)/\kK^b$.
		
		Finally, let $(f, Z, g)$ be a morphism $E_1 \to E_2$ in $\dD^-(X\times_Y X^+)/\kK$ and assume there exists $l\in \mathbb{Z}$ such that both $\hH^i(E_1)$ and $\hH^i(E_2)$ vanish, for $i\leq l$. Then (\ref{eqtn_cone_of_trunc}) implies that $\tau_{\gge l}f \in \textrm{Mor}_{\kK}$. Hence, $(\tau_{\gge l} f, \tau_{\gge l} Z, \tau_{\gge l} g)$ is a morphism $E_1 \to E_2$ in $\dD^b(X\times_Y X^+)/\kK^b$ which functor $\chi$ takes to a morphism equivalent to $(f,Z,g)$. Thus, functor $\chi$ is also full. 	
	\end{proof}
	In view of Lemma \ref{lem_theta_fullly_faith}, we can regard $\dD^b(X\times_Y X^+)/\kK^b$ as a full subcategory of \mbox{$\dD^-(X\times_Y X^+)/\kK$}.
	\begin{PROP}\label{prop_Lp_bounded_up_to_torsion}
		The composite of $Lp^* \colon \dD^b(X) \to \dD^-(X\times_Y X ^+)$ with the quotient functor $Q\colon \dD^-(X\times_Y X^+) \to \dD^-(X\times_Y X^+)/\kK$ takes $\dD^b(X)$ to the essential image of $\chi$.
	\end{PROP}
	\begin{proof}
		Let $E\in \dD^b(X)^{\gge l}$. Since $p$ has fibers of relative dimension bounded by one and $Rp_*Lp^*(E) \simeq E$, we have $Rp_* \hH^j(Lp^* E) = 0$, for $j<l-1$.
		
		As $\Per{-1}(X/Y)$ is obtained from $\Coh(X)$ by means of the tilt in a torsion pair, sequence (\ref{eqtn_cohom_after_tilt}) implies that $\hH^j_{\Per{-1}(X/Y)}(E) =0$, for $j<l$. By Corollary \ref{cor_flop_t-exact_and_equiv_on_A_f}, object $Rp^+_* Lp^*(E)$ lies in $\dD^b(X^+)^{\gge l}_{\Per{0}(X^+/Y)}$. Category $\Coh(X^+)$ is obtained from $\Per{0}(X^+/Y)$ by the tilt in a torsion pair, hence $Rp^+_*Lp^*E \in \dD^b(X^+)_{X^+}^{\gge l-1}$. Finally, the relative dimension fibers of $p^+$ being bounded by one implies that $Rp^+_* \hH^j(Lp^* E) = 0$, for $j< l-2$. Thus, $\tau_{\lle l-3} Lp^*E$ lies in $\kK$, i.e. morphism $Lp^*E \to \tau_{\gge l-2} Lp^* E$ is an isomorphism in $\dD^-(X\times_Y X^+)/\kK$.	 
	\end{proof}
	
	We denote by
	$$
	\wt{L}p^* \colon \dD^b(X) \to \dD^b(X\times_Y X^+)/\kK^b
	$$
	the functor invoked by Lemma \ref{lem_theta_fullly_faith} and Proposition \ref{prop_Lp_bounded_up_to_torsion}.
	
	In order to show existence of a right adjoint to $\wt{L}p^*$, we first prove the following general statement.
	\begin{LEM}\label{lem_homs_in_D/K}
		Let $\dD$ be a triangulated category and $\kK$ a thick triangulated subcategory. Let further $E\in \dD$ be such that $\Hom_{\dD}(E, K) = 0$, for any $K\in \kK$. Then, for any $F \in \dD$, we have an isomorphism
		$$
		\Hom_{\dD}(E,F) \xrightarrow{\simeq} \Hom_{\dD/\kK}(E,F).
		$$
	\end{LEM}
	\begin{proof}
		We have the map
		\begin{align*}
		&\alpha \colon \Hom_{\dD}(E, F) \to  \Hom_{\dD/\kK}(E, F),& &\alpha\colon h \mapsto (\textrm{id}, E, h).&
		\end{align*}
		
		Let now $(f, Z, g)$ be any element in $\Hom_{\dD/\kK}(E, F)$. By definition, $f$ is a morphism in $\textrm{Mor}_{\kK}$. In other words we have an exact triangle
		$$
		Z \xrightarrow{f} E \rightarrow K \rightarrow Z[1],
		$$
		with $K\in \kK$. Since both $\Hom_{\dD}(E,K)$ and $\Hom_{\dD}(E,K[-1])$ vanish, object $Z$ is uniquely decomposed as $Z \simeq K[-1] \oplus E$ and $f=(0, \textrm{id})$ is the projection to $E$. Accordingly, we have a decomposition $g=(g_K,g_E)$. Diagram
		\[
		\xymatrix{& K[-1] \oplus E \ar[dl]_{(0,\textrm{id})} \ar[dr]^{(g_K,g_E)} & \\ E & E \ar[l]_{\textrm{id}} \ar[u]|{(0,\textrm{id})} \ar[r]^{g_E} & F}
		\]
		commutes, i.e. $(f,Z,g) \sim (\textrm{id}, E, g_E)$. Thus, map $\alpha$ is surjective.
		
		Let now $(\textrm{id}, E, g) \sim (\textrm{id}, E, g')$ in $\dD/\kK$. For the triple $(f'', Z'', g'')$ giving the equivalence (see diagram (\ref{eqtn_diagram_for_equivalence_classes})), we have, as above, the decomposition $Z'' = K''[-1] \oplus E$. Commutativity of
		\[
		\xymatrix{&&  E \ar[dll]_{\textrm{id}}\ar[drr]^g & &\\
			E & & K''[-1] \oplus E  \ar[ll]|{(0,\textrm{id})} \ar[rr]|{(g''_K,g''_E)} \ar[u]|{(0,\textrm{id})} \ar[d]|{(0,\textrm{id})} & & F \\
			& & E \ar[ull]^{\textrm{id}} \ar[urr]_{g'} & & }
		\]
		yields $g = g''_E = g'$, i.e. $\alpha$ is also a monomorphism.
	\end{proof}

	\begin{PROP}\label{prop_Lp_has_adjoint}
		Functor $\wt{L}p^*$ is fully faithful and has the right adjoint $Rp_* \colon \dD^b(X\times_Y X^+)/\kK^b\to \dD^b(X)$.
	\end{PROP}
	
	\begin{proof}
		First, note that functor $Rp_* \colon \dD^b(X\times_Y X^+)/\kK^b \to \dD^b(X)$ is well-defined, because $\kK^b\subset \textrm{Ker} Rp_*$.
		
		For any $E\in \dD^b(X)$ and $F \in \dD^-(X\times_Y X^+)$, we have an isomorphism $\Hom_{X\times_Y X^+}(Lp^* E, F) \simeq \Hom_X(E, Rp_* F)$. Since functor $\chi$ is fully faithful (Lemma \ref{lem_theta_fullly_faith}), in order to prove that $Rp_*$ is right adjoint to $\wt{L}p^*$, it suffices to show that, for any $F \in \dD^-(X\times_Y X^+)$, we have
		\begin{equation}\label{eqtn_iso_Lp}
		\Hom_{X\times_Y X^+}(Lp^* E, F) \simeq \Hom_{\dD^-(X\times_Y X^+)/\kK}(Lp^*E, F).
		\end{equation}
		By adjunction, $\Hom_{X\times_Y X^+}(Lp^*E, K) = 0$, for any $K\in \kK$, thus (\ref{eqtn_iso_Lp}) follows from Lemma \ref{lem_homs_in_D/K}.
		
		Since functor $Lp^* \colon \dD^b(X) \to \dD^-(X\times_Y X^+)$ is fully faithful, isomorphism (\ref{eqtn_iso_Lp}), for $F = Lp^*(E')$, implies that functor $\wt{L}p^*$ is fully faithful too.
	\end{proof}
	
	Since $Rp_* \wt{L}p^* \simeq \textrm{Id}_{\dD^b(X)}$, we have a semi-orthogonal decomposition \cite[Lemma 3.1]{B}
	\begin{equation}\label{eqtn_SOD_Ker_Lp}
	\dD^b(X\times_Y X^+)/\kK^b  = \langle \cC, \wt{L}p^* \dD^b(X)\rangle,
	\end{equation}
	with $\cC = \{ E\in \dD^b(X\times_Y X^+)/\kK^b\,|\, Rp_* E = 0\}$. We aim at showing that $\cC \simeq \dD^b(\mathscr{A}_{f^+})$.
	
	By Proposition \ref{prop_pP_no_higher_ext}, we know that, for a projective object $\mM\in \Per{-1}(X/Y)$ and $\pP= \hH^{-1}_{X} \iota_{f}^* \mM$, the endomorphisms algebras of $p^{*}\pP$ and $\pP$ are isomorphic. Next, we show that they remain the same when we pass to the quotient category $\dD^-(X\times_Y X^+)/\kK$.
	
	\begin{LEM}\label{lem_pP_nice_endo_in_quot}
		Let $\pP$ be as above. Then, for any $E\in \dD^-(X\times_Y X^+)$,
		$$
		\Hom_{X\times_Y X^+}(p^* \pP, E) \simeq \Hom_{\dD^-(X\times_Y X^+)/\kK}(p^*\pP, E).
		$$
	\end{LEM}
	\begin{proof}
		By Lemma \ref{lem_homs_in_D/K}, it suffices to check that $\Hom_{X\times_Y X^+}(p^*\pP, K) = 0$, for any $K \in \kK$. It immediately follows from applying $\Hom(-, K)$ to triangle (\ref{eqtn_triangle_for_pP}).
	\end{proof}
	Consider the composite functor
	$$
	\wt{p}^* \colon \dD^b(\mathscr{A}_f) \xrightarrow{\bar{p}^*|_{\dD^b(\mathscr{A}_f)}} \dD^-(X\times_Y X^+) \xrightarrow{Q} \dD^-(X\times_Y X^+)/\kK,
	$$
	where $Q$ denotes the quotient functor and $\overline{p}^*$ is as in Proposition \ref{prop_fully_faith_D(A_f)_to_D(W)}. Recall that we assume $Y= \Spec R$ to be affine and Noetherian. Endomorphism algebra $A_P$ of $\pP$ is finite over $R$, hence Noetherian. Thus, the category $\textrm{mod--}A_P$ of finitely generated right $A_P$ modules is abelian. Since $\pP$ is a projective generator of $\mathscr{A}_f$, $\dD^b(\mathscr{A}_f)$ is equivalent to  $\dD^b(\textrm{mod--}A_P)$. Using this equivalence and taking into account Proposition \ref{prop_pP_no_higher_ext}, we can define $\wt{p}_*$, right adjoint to $\wt{p}^*$, as
	\begin{equation}\label{eqtn_adjoint_to_p}
	\wt{p}_* \colon \dD^b(X\times_Y X^+)\to \dD^b(\textrm{mod--}A_P),\quad \wt{p}_*\colon E \mapsto R\Hom_{X\times_Y X^+}(p^* \pP, E).
	\end{equation}
	Finally, let $\mathscr{P}$ be the category of projective objects in $\mathscr{A}_f$ of the form $\hH^{-1}_X(\iota_f^* \mM)$, for $\mM\in \mathscr{P}_{-1}$. Since there is no non-trivial extension of objects in $\mathscr{P}$, it is an exact subcategory of $\mathscr{A}_f$. Moreover, for any $E \in \mathscr{A}_f$, there exists $\mM \in \mathscr{P}_{-1}$ and $\f\colon \mM \to E[1]$, surjective in $\Per{-1}(X/Y)$. As functor $\iota_{f*}\colon \mathscr{A}_f[1] \to \Per{-1}(X/Y)$ is $t$-exact, its left adjoint $\iota_f^*$ is right exact, i.e. morphism $\hH_X^{-1}(\iota_f^* \mM) \to \hH^{-1}_X(\iota_f^* E[1]) \simeq E$ is surjective in $\mathscr{A}_f$. This implies that
	$$
	\dD^b(\mathscr{A}_f) \simeq \Hot^{-,b}(\mathscr{P}).
	$$
	The last equivalence allows us to view $\wt{p}^*$ as a functor  
	\begin{equation}\label{eqtn_def_wt_p_term-wise}
	\wt{p}^{*}\colon \Hot^{-,b}(\mathscr{P}) \to \dD^-(X\times_Y X^+)/\kK.
	\end{equation}
	It is the term-wise extension of the functor $p^*\colon \mathscr{P}\to \dD^-(X\times_Y X^+)/\kK$.
	
	Lemma \ref{lem_proj_form_for_P} implies that the adjunction unit $\Id \to \wt{p}_* \wt{p}^*$ is an isomorphism on $\pP$. By Lemma \ref{lem_iso_on_perf_is_iso_on_Db} we have an isomorphism
	\begin{equation}\label{eqtn_pp=Id}
	\Id_{\dD^b(\mathscr{A}_f)} \to \wt{p}_*\wt{p}^*
	\end{equation}
	of functors $\dD^b(\mathscr{A}_f) \to \dD^b(\mathscr{A}_f)$.

	The embedding $\mathscr{A}_f \to \Per{0}(X/Y)$ is an exact functor of abelian categories. Hence, it induces the functor
	$$
	\Psi \colon \dD(\mathscr{A}_f) \to \dD_{\textrm{qc}}(X).
	$$	
	\begin{LEM}\label{lem_Rpp=Id_term-wise}
		There exists an isomorphism $Rp_* \wt{p}^* \xrightarrow{\simeq} \Psi|_{\dD^b(\mathscr{A}_f)}$ of functors $\Hot^{-,b}(\mathscr{P}) \to \dD^b(X)$.
	\end{LEM}
	\begin{proof}
		Lemma \ref{lem_proj_form_for_P} implies that the adjunction unit gives an isomorphism $\Psi|_{\mathscr{P}} \xrightarrow{\simeq} p_*p^*|_{\mathscr{P}}$ of functors $\mathscr{P} \to \Coh(X)$. Since functor $\Psi|_{\dD^b(\mathscr{A}_f)}$ takes $\Hot^{-,b}(\mathscr{P})$ to $\dD^b(X)$, Lemma \ref{lem_iso_on_perf_is_iso_on_Db} implies that the term-wise extension of $p_*p^*$ to the category $\Hot^{-,b}(\mathscr{P})$ is isomorphic to $\Psi|_{\dD^b(\mathscr{A}_f)}$.
		
		Let now $\pP_{\bcdot}$ be a complex in $\Hot^{-,b}(\mathscr{P})$ and $F = \wt{p}^* \pP_{\bcdot}$. Lemma \ref{lem_proj_form_for_P} implies that the first layer of spectral sequence
		$$
		E_1^{q,r} = R^qp_* p^*\pP_r \, \Rightarrow R^{q+r} p_* F.
		$$
		has one non-zero row only. Hence, $Rp_* F$ is quasi-isomorphic to $p_* p^* \pP_{\bcdot}$, which finishes the proof.
	\end{proof}
	
	\begin{PROP}\label{prop_wt_p_ff}
		There exists a fully faithful functor
		$$
		\wt{p}^* \colon \dD^b(\mathscr{A}_f) \to \dD^b(X\times_Y X^+)/\kK^b.
		$$
	\end{PROP}
	\begin{proof}
		Proposition \ref{prop_fully_faith_D(A_f)_to_D(W)} implies that $\wt{p}^* \colon \dD^b(\mathscr{A}_f) \to \dD^-(X\times_Y X^+)$ is fully faithful. In order to show that the composite $\dD^b(\mathscr{A}_f) \xrightarrow{\wt{p}^*} \dD^-(X\times_Y X^+) \xrightarrow{Q} \dD^-(X\times_Y X^+)/\kK$ is fully faithful, we check that, for any $E \in \dD^b(\mathscr{A}_f)$ and $F\in \dD^-(X\times_Y X^+)$, we have
		$$
		\Hom_{X\times_Y X^+}(\wt{p}^* E,F) \simeq \Hom_{\dD^-(X\times_Y X^+)/\kK}(\wt{p}^* E,F).
		$$
		By Lemma \ref{lem_homs_in_D/K}, it suffices to check that $\Hom(\wt{p}^*E, K) = 0$, for any $E \in \dD^b(\mathscr{A}_f)$ and any $K \in \kK$. By induction on triangles, Lemma \ref{lem_pP_nice_endo_in_quot} implies that, for any $i\in \mathbb{Z}$, $\Hom(\sigma_{\gge i} \wt{p}^*E, K) =0$. Since, for any $P\in \mathscr{P}$, the space $\Hom(p^*P, K)$ vanishes, we have $\Ext^j((\wt{p}^*E)^i,K) =0$, for any $i$ and $j$. It then follows from Lemma \ref{lem_limit_of_hom_is_hom} that $\Hom(\wt{p}^*E, K) = 0$.
		
		Thus, it suffices to show that functor $\wt{p}^*$ takes $\dD^b(\mathscr{A}_f)$ to the image of $\dD^b(X\times_Y X^+)/\kK^b$ in $\dD^-(X\times_Y X^+)/\kK$, see (\ref{eqtn_def_Theta}). Consider $E\in\dD^b(\aA_f)$. We shall show that $\wt{p}^* E$ is quasi-isomorphic in $\dD^-(X\times_Y X^+)/\kK$ to a complex with bounded cohomology. To this end, we shall find $N$ such that $\tau_{< N}\, \wt{p}^* E$ is an object in $\kK$.
		
		First, we note that $Rp^+_*(p^* \pP_i) = 0$, for any $i\in \mathbb{Z}$ (Corollary \ref{cor_pull-back_of_A_f}). Thus, the term-wise definition of $\wt{p}^*$ (\ref{eqtn_def_wt_p_term-wise}) implies that $Rp^+_*(\wt{p}^* E) = 0$ and, as $p^+$ has fibers of relative dimension bounded by one, $Rp^+_* \hH^i(\wt{p}^* E) =0$, for any $i\in \mathbb{Z}$ (see Lemma \ref{lem_A_f_is_a_heart}).
		
		By Lemma \ref{lem_Rpp=Id_term-wise}, we have $Rp_* \wt{p}^* E \simeq \Psi(E) \in \dD^b(X)^{\gge k}$, for some $k \in \mathbb{Z}$. Since $p$ has fibers of relative dimension bounded by one, it follows that $Rp_* \hH^i(\wt{p}^* E) = 0$, for $i<k-1$. Thus, $\tau_{< k-1} \wt{p}^* E $ lies in $\kK$.
	\end{proof}
	
	\begin{LEM}\label{lem_RHom(pP,C)}
		Let $\mM$ be a projective generator of $\Per{-1}(X/Y)$ and put $\pP = \hH^{-1}_X(\iota_f^* \mM)$. Further, let $C\in \dD^-(X\times_Y X^+)$ be such that $Rp^+_* C = 0$ and $R\Hom_{X\times_Y X^+}(p^*P, C) =0$. Then $Rp_* C = 0$.
	\end{LEM}
	\begin{proof}
		Applying $R\Hom_{X\times_Y X^+}(-,C)$ to triangle (\ref{eqtn_triangle_for_pP}) yields, $R\Hom_{X\times_Y X^+}(Lp^*\mM, C) \simeq R\Hom_X(\mM, Rp_* C) =0$. Since $\mM$ is a compact generator of $\dD_{\textrm{qc}}(X)$, cf. \cite[Lemma 3.2.2]{VdB} and proof of Proposition \ref{prop_DG_enha_of_sigma}, it follows that $Rp_* C = 0$.
	\end{proof}
	
	\begin{PROP}\label{prop_SOD_for_D(W)/K}
		Category $\dD^b(X\times_Y X^+)/\kK^b$ admits a semi-orthogonal decomposition
		$$
		\dD^b(X\times_Y X^+)/\kK^b = \langle \wt{p}^* \dD^b(\mathscr{A}_f), \wt{L}p^{+*} \dD^b(X^+)\rangle.
		$$
	\end{PROP}
	\begin{proof}
		Since $\wt{p}^*$ is defined term-wise, Corollary \ref{cor_pull-back_of_A_f} implies that $Rp^+_*\circ\wt{p}^* = 0$. In view of semi-orthogonal decomposition (\ref{eqtn_SOD_Ker_Lp}) and Proposition \ref{prop_wt_p_ff}, it suffices to show that any $E\in \dD^b(X\times_YX^+)/\kK^b$ such that $Rp^+_* E = 0$ is isomorphic to $\wt{p}^* F$, for some $F \in \dD^b(\mathscr{A}_f)$.
		
		Consider triangle
		\begin{equation}\label{eqtn_trian_Hom_(pP)}
		\wt{p}^*\wt{p}_*(E) \to E\to C_E \to  \wt{p}^*\wt{p}_*(E)[1]
		\end{equation}
		in $\dD^-(X\times_Y X^+)$. Since $Rp^+_* E = 0$ and $Rp^+_* \wt{p}^*(F)=0$, for any $F \in \dD^b(\mathscr{A}_f)$, by applying $Rp^+_*$ to triangle (\ref{eqtn_trian_Hom_(pP)}) we obtain $Rp^+_* C_E  =0$. In view of isomorphism (\ref{eqtn_pp=Id}), by applying $\wt{p}_*$ to triangle (\ref{eqtn_trian_Hom_(pP)}), we get $\wt{p}_* C_E = 0$. By formula (\ref{eqtn_adjoint_to_p}), $R\Hom(p^* \pP, C_E) = 0$. By Lemma \ref{lem_RHom(pP,C)}, $Rp_* C_E = 0$, i.e. $C_E \in \kK$. Hence, $\wt{p}^*\wt{p}_*E\to E $ is an isomorphism in $\dD^b(X\times_Y X^+)/\kK^b$.
	\end{proof}

	The following Corollary gives a geometric description for the category $\dD^b(\mathscr{A}_f)$.
	\begin{COR}\label{cor_D(A_f)_as_a_quotient}
		We have an equivalence of categories
		$$
		\dD^b(\mathscr{A}_f) = \{E \in \dD^b(X\times_Y X^+)\, |\, Rp^+_*(E) = 0\}/\kK^b.
		$$
	\end{COR}

	Proposition \ref{prop_SOD_for_D(W)/K} implies the existence of $\wt{p}_!\colon \dD^b(X \times_Y X^+)/\kK^b \to \dD^b(\mathscr{A}_f)$, left adjoint to $\wt{p}^{*}$, which is defined by the functorial exact triangle 
	\begin{equation}\label{eqtn_def_p!}
	\wt{L}p^{+*} Rp^+_*\to \Id \to \wt{p}^{*}\wt{p}_! \to \wt{L}p^{+*}Rp^+_*[1]
	\end{equation}
	associated to a semi-orthogonal decomposition. By exchanging the role of $X$ and $X^+$ we also get $\wt{p}^+_! \colon \dD^b(X\times_YX^+)/\kK^b \to \dD^b(\mathscr{A}_{f^+})$.

	\begin{LEM}\label{lem_cotwist_is_shift}
		The composite $\wt{p}_!\wt{p}^{+*} \wt{p}^+_!\wt{p}^*\colon \dD^b(\mathscr{A}_f)\to \dD^b(\mathscr{A}_f)$ is isomorphic to $\Id_{\dD^b(\mathscr{A}_f)}[4]$.
	\end{LEM}
	\begin{proof}
		First, we check that $\wt{p}^+_!\wt{p}^*$ restricts to a functor $\mathscr{P}\to \mathscr{P}^+$. By Lemma \ref{lem_Rpp=Id_term-wise}, $Rp_*\wt{p}^* \simeq \Psi$. Hence, for $P\in \mathscr{A}_f$ projective, triangle (\ref{eqtn_def_p!}), with $p$ and $p^+$ interchanged, applied to $p^*P = \wt{p}^*P$ yields an exact triangle
		$$
		\wt{L}p^*P\to p^*P \to \wt{p}^{+*}\wt{p}^+_!\wt{p}^*P\to \wt{L}p^*P[1].
		$$
		By Lemma \ref{lem_vanishing_of_LpP} object $Lp^*P$ has two non-zero cohomology sheaves, $p^*P$ and $L^1p^*P$. It follows that $\wt{p}^{+*}\wt{p}^+_!\wt{p}^*P\simeq L^1p^*P[2]$. Hence, in view of Lemma \ref{lem_Rpp=Id_term-wise}, $\Psi \wt{p}^+_!\wt{p}^*P \simeq Rp^+_*\wt{p}^{+*}\wt{p}^+_!\wt{p}^*P \simeq Rp^+_*L^1p^*P[2]$. As $Rp^+_*Lp^*$ restricts to an equivalence $\mathscr{A}_f\xrightarrow{\simeq} \mathscr{A}_{f^+}[1]$, Corollary \ref{cor_flop_t-exact_and_equiv_on_A_f}, object $Rp^+_*L^1p^*P = P^+$ is projective in $\mathscr{A}_f^+$. Hence $\wt{p}^+_!p^*P[-2] = P^+$ is projective in $\mathscr{A}_{f^+}$.
		
		As $\wt{p}^+_! \wt{p}^*[-2]$ maps objects of $\mathscr{P}$ to objects of $\mathscr{A}_{f^+}$, it maps acycylic complexes in $\mathscr{P}$ to acycylic complexes. Hence, the class $\mathscr{P}$ is adapted to $\wt{p}^+_!\wt{p}^*[-2]$, i.e. the functor can be calculated term-wise using the equivalence $\dD^b(\mathscr{A}_f) \simeq \Hot^{-,b}(\mathscr{P})$. The same is true for the functor $\wt{p}_!\wt{p}^{+*}[-2]$ in the opposite direction. Moreover, since $\wt{p}^+_!\wt{p}^*[-2]$ maps complexes of objects of $\mathscr{P}$ to complexes of objects of $\mathscr{P}^+$ also the composite $\wt{p}_!\wt{p}^{+*}\wt{p}^+_!\wt{p}^*[-4]$ can be calculated term-wise. Therefore, it suffices to check that $\wt{p}_!\wt{p}^{+*}\wt{p}^+_!\wt{p}^*[-4]|_{\mathscr{P}}\simeq \Id_{\mathscr{P}}$.
		
		Precomposing triangle (\ref{eqtn_def_p!}), with $p$ and $p^+$ interchanged, with $\wt{p}^*$, composing with $\wt{p}_!$ and using the isomorphism $\wt{p}_!\wt{p}^*\simeq \Id$, we arrive at an exact triangle:
		\begin{equation}\label{eqtn_tr}
		\wt{p}_!\wt{L}p^*Rp_*p^*P \to P \to \wt{p}_!\wt{p}^{+*} \wt{p}^+_!\wt{p}^*P \to \wt{p}_!\wt{L}p^*Rp_*p^*P[1].
		\end{equation} 
		First, we calculate the image of $\wt{p}_!\wt{L}p^*Rp_*p^*P \simeq \wt{p}_!\wt{L}p^*\Psi P \simeq \wt{p}_!\wt{L}p^*P$ under  $\Psi = Rp_*\wt{p}^*$ using triangle (\ref{eqtn_def_p!}) composed with $Rp_*$ and applied to $\wt{L}p^*P$:
		$$
		Rp_*Lp^{+*}Rp^+_*\wt{L}p^*P \to Rp_* \wt{L}p^*P \to \Psi\wt{p}_!\wt{L}p^*P \to Rp_*Lp^{+*}Rp^+_*\wt{L}p^*P [1].
		$$
		By Lemma \ref{lem_proj_form_for_P} we have $Rp_*\wt{L}p^*P \simeq P$. As $Rg_*P =0$, triangle (\ref{eqtn_triangle_for_ff_from_f}) implies that $Rp_*Lp^{+*}Rp^+_*\wt{L}p^*P \simeq P[2]$. It follows that  $\Psi\wt{p}_!\wt{L}p^*P$ has non zero cohomology in degrees $0$ and $-3$ only, both isomorphic to $P$. The map $\wt{p}_!\wt{L}p^*Rp_*p^*P \to P$ in (\ref{eqtn_tr}) is the composition of 2 adjunction counits $\wt{L}p^*Rp_*\to \Id$ and $\wt{p}_!\wt{p}^* \to \Id$. The first when applied to $p^*P$ is the truncation at the zero cohomology, the second is an isomorphism. Hence, $\wt{p}_!\wt{p}^{+*} \wt{p}^+_!\wt{p}^*P \simeq P[4]$.
	\end{proof}
	
		\begin{PROP}\label{prop_another_SOD_for_D(W)}
		Category $\dD^b(X\times_Y X^+)/\kK^b$ admits a semi-orthogonal decomposition
		\begin{equation}\label{eqtn_another_SOD_for_D(W)}
		\dD^b(X\times_Y X^+)/\kK^b \simeq \langle \wt{L}p^*\dD^b(X), \wt{p}^*\dD^b(\mathscr{A}_f)\rangle.
		\end{equation}
	\end{PROP}
	\begin{proof}
		Functor $\wt{p}_*(-) \simeq R\Hom_{\dD^b(X\times_Y X^+)/\kK^b}(p^*\pP, -)$ is right adjoint to the fully faithful functor $\wt{p}^*$ of Proposition \ref{prop_wt_p_ff}.
		
		First, we check that $\wt{p}_*\wt{L}p^*(E) = 0$, for any $E \in \dD^b(X)$. It follows from Lemma \ref{lem_pP_nice_endo_in_quot} that the above statement is equivalent to
		$$
		\Hom_{X\times_Y X^+}(p^*\pP, Lp^*E) = 0,
		$$
		for any $E \in \dD^b(X)$.
		
		By applying $\Hom_{X\times_Y X^+}(-, Lp^*E)$ to triangle (\ref{eqtn_triangle_for_pP}), we learn that $\Hom_{X\times_Y X^+}(p^* \pP, Lp^* E)$ is a cone of a morphism
		$$
		\theta_E\colon \Hom_{X\times_Y X^+}(Lp^*\mM, Lp^*E) \to \Hom_{X\times_Y X^+}(Lp^{+*}(f^{+*}f^+_*\nN^+), Lp^*E).
		$$
		Since $Rp_*\oO_{X\times_Y X^+} \simeq \oO_X$, we have $\Hom_{X\times_Y X^+}(Lp^*\mM, Lp^*E)\simeq \Hom_X(\mM, E)$. Moreover, in view of Proposition \ref{prop_iso_of_ff_and_pp_for_flop} and definition (\ref{eqtn_def_N+}) of $\nN^+$, we have $\Hom_{X\times_Y X^+}(Lp^{+*}(f^{+*}f^+_*\nN^+), Lp^*E) \simeq \Hom_{X^+}(F(\mM), F(E))$. Therefore, $\theta_E$ is interpreted as the morphism induced by functor $F$
		$$
		\nu_E\colon \Hom_X(\mM, E) \to \Hom_{X^+}(F(\mM), F(E)).
		$$
		Since the flop functor is an equivalence of categories (see Corollary \ref{cor_flop_is_equiv}), $\nu_E$ is an isomorphism, i.e. $\Hom_{X \times_Y X^+}(p^*\pP, Lp^*E) \simeq \textrm{Cone}(\theta_E) \simeq 0$.
		
		In order to prove the semi-orthogonal decomposition (\ref{eqtn_another_SOD_for_D(W)}) it suffices to show that, for any $E\in \dD^b(X\times_Y X^+)/\kK^b$ such that $\wt{p}_*(E)\simeq 0$, we have $E\simeq \wt{L}p^*(E')$, for some $E'\in \dD^b(X)$.
		
		Let now $E\in \dD^b(X\times_Y X^+)/\kK^b$ satisfy $\wt{p}_*E \simeq 0$. Consider an exact triangle
		\begin{equation}\label{eqtn_proof_of_SOD}
		\wt{L}p^*Rp_* E\to E \to \wt{C}_E \to \wt{L}p^*Rp_* E [1].
		\end{equation}
		We have $Rp_* \wt{C}_E =0$. Moreover, since $\wt{p}_* \wt{L}p^*Rp_*E = 0$, it follows that $\wt{p}_*C_E$. Now, we proceed as in the proof of Lemma \ref{lem_RHom(pP,C)} to show that $Rp^+_*\wt{C}_E \simeq 0$.
		
		Applying $R\Hom_{X\times_Y X^+}(-,\wt{C}_E)$ to triangle (\ref{eqtn_triangle_for_pP}) yields
		$$
		R\Hom_{X\times_Y X^+}(Lp^{+*}(f^{+*}f^+_*\nN^+), \wt{C}_E) \simeq R\Hom_{X^+}(f^{+*}f^+_*\nN^+, Rp^+_* \wt{C}_E) =0.
		$$
		As we have already mentioned, the sheaf $f^{+*}f^+_*\nN^+$ is isomorphic to the flop $F(\mM)$. Since the flop functor $F$ is an equivalence (see Corollary \ref{cor_flop_is_equiv}), $F(\mM)$ is a compact generator of $\dD^b(X^+)$. Thus, vanishing of $R\Hom_{X^+}(F(\mM), Rp^+_*\wt{C}_E)$ implies that $Rp^+_*\wt{C}_E \simeq 0$.
		
		Hence, in triangle (\ref{eqtn_proof_of_SOD}) sheaf $\wt{C}_E$ is an object of $\kK^b$. Therefore, $E \simeq \wt{L}p^*Rp_* E$ in $\dD^b(X\times_Y X^+)/\kK^b$.
	\end{proof}
	
	\begin{THM}\label{thm_schober}
		Category $\dD^b(X\times_YX^+)/\kK^b$ admits 4-periodical SODs
		\begin{align*}
		\dD^b(X\times_YX^+)/\kK^b &= \langle \wt{p}^*\dD^b(\mathscr{A}_f), \wt{L}p^{+*}\dD^b(X^+)\rangle = \langle \wt{L}p^{+*}\dD^b(X^+), \wt{p}^{+*}\dD^b(\mathscr{A}_{f^+})\rangle = \\
		&= \langle \wt{p}^{+*}\dD^b(\mathscr{A}_{f^+}), \wt{L}p^{*}\dD^b(X)\rangle = \langle \wt{L}p^{*}\dD^b(X), \wt{p}^{*}\dD^b(\mathscr{A}_{f})\rangle.
		\end{align*}
		In particular, $(\wt{p}^*\dD^b(\mathscr{A}_f), \wt{p}^{+*}\dD^b(\mathscr{A}_{f^+}))$ and $(\wt{L}p^*\dD^b(X), \wt{L}p^{*}\dD^b(X^+))$ are spherical pairs. 
	\end{THM}
	\begin{proof}
		Propositions \ref{prop_SOD_for_D(W)/K} and \ref{prop_another_SOD_for_D(W)} for $X$ and $X^+$ imply the 4-periodical SODs. Proposition \ref{prop_quad_of_recol} claims that $(\wt{p}^*\dD^b(\mathscr{A}_f), \wt{p}^{+*}\dD^b(\mathscr{A}_{f^+}))$ and $(\wt{L}p^*\dD^b(X), \wt{L}p^{*}\dD^b(X^+))$ are spherical pairs.
	\end{proof}
	
	\begin{COR}\label{cor_co_twist_for_Psi}
		Functor $\Psi \colon \dD^b(\mathscr{A}_f) \to \dD^b(X)$ is spherical. 
		The unit and counit for the $\Psi^* \dashv \Psi$ adjunction fit into functorial exact triangles:
		\begin{align}\label{eqtn_func_tr_Psi} 
		&\Psi^*\Psi \to \Id_{\dD^b(\mathscr{A}_f)} \to \Id_{\dD^b(\mathscr{A}_f)}[4] \to \Psi^*\Psi[1] ,& &F^+F \to \Id_{\dD^b(X)} \to \Psi \Psi^* \to {F^+F}[1].& 
		\end{align} 
	\end{COR}
	\begin{proof}
		By Proposition \ref{prop_quad_of_recol} and Theorem \ref{thm_schober} $(\wt{p}^*\dD^b(\mathscr{A}_f), \wt{p}^{+*}\dD^b(\mathscr{A}_{f^+}))$ is a spherical pair. 
		By Proposition \ref{prop_Kap_Sche} the corresponding spherical functor is $Rp_*\wt{p}^*$. 
		By Lemma \ref{lem_Rpp=Id_term-wise} the latter is isomorphic to $\Psi$. 
		
		By Proposition \ref{prop_Kap_Sche} the cone of the counit $\Psi^*\Psi \to \Id_{\dD^b(\mathscr{A}_f)}$ is isomorphic to $\wt{p}_!\wt{p}^{+*}\wt{p}^+_!\wt{p}^*$ and the cone of the unit $\Id_{\dD^b(X)} \to \Psi \Psi^*$ is $ Rp_*\wt{L}p^{+*}Rp^+_*\wt{L}p^*\simeq Rp_*Lp^{+*}Rp^+_*Lp^*=F^+F$. We conclude by the isomorphism $\wt{p}_!\wt{p}^{+*}\wt{p}^+_!\wt{p}^* 
		\simeq \Id_{\dD^b(\mathscr{A}_f)}[4]$, Lemma \ref{lem_cotwist_is_shift}. 
	\end{proof}

	\section{Contraction algebra as the endomorphism algebra of a projective generator}\label{sec_contr_alg}

	In \cite{DonWem} W. Donovan and M. Wemyss considered a morphism $f\colon X\to Y$ between Gorenstein varieties of dimension three that contracts a rational irreducible curve $C$ to a point. If $Y$ is a spectrum of a complete local Noetherian ring, the category $\Per{-1}(X/Y)$ has projective generator. In \cite{DonWem} a \emph{contraction algebra} $A_\textrm{con}$ is defined as the quotient of the endomorphism algebra of the projective generator for $\Per{-1}(X/Y)$ by morphisms which factor through $\oO_X$. Algebra $A_\textrm{con}$ is shown to govern non-commutative deformations of $\oO_C(-1)$. If $N_{X/C} \simeq \oO_C\oplus \oO_C(-2)$, algebra $A_\textrm{con}$ is commutative and was first considered in terms of deformation theory by Y. Toda in \cite{Toda3}.
	
	We prove in our, more general, set-up that $A_\textrm{con}$ is isomorphic to the endomorphism algebra of a projective generator for $\mathscr{A}_f$.

	Category $\Per{0}(X/Y)$ has a projective generator $\nN$ and $Q = \hH^0_X(\iota_f^*\nN)$ is a projective generator for $\mathscr{A}_f$, Lemma \ref{lem_ses_P_N}. Endomorphisms algebras $A_Q = \End_X(\qQ)$ and $A_X = \End_X(\nN)$ are Noetherian, and we have equivalences $\dD(\textrm{Mod--}A_Q)\simeq \dD(\mathscr{A}_f)$, $\dD(\textrm{Mod--}A_X)\simeq \dD_{\textrm{qc}}(X)$.
	
	Projective generator $\qQ$ for $\mathscr{A}_f$ is an object in $\Per{0}(X/Y)$. It follows that $\qQ' :=R\Hom_X(\nN,\qQ)$ is an $A_Q^{\opp}\otimes_k A_X$ module and, under the above equivalences, functor $\Psi$ is isomorphic to 
	\begin{equation*}
	\Psi\simeq \Phi_{Q'} \colon \dD(\textrm{Mod--}A_Q)\to \dD(\textrm{Mod--}A_X).
	\end{equation*}
	Proposition \ref{prop_proj_res_of_P_i} yields a finite projective resolution of $\qQ$ in $\Per{0}(X/Y)$. Hence, $\qQ'$ is a perfect  $A_X$ DG module and by Proposition \ref{prop_AnnLog} the right adjoint $\Psi^!$ exists and is isomorphic to a bimodule functor.

	\begin{PROP}\label{prop_Psi!_M_i}
		Let $f$ satisfy (a), $\mM_i$ be as in (\ref{eqtn_ses_def_M_i}), and $\pP_i$ as in (\ref{eqtn_def_P_i}). Then $\Psi\Psi^!(\mM_i)\simeq \Psi \pP_i[-2]$.
	\end{PROP}
	\begin{proof}
		Considering the right adjoint to the functorial exact triangles of Corollary \ref{cor_co_twist_for_Psi} yields an exact triangle:
		\begin{equation}\label{eqtn_tr_M_i_K_i} 
		\Psi \Psi^!(\mM_i) \to \mM_i \to (F^+F)^{-1}(\mM_i)\to 	\Psi \Psi^!(\mM_i)[1].
		\end{equation} 
		Put $\kK_i=(F^+F)^{-1}(\mM_i)$. As $Rf_* F^+F = Rf_*$, we have $Rg_*\kK_i \simeq Rg_*\mM_i$. Hence, by Proposition \ref{prop_Lgg_M}, $\hH^0_X(Lg^*Rg_*\kK_i)=f^*f_*\mM_i$, $\hH^{-1}_X(Lg^*Rg_*\kK_i) = \mM_i$ and the remaining cohomology sheaves vanish. Considering long exact sequence of cohomology sheaves of the exact triangle obtained by applying functorial exact triangle (\ref{eqtn_triangle_for_ff_from_f}) to $\kK_i$:
		$$
		\kK_i[1] \to Lg^*Rg_*\kK_i \to \mM_i \to \kK_i[2],
		$$
		we get that $\hH^0(\kK_i) \simeq \mM_i$ and $\hH^1(\kK_i) \simeq \pP_i$ (as $\pP_i$ is the kernel of the surjective morphism $f^*f_*\mM_i \to \mM_i$). The long exact sequence of cohomology sheaves of (\ref{eqtn_tr_M_i_K_i}) yields an isomorphism $\Psi \Psi^!(\mM_i) \simeq \pP_i[-2]\simeq \Psi \pP_i[-2]$.
	\end{proof}
	
	\begin{THM}\label{thm_A_P_deform_alge}
		Let $f$ satisfy (c), $\mM_i$ be as in (\ref{eqtn_ses_def_M_i}), and $\pP_i$ as in (\ref{eqtn_def_P_i}). The endomorphism algebra
		$$
		A_P= \Hom_{\mathscr{A}_f}(\bigoplus_{i=1}^n \pP_i, \bigoplus_{i=1}^n \pP_i)
		$$
		is isomorphic to the quotient of $\Hom_X(\bigoplus_{i=1}^n\mM_i, \bigoplus_{i=1}^n \mM_i)$ by the subspace of morphisms that factor via direct sums of copies of the structure sheaf.
	\end{THM}
	\begin{proof}
		By Proposition \ref{prop_Psi!_M_i} $\Psi \Psi^!(\mM_i) \simeq \Psi \pP_i[-2]$. Since $\Psi|_{\mathscr{A}_f}$ is fully faithful, we have $\Psi^!(\mM_i) \simeq \pP_i[-2]$. Thus, functor $\Psi^!$ gives a homomorphism of algebras
		$$
		\f\colon \Hom_X(\mM, \mM) \to A_P = \Hom_{\dD^b(\mathscr{A}_f)}(\Psi^! \mM, \Psi^! \mM) \simeq \Hom_X(\Psi \Psi^! \mM, \mM),
		$$
		where $\mM = \bigoplus_{i} \mM_i$. Triangle of functors right adjoint to the second triangle of (\ref{eqtn_func_tr_Psi}) implies that the morphism $\f$ fits into an exact sequence
		\begin{align*}
		&\Hom_X((F^+F)^{-1}\mM, \mM) \xrightarrow{\psi} \Hom_X(\mM, \mM)\xrightarrow{\f} \Hom_X(\Psi \Psi^! \mM, \mM) \to\\
		&\to \Hom_X((F^+F)^{-1}\mM[-1], \mM).
		\end{align*}
		By Corollary \ref{cor_flop_is_equiv}:
		$$
		\Hom_X(F^{-1} {F^+}^{-1}(\mM)[-1], \mM) \simeq \Ext^1_{X^+}({F^+}^{-1}(\mM), F(\mM)).
		$$
		Theorem \ref{thm_flop_is_inv_of_vdB} implies that ${F^+}^{-1}(\mM) \simeq \nN^+$ is a projective object in $\Per{0}(X^+/Y)$. On the other hand, $F(\mM) \simeq f^{+*}f_* \mM$ by Proposition \ref{prop_flop_flop_as_ff}. By Lemma \ref{lem_equiv_of_refl}, $f^{+*}f_*\mM \simeq f^{+*}f^+_*\nN^+$. Then, Lemma \ref{lem_Lif_in_A_f} implies that $R^1f^+_*(f^{+*}f_* \mM) = 0$, i.e. $f^{+*}f_* \mM$  is an object in $\Per{0}(X^+/Y)$ (see Remark \ref{rem_def_of_perv} and formula (\ref{eqtn_def_T_0}) therein). Thus,
		$$
		\Hom_X((F^+F)^{-1} \mM[-1], \mM) \simeq \Ext^1_{X^+}(\nN^+, f^{+*}f_* \mM) \simeq \Ext^1_{\Per{0}(X/Y^+)}(\nN^+, f^{+*}f_* \mM)= 0,
		$$
		i.e. morphism $\f$ is surjective.
		
		Let us now describe the image of $\psi$. Triangle (\ref{eqtn_flop_LgRg_flop}) with roles of $X$ and $X^+$ interchanged, when applied to $\nN^+$, gives an exact triangle
		$$
		\mM[-2] \to (F^+F)^{-1}\mM \to Lg^*Rg_* \mM[-1] \to \mM[-1],
		$$
		in view of $(F^+F)^{-1} \mM \simeq F^{-1}(\nN^+)$, $Rg^+_* \nN^+ \simeq Rg_* \mM$, and $F^+(\nN^+)\simeq \mM$. Since $\Ext^1_X(\mM, \mM)= 0 = \Ext^2_X(\mM, \mM)$, we have isomorphisms $\Hom_X(\mM, (F^+F)^{-1}\mM) \simeq \Hom_X(\mM, Lg^*Rg_* \mM[-1])$ and $\Hom_X((F^+F)^{-1}\mM, \mM) \simeq \Hom_X(Lg^*Rg_* \mM[-1], \mM)$. Thus, a morphism $\mM\to \mM$ is decomposed as $\mM \to (F^+F)^{-1}\mM \to \mM$ if an only if it is decomposed as $\mM \to Lg^*Rg_* \mM[-1] \to \mM$. By Lemma \ref{lem_loc_free_res_of_g_M}, object $Lg^*Rg_* \mM$ is quasi-isomorphic to a complex $\{\oO_X^S \to \oO_X^S\}$, for some $S\in \mathbb{Z}_+$. Thus, ``stupid'' truncation gives an exact triangle
		$$
		\oO_X^S[-2] \to \oO_X^S \to Lg^*Rg_* \mM[-1] \to \oO_X^S[-1].
		$$
		Since $\Ext^1_X(\oO_X^S, \mM) = 0 = \Ext^2_X(\oO_X^S, \mM)$, we have: $\Hom_X(\mM, Lg^*Rg_*\mM[-1]) \simeq \Hom_X(\mM, \oO_X^S)$ and $\Hom_X(Lg^*Rg_* \mM[-1], \mM) \simeq \Hom_X(\oO_X^S, \mM)$. Hence, a morphism $\mM \to \mM$ factors via $Lg^*Rg_* \mM[-1]$ if and only if it factors via $\oO_X^S$.
		
		Therefore, the image of $\psi$ is contained in the space of morphisms $\mM \to \mM$ that factor via $\oO_X^S$. Conversely, for a projective generator $\pP$ of $\mathscr{A}_f$,
		\begin{align*}
		&R\Hom_{\dD^b(\mathscr{A}_f)}(\Psi^!\mM, \Psi^!\oO_X) \simeq R\Hom_{\dD^b(\mathscr{A}_f)}(\pP, \Psi^!\oO_X[2]) \simeq R\Hom_X(\Psi\pP, \oO_X[2]) \\
		& \simeq R\Hom_X(\Psi\pP, f^! \oO_Y[2]) \simeq R\Hom_Y(Rf_* \Psi \pP, \oO_Y[2]) = 0,
		\end{align*}
		which implies that the space of morphisms $\mM \to \mM$ that factor via direct sums of copies of $\oO_X$ is annihilated by $\Psi^!$, i.e. it is contained in the kernel of $\f$.
	\end{proof}

	\appendix
	\section{Grothendieck duality and the twisted inverse image functor}\label{sec_Grot-Ver-dual}
	
	Let $f\colon X \to Y$ be a proper morphism of Noetherian schemes. Direct image functor $Rf_*$ maps $\dD^{-}(X)$ to $\dD^{-}(Y)$. Considered as a functor $\dD_{\textrm{qc}}(X)\to \dD_{\textrm{qc}}(Y)$, $Rf_*$ admits the right adjoint denoted by $f^!$. Functor $f^!$ maps $D^{+}_\textrm{qc}(Y)$ to $D^{+}_{\textrm{qc}}(X)$ and, for any $E^\bcdot\in \dD_{\textrm{qc}}(X)$ and $F^\bcdot\in \dD^{+}_{\textrm{qc}}(Y)$, the \emph{sheafified Grothendieck duality}
	\begin{equation}\label{eqtn_G-Vdual}
	Rf_* R\hHom_X(E^\bcdot, f^!(F^\bcdot)) = R\hHom_Y(Rf_*(E^\bcdot), F^\bcdot),
	\end{equation}
	is satisfied, see \cite[Cor 4.2.2]{LM}.
	
	Let $f$ be a proper morphism. We construct a morphism $\alpha \colon Lf^*(-) \otimes f^!(\oO_Y) \rightarrow f^!(-)$. In view of adjunction
	$$
	\Hom_X(Lf^*(F^\bcdot) \otimes f^!(\oO_Y), f^!(F^\bcdot)) \simeq  \Hom_Y(Rf_*(Lf^*(F^\bcdot) \otimes f^!(\oO_Y)), F^\bcdot),
	$$
	finding $\alpha_{F^\bcdot}$ is equivalent to finding a morphism $Rf_*(Lf^*(F^\bcdot)\otimes f^!(\oO_Y)) \to F^\bcdot$. By projection formula, we have
	$$
	Rf_*(Lf^*(F^\bcdot) \otimes f^!(\oO_Y)) = F^\bcdot \otimes Rf_*f^!(\oO_Y).
	$$
	Let $\varepsilon \colon Rf_* f^!(\oO_Y)\to \oO_Y$ be the counit of adjunction. Then morphism
	$$
	\textrm{id}_{F^\bcdot}\otimes \varepsilon \colon F^\bcdot \otimes Rf_*f^!(\oO_Y) \to F^\bcdot \otimes \oO_Y \simeq F^\bcdot
	$$
	gives a morphism
	\begin{equation}\label{eqtn_def_alpha}
	\alpha_{F^\bcdot} \colon Lf^*(F^\bcdot)\otimes f^!(\oO_Y) \to f^!(F^\bcdot).
	\end{equation}
	\begin{LEM}\label{lem_two_def_of_f_for_perf_compl}
		Let $f\colon X \to Y$ be a proper morphism. Then $\alpha_{F^\bcdot}$ is an isomorphism for $F^\bcdot \in \Perf(Y)$.
	\end{LEM}
	\begin{proof}
		Morphism $\alpha_{\oO_Y}$ is clearly an isomorphism. By P. Deligne's appendix to \cite{Har_res}, functor $f^!$ is defined locally on $Y$, hence $\alpha$ is also an isomorphism for any complex $F^\bcdot$ which is locally quasi-isomorphic to a complex of locally free sheaves, i.e. for $F^\bcdot\in \Perf(Y)$.
	\end{proof}
	By \cite[\textsection 4.7]{LM} morphism $Lf^*(F^\bcdot) \otimes f^!(\oO_Y) \to f^!(F^\bcdot)$ is an isomorphism for all $F^\bcdot$ in $\dD_{\textrm{qc}}(Y)$ if and only if $f$ is proper and of finite Tor dimension.

\section{Functorial exact triangles, spherical functors and pairs}\label{sec_func_ex_tr}

Let $\cC$ and $\dD$ be triangulated categories. The category of exact functors $\cC\to \dD$ is not triangulated, thus we cannot speak about exact triangles of functors. Instead, we work with a weaker notion of a \emph{functorial exact triangle} by which  we mean a triple of exact functors $\cC \to \dD$ and natural transformations between them
\begin{equation}\label{eqtn_funct_ex_tr}
F_1(-) \to F_2(-) \to F_3(-)\to F_1(-)[1]
\end{equation} 
such that when applied to any object $C\in \cC$ they give an exact triangle in $\dD$. Note that (\ref{eqtn_funct_ex_tr}) is not an exact triangle in any triangulated category.

A basic example of a functorial exact triangle is given by an SOD
$$
\dD = \langle \aA, \bB\rangle.
$$ 
Denote by $i_{\aA}\colon \aA\to \dD$ the embedding functor and by $i_{\aA}^*$ its left adjoint. Similarly, let $i_{\bB}\colon \bB \to \dD$ be the embedding and $i_{\bB}^!$ its right adjoint. Then the $i_{\aA}^*\dashv i_{\aA}$ adjunction unit $\eta$ and the $i_{\bB}^!\dashv i_{\bB}$ adjunction counit $\ve$ fit into a functorial exact triangle \cite{B}:
\begin{equation}\label{eqtn_func_tr_SOD} 
i_{\bB}i_{\bB}^!\xrightarrow{\ve} \Id_{\dD} \xrightarrow{\eta} i_{\aA}i_{\aA}^* \rightarrow i_{\bB}i_{\bB}^![1].
\end{equation}

Consider a general adjoint pair $F^*\dashv F$ of exact functors with  $F\colon \cC\to \dD$. If the adjunction counit $\ve$ fits into a functorial exact triangle
$$
F^*F \xrightarrow{\ve} \Id_{\cC} \to T\to F^*F[1]
$$
we shall refer to $T$ as a \emph{twist functor}. If the adjunction unit $\eta$ fits into a functorial exact triangle
$$
C \to \Id_{\dD}\xrightarrow{\eta} FF^* \to H[1]
$$ 
we shall refer to $C$ as a \emph{cotwist functor}. 

Constructions and uniqueness of  twist and cotwist functors are discussed in details in Subsection \ref{ssec_sph_funct} in the framework of 2-categories. 
Let us briefly say that twist and cotwist are well-defined if $F$ is a functor $\dD(A)\to \dD(B)$, for some DG algebras $A$ and $B$, given by the tensor product with an $A^{\opp}\otimes B$ DG bimodule and both of his adjoints are of this form too. If $F$ satisfies this condition, we say that it is an \emph{$A$ and $B$ perfect bimodule functor}. Indeed, in this case the bimodule defining $F$ is perfect as an $A$ and $B$ module, \textit{cf.} Proposition \ref{prop_AnnLog}. Once the bimodule corresponding to $F$ is fixed, the twist $T$ and the cotwist $C$ are defined uniquely up to isomorphisms.

Consider an  \emph{$A$ and $B$ perfect bimodule functor} $F\colon \dD(A)\to \dD(B)$. We say that $F$ is \emph{spherical} \cite{AnnLog} if the twist $T$ and the cotwist $C$ are both equivalences. If this is the case we shall refer to $T$ as a \emph{spherical twist} and to $C$ as a \emph{spherical cotwist}.

Spherical functors can be constructed via spherical pairs.
We recall the definition following \cite{KapSche}.
Let $\eE$ be a triangulated category and $\eE_+, \eE_- \subset \eE$ a pair of admissible subcategories \cite{B}. Let $i_{\pm}\colon \eE_{\pm} \to \eE$ be the inclusions. They admit left $i_{\pm}^*$ and right $i_{\pm}^!$ adjoint functors $\eE\to \eE_{\pm}$. By $j_{\pm} \colon {}^\perp \eE_{\pm} \to \eE$ we denote the inclusions of the left orthogonal complements and by $j_{\pm}^! \colon \eE\to {}^\perp \eE_{\pm}$ their right adjoints. We assume that all these functors are induced by appropriate DG functors.

\begin{DEF}\cite[Definition 3.4]{KapSche}\label{def_spherical_pair}
	The pair $(\eE_{+}, \eE_-)$ of admissible subcategories is a \emph{spherical pair} if
	\begin{enumerate}
		\item The composites $i_+^* i_- \colon \eE_- \to \eE_+$, $i_-^* i_+ \colon \eE_+ \to \eE_-$ are equivalences,
		\item The composites $j_+^!j_- \colon {}^\perp \eE_- \to {}^\perp \eE_+$, $j_-^!j_+ \colon {}^\perp \eE_+ \to {}^\perp \eE_-$ are equivalences.
	\end{enumerate}
\end{DEF}

\begin{PROP}\cite[Propositions 3.6, 3.7]{KapSche}\label{prop_Kap_Sche}
	Let $(\eE_{+}, \eE_-) \subset \eE$ be a spherical pair. 
	Then functor $\Psi := j_-^! i_+ \colon \eE_+ \to {}^{\perp}\eE_-$ is spherical. 
	The unit and the counit for the $\Psi^*\dashv \Psi$ adjunction fit into functorial exact triangles:
	\begin{align}\label{eqtn_sphr_tw_and_cotw}
	&\Psi^* \Psi \xrightarrow{\ve} \Id_{\eE_+} \to i_+^*i_-i_-^*i_+\to \Psi^* \Psi [1],& &j_-^!j_+j_+^!j_- \to \Id_{{}^\perp \eE_-} \xrightarrow{\eta} \Psi \Psi^*\to j_-^!j_+j_+^!j_-[1].&
	\end{align}
\end{PROP}
The functorial exact triangles (\ref{eqtn_sphr_tw_and_cotw}) are deduced from functorial exact triangles (\ref{eqtn_func_tr_SOD}) for the SOD's related to (admissible subcategories) $\eE_+$ and $\eE_-$. Therefore, in the presence of spherical pairs, one can define spherical functors purely in the realm of triangulated categories.

With the following proposition we describe spherical pairs associated to 4-periodical SODs
\begin{equation}\label{eqtn_quad_of_rec}
\tT = \langle \aA, \bB \rangle = \langle \bB, \cC \rangle = \langle \cC, \dD \rangle = \langle \dD, \aA \rangle.
\end{equation}
\begin{PROP}\label{prop_quad_of_recol}
	Let $\tT$ be a triangulated category with 4-periodical SODs (\ref{eqtn_quad_of_rec}). Then pairs $(\aA, \cC)$, $(\bB, \dD)$ of subcategories of $\tT$ are spherical. 
\end{PROP}
\begin{proof}
	We prove that $(\aA, \cC)$ is a spherical pair. Let $i_\aA \colon \aA \to \tT$ denote the inclusion functor and $i_\aA^*$ its left adjoint, similarly for $\cC \subset \tT$. The composite $i_\cC^* i_{\aA}\colon \aA \to \cC$ is the right mutation over $\bB$, hence it is an equivalence \cite[Lemma 1.9]{BK1}. Similarly, the right mutation $i_\aA^* i_\cC$ over $\dD$ is an equivalence, i., the first condition of Definition \ref{def_spherical_pair} is satisfied.
	
	Let now $i_\bB \colon \bB\to \tT$ denote the inclusion functor and $i_\bB^! \colon \tT\to \bB$ its right adjoint, similarly for $\dD \subset \tT$. The composites $i_{\bB}^! i_\dD$, respectively $i_\dD^! i_\bB$, are the left mutations over $\cC$, respectively $\aA$, hence they are equivalences. It follows that $(\aA, \cC)$ is a spherical pair in $\tT$. 
	The proof for the pair $(\bB, \dD)$ is analogous.   
\end{proof}

\begin{REM}
	In \cite[Theorem 3.11]{HLSchi} D. Halpern-Leistner and I. Shipman proved that any spherical functor $F\colon \cC\to \dD$  is given by 4-periodical SODs.
\end{REM}

\section{Lifting to the bicategory $\textrm{Bimod}$}\label{sec_sph_funct_and_enh}

Let $k$ be a field. For a $k$-linear DG category $\eE$, we denote by $[\eE ]$ the \emph{homotopy category} of $\eE$, i.e. the category with the same objects as in $\eE$ and morphisms  the 0-th cohomology of the DG complexes of morphisms in $\eE$   \cite{BK}. It is an ordinary $k$-linear category. 

By a \emph{DG enhancement} of a triangulated category $\cC$ we mean a choice of a pre-triangulated 
DG category $\eE$ together with an equivalence $\cC \simeq [\eE ]$ which is compatible with the induced triangulated structure on $[\eE ]$ \cite{BK}. 
A choice of suitable DG enhancements for our categories prompts a convenient replacement for the category of functors, as well as a lift of adjunction (co)units to morphisms in this replacement. This allows us to get an interpretation for 'cones' of morphisms of functors that we need in the main body of the text.

For the sake of simplicity, we take as DG enhancements of our categories only suitable categories of DG modules over DG algebras. The category of functors is replaced by the derived category of bimodules, and natural transformations by morphisms in the latter category. As a result we arrive at the bicategory $\Bimod$ whose objects are DG algebras and 1-morphisms are objects of the derived categories of DG bimodules. There is a  bifunctor $\Phi :\Bimod \to \textbf{Cat}$ to the 2-category of categories. We are interested in lifting of adjoint pairs of functors in $ \textbf{Cat}$ along the functor $\Phi$.

In the general framework of 2-categories (or bicategories), we have the uniqueness of the 2-categorical adjunction for a given 1-morphism.  Also we show that 2-adjunctions can be transported over equivalences of objects in 2-categories, which implies an invariance of 2-categorical adjunctions in $\Bimod$ under the choice of DG enhancements.

By fixing a lift of a functor to a 1-morphism in $\Bimod$, we get an essentially  unique 2-categorical adjunction. This allows us to define a unique twist and cotwist of 2-categorical adjunctions as 1-morphisms in $\Bimod$, thus interpreting constructions of R. Anno and T. Logvinenko \cite{AnnLog}. 
In particular, we have the notion of spherical twists for spherical functors and, more generally, for spherical couples. 

Further we show how the 2-adjunction theory can be applied to functors between derived categories of (quasi)coherent sheaves. Here another suitable 2-category is the category $\FM$ whose objects are schemes and 1-morphisms are objects of 
the derived category of the product of two schemes. By the results of V. Lunts and O.Schn{\"u}rer \cite{LunSchnu} the 2-categorical adjunctions, twists and cotwists are readily transferred to this context. We discuss uniqueness of the functorial exact triangles associated to an FM functor and its adjoints. We also construct a 2-categorical lift to $\Bimod$ of the base-change morphism.

Finally, we provide a criterion for isomorphism of functors in terms of the restriction to one (generating) object.

\vspace{0.3cm}
\subsection{2-categorical adjunctions}\label{ssec_2-cat_adj}~\\

Let $\cC$ be a 2-category. A quadruple $(s,r, \eta, \varepsilon)$ of 1-morphisms $s\in \Hom_{\cC}(A,B)$, $r\in \Hom_{\cC}(B,A)$ and 2-morphisms $\eta \colon \Id_A \to rs$, $\varepsilon \colon sr \to \Id_B$ is a \emph{2-categorical $(A,B)$-adjunction} if 
\begin{align}\label{eqtn_triangle_equalities} 
&s \xrightarrow{s\eta}srs \xrightarrow{\varepsilon s}s,& &r \xrightarrow{\eta r}rsr  \xrightarrow{r\varepsilon} r&
\end{align} 
are equal to the identity morphisms of $s$ and $r$ respectively. In this case, $r$ is said to be a \emph{right 2-categorical adjoint} of $s$, $s$ a \emph{left 2-categorical adjoint} of $r$ and $\eta$ and $\varepsilon$ the \emph{unit} and \emph{counit} of the adjunction. We write $s\dashv r$ if there exist $\eta$ and $\ve$ such that $(s,r,\eta, \ve)$ is a 2-categorical adjunction. The choice of such $\eta$ and $\ve$ is not unique.

The proof of the following fact is standard.
\begin{LEM}\emph{cf.}\cite{Benab, ToeVez2}\label{lem_uniqueness_of_adjoint}
	Let $s$ be a 1-morphism. If $s$ has a right 2-categorical adjoint, then the 2-categorical adjunction $(s,r, \eta, \varepsilon)$ is unique up to a canonical 2-isomorphism. More precisely, for any other adjunction $(s,r', \eta', \varepsilon')$, the composite $\alpha\colon r\xrightarrow{\eta'r}r'sr  \xrightarrow{r\varepsilon}r'$ is a unique 2-isomorphism commuting with the units and counits. Conversely, any 2-isomorphism $\alpha \colon r \to r'$ yields a 2-categorical adjunction $(s, r',\alpha s\circ \eta, \ve \circ s\alpha^{-1})$.
\end{LEM}

An \emph{$(A,B)$-equivalence} in $\cC$ is a quadruple $(f,g,\nu, \mu)$ of 1-morphisms $f\in \Hom_{\cC}(A,B)$, $g\in \Hom_\cC(B,A)$ and 2-isomorphisms $\nu\colon \Id_A \xrightarrow{\simeq} gf$, $\mu \colon fg \xrightarrow{\simeq} \Id_B$. It is an \emph{adjoint $(A,B)$-equivalence} if $(f,g,\nu,\mu)$ is, in addition, a 2-categorical $(A,B)$-adjunction.

Let $\cC_1$ and $\cC_2$ be 2-categories. A \emph{pseudo-functor} $\theta \colon \cC_1 \to \cC_2$ consists of a map $\theta \colon \textrm{Ob}(\cC_1)\to \textrm{Ob}(\cC_2)$,  functors $\theta_{C,C'} \colon \Hom_{\cC_1}(C,C') \to \Hom_{\cC_2}(\theta(C), \theta(C'))$, for any pair $(C,C')$ of objects of $\cC_1$, and 2-isomorphisms $u\colon \theta(\Id_C) \to \Id_{\theta(C)}$, for any object $C\in \cC_1$, and $a\colon \theta(g\circ f) \to \theta(g)\circ \theta(f)$, for any pair $(f, g)$ of composable 1-morphisms in $\cC_1$. These need to satisfy coherence conditions: for any triple $(h, g, f)$ of composable 1-morphisms in $\cC_1$, the two possible 2-morphisms $\theta(h\circ g\circ f)\to \theta(h)\circ \theta(g)\circ \theta(f)$ induced by $a$ are equal, and so are the 2-morphisms $\theta(f) \to \theta(f)\circ \theta(\Id)$, $\theta(g) \to \theta(\Id)\circ \theta(g)$ induced by $a$ and $u^{-1}$.

We say that a pseudo-functor $\theta\colon \cC_1\to \cC_2$ is a \emph{pseudo-equivalence} if $\theta_{C_1,C_2}$ is an equivalence for all $(C_1,C_2)\in \cC_1$. 

Let $\alpha \colon e\to f$, $\alpha'\colon e'\to f'$ be 2-morphisms in a 2-category $\cC$. We say that $\alpha$ and $\alpha'$ are \emph{isomorphic}, and depict it by $\alpha \simeq \alpha'$, if there exist isomorphisms $\tau_e\colon e \xrightarrow{\simeq} e'$, $\tau_f\colon f\xrightarrow{\simeq}f'$ such that $\tau_f\circ \alpha = \alpha'\circ \tau_e$.

\begin{LEM}\label{lem_pseudo-funct_and_co_unit}
	Let $\theta \colon \cC_1 \to \cC_2$ be a pseudo-functor and $(s,r, \eta, \ve)$ a 2-categorical $(A,B)$-adjunction in $\cC_1$. Then $\theta(s)$ fits into a 2-categorical $(\theta(A),\theta(B))$-adjunction $(\theta(s), \theta(r),\eta', \ve')$ such that $\theta_{A,A}(\eta) \simeq \eta'$, $\theta_{B,B}(\ve)\simeq \ve'$.
\end{LEM}
\begin{proof} 
	A standard diagram chasing shows that for $\eta'$ and $\ve'$ defined as the composites:
	\begin{align*}
	& \eta'\colon \Id_{\theta(A)} \xrightarrow{u^{-1}} \theta(\Id_A) \xrightarrow{\theta(\eta)} \theta(rs) \xrightarrow{a} \theta(r)\theta(s),&\\
	&\ve' \colon \theta(s)\theta(r) \xrightarrow{a^{-1}} \theta(sr)\xrightarrow{\theta(\ve)} \theta(\Id_B) \xrightarrow{u}\Id_{\theta(B)}& 
	\end{align*} 
	equalities (\ref{eqtn_triangle_equalities}) hold, i.e. $(\theta(s),\theta(r),\eta', \ve')$ is a 2-categorical $(\theta(A), \theta(B))$-adjunction. 
	
	By definition of $\eta'$ and $\ve'$, diagrams
	\[
	\xymatrix{\Id_{\theta(A)} \ar[r]^{\eta'} & \theta(r)\theta(s)\\
		\theta(\Id_A) \ar[u]^{u}_{\simeq} \ar[r]^{\theta_{A,A}(\eta)} & \theta(rs) \ar[u]_{a}^{\simeq}}
	\xymatrix{\theta(s)\theta(r) \ar[r]^{\ve'} & \Id_{\theta(B)} \\
		\theta(sr) \ar[u]^{a}_{\simeq} \ar[r]^{\theta_{B,B}(\ve)} & \theta(\Id_B) \ar[u]_{u}^{\simeq}}
	\]
	commute, which proves isomorphisms $\theta_{A,A}(\eta)\simeq \eta'$, $\theta_{B,B}(\ve)\simeq \ve'$.
\end{proof}
\begin{PROP}\label{prop_adjun_inv_under_equiv}
	Let $\cC$ be a 2-category, $(f_A,g_A, \nu_A, \mu_A)$ an $(A,A')$-equivalence and $(f_B,g_B, \nu_B, \mu_B)$ a $(B,B')$-equivalence in $\cC$. Denote by $\cC_{A,B}$, resp. $\cC_{A',B'}$, the full 2-subcategories of $\cC$ with two objects $A$ and $B$,  resp. $A'$ and $B'$. The four functors
	$$
	\theta_{S,T}(-) := f_T\circ(-) \circ g_S\colon \Hom_{\cC}(S,T) \to \Hom_{\cC}(S',T'),
	$$ 
	defined for pairs $(S,T)\in\{(A,A), (A,B), (B,A), (B,B)\}$ of objects of $\cC_{A,B}$ and the corresponding pair $(S',T')$ of  objects of $\cC_{A',B'}$, can be extended to a
	pseudo-equivalence $\theta \colon \cC_{A,B} \to \cC_{A',B'}$.
\end{PROP}
\begin{proof}
	By changing the 2-isomorphism $\nu_A$, $\mu_A$, $\nu_B$, and $\mu_B$ if necessary, one can assume that $(f_A,g_A, \nu_A, \mu_A)$ and $(f_B, g_B, \nu_B, \mu_B)$ are adjoint equivalences  \cite{MacLane2}.
	
	Functor $\theta_{S,T}$ is an equivalence with quasi-inverse $g_T\circ(-)\circ f_S$.
	We need to define morphisms $u$ and $a$ from the definition of pseudo-functors.
	The 2-isomorphisms $u$ are defined as $\mu_A$ and $\mu_B$:
	\begin{align*}
	&\theta_{A,A}(\Id_A) = f_A\circ g_A \xrightarrow{\mu_A} \Id_{A'},& &\theta_{B,B}(\Id_B) = f_B \circ g_B\xrightarrow{\mu_B} \Id_{B'}.&
	\end{align*}
	For a pair $(f,g)$ of 1-morphisms such that $B$ is the target of $f$ and the source of $g$, the 2-isomorphism $a$:
	\begin{align*}
	&f_{T_g}\circ g \circ f  \circ g_{S_f}\xrightarrow{f_{T_g} \circ g \circ \nu_B \circ f\circ g_{S_f}} f_{T_g}\circ  g \circ f_B \circ g_B \circ  f\circ g_{S_f},&
	\end{align*}
	where $S_f, T_g\in \{A,B\}$ are respectively the source of $f$ and the target of $g$. Similarly, we use $\nu_A$ if $A$ is the target of $f$ and the source of $g$. 
	
	Equalities (\ref{eqtn_triangle_equalities}) for 2-adjunctions $(f_A,g_A,\nu_A,\mu_A)$ and $(f_B,g_B,\nu_B, \mu_B)$ imply that $u$ and $a$ defined in this way satisfy the required coherence conditions.
\end{proof}

We say that a 2-category $\cC$ is \emph{1-triangulated} if, for any pair $(B,C)$ of objects of $\cC$, the category $\Hom_{\cC}(B,C)$ is triangulated and, for any $\f_A\in \Hom_\cC(A,B)$, $\f_D \in \Hom_\cC(C,D)$, functors $(-)\circ \f_A\colon \Hom_{\cC}(B,C) \to \Hom_{\cC}(A,C)$, $\f_D\circ(-) \colon \Hom_{\cC}(B,C) \to \Hom_{\cC}(B,D)$ are exact.

A 2-categorical $(A,B)$-adjunction $(s,r,\eta,\ve)$ in a 1-triangulated 2-category $\cC$ induces 1-endomorphisms $c_s$, $t_s$ defined via exact triangles:
\begin{equation}\label{eqtn_def_of_tw_and_cotw} 
\begin{aligned}
&c_s \to \Id_A \xrightarrow{\eta} rs \to c_s[1],& &sr \xrightarrow{\ve} \Id_B\to t_s \to sr[1].&
\end{aligned}
\end{equation}
We say that $t_s$ is the \emph{twist} and $c_s$ the \emph{cotwist} for the 2-categorical adjunction $(s,r,\eta, \ve)$. 

\begin{PROP}\label{prop_eqiuv_and_tw}
	Let $(s,r,\eta,\ve)$ be a 2-categorical $(A,B)$ adjunction in a 1-triangulated 2-category $\cC$.  Consider a pair of  $(A,A')$- and $(B,B')$-equivalences and the induced pseudo-equivalence $\theta \colon \cC_{A,B}\to \cC_{A',B'}$. Then the twist $t_{\theta(s)}$ is isomorphic to $\theta(t_s)$ and the cotwist $c_{\theta(s)}$ is isomorphic to $\theta(c_s)$.
\end{PROP}	
\begin{proof}
By Proposition \ref{prop_adjun_inv_under_equiv} $(A,A')$ and $(B,B')$-equivalences induce a pseudo-equivalence $\theta \colon \cC_{A,B}\to \cC_{A',B'}$ 
Lemma \ref{lem_pseudo-funct_and_co_unit} implies that 
$(\theta(s), \theta(r), \eta', \ve')$ is an $(A',B')$-adjunction and 
$\theta(\eta) \simeq \eta'$, $\theta(\ve) \simeq \ve'$. The isomorphism of cones follows.
\end{proof}

We say that a 2-categorical $(A,B)$-adjunction $(s, f, \eta, \ve)$ is a \emph{spherical couple} if the twist $t_s$, respectively the cotwist $c_s$, defined in (\ref{eqtn_def_of_tw_and_cotw}), is a  $(B,B)$-, resp. $(A,A)$-equivalence.

Under the above assumption, octahedron
\[
\xymatrix{sc_s[1] \ar[r] & 0 \ar[r] & t_ss \\
	0 \ar[r] \ar[u] & s \ar[r]^{\simeq} \ar[u]& s\ar[u] \\
	sc_s \ar[r] \ar[u] & s \ar[r] \ar[u]^{\simeq} & srs\ar[u]}
\]
implies an isomorphism 
\begin{equation}\label{eqtn_tw_psi=psi_ctw}
t_ss\simeq sc_s[2].
\end{equation}

A 1-morphism $s\colon A\to B$ 
is said to be \emph{spherical} \cite{AnnLog} if it has the 2-categorical left and right adjoints and $(s,r,\eta,\ve)$ is a spherical couple.
\vspace{0.3cm}
\subsection{The bicategory $\Bimod$}~\\

Let $k$ be a field. We consider the bicategory $\Bimod$ whose objects are unital DG $k$-algebras, 1-morphisms $A \to B$ are objects in the derived category $\dD(A^\textrm{op} \otimes_k B)$ of $A^{\opp}\otimes_k B$ DG modules and 2-morphisms  are morphisms in $\dD(A^\textrm{op} \otimes_k B)$. 
The derived tensor product of bimodules $M \otimes^L_B N$, for $M \in  \dD(A^\textrm{op} \otimes_k B)$ and $N\in \dD(B^\textrm{op} \otimes_k C)$, defines the composition of 1-morphisms in $\Bimod$. As the derived tensor product is unique up to a canonical isomorphism, we shall write formulas as if the composition of 1-morphisms in $\Bimod$ was
strictly associative and the identity morphisms were strict. If necessary, all morphisms of associativity and unitors can readily find their places in formulas. Equally, one can refer to the fact that every bicategory is biequivalent to a 2-category \cite{MacLanePare}.

The tensor product is an exact functor, hence the 2-category $\Bimod$ is 1-triangulated.

Let $\textbf{Cat}$ denote the 2-category of categories, functors and natural transformations. Consider the 2-functor 
\begin{align}\label{eqtn_func_Phi}
&\Phi \colon \Bimod \to \textbf{Cat}& &\Phi(A) = \dD(A).&
\end{align}
For $M\in \dD(A^{\opp} \otimes_k B)$, $\Phi_M \colon \dD(A) \to \dD(B)$ is the functor 
$$
\Phi_M(-) = (-) \otimes_A^LM.
$$ 
Finally, morphism $\alpha \colon M \to M'$ induces a natural transformation $\Phi_{\alpha}\colon \Phi_M \to \Phi_{M'}$.

In particular, as $\Phi$ is a 2-functor, the tensor product of bimodules corresponds to the composition of tensor functors: 
$$
\Phi_{M_1}\circ \Phi_{M_2} \simeq \Phi_{M_2\otimes^L M_1}.
$$

We note that the 2-functor $\Phi$ is in general neither full nor faithful.

Note that by \cite[Corollary 7.6]{Toen}, any DG enhanceable (i.e. admitting a lift to a functor of suitable DG enhancements) commuting with direct sums functor $\dD(A)\to \dD(B)$ is of the form $\Phi_M$, for some bimodule $M$. We shall refer to functors of the form $\Phi_M$ as \emph{bimodule functors}.

An exact triangle of bimodules in the category $\dD(A^\textrm{op}\otimes B)$:
$$
M_1\to M_2\to M_3 \to M_1[1]
$$ 
induces 
via the functor $\Phi$ a functorial exact triangle
\begin{equation}\label{eqtn_exact_tr_tens_func}
\Phi_{M_1}(-) \to \Phi_{M_2}(-) \to \Phi_{M_3}(-)\to \Phi_{M_1}(-)[1].
\end{equation} 

\begin{LEM}\label{lem_Psi_conservative}
	Functor $\Phi$ is conservative on the categories of 1-morphisms in $\Bimod$.
\end{LEM}
\begin{proof}
	Let $A$ and $B$ be unital DG algebras and $f\colon M_1\to M_2$ a morphism in $\dD(A^{\opp}\otimes B)$ such that $\Phi_f\colon \Phi_{M_1} \to \Phi_{M_2}$ is an isomorphism. Morphism $f$ fits into an exact triangle $M_1\xrightarrow{f} M_2\to M_3\to M_1[1]$.
	For any $E\in \dD(A)$, complex $\Phi_{M_1}(E)\xrightarrow{\Phi_f} \Phi_{M_2}(E)\to \Phi_{M_3}(E) \to \Phi_{M_1}(E)[1]$ is an exact triangle. As $\Phi_f$ is an isomorphism, $\Phi_{M_3}(E) \simeq 0$. In particular, $M_3\simeq \Phi_{M_3}(A) \simeq 0$, which implies that $f$ is an isomorphism.
\end{proof}

\vspace{0.3cm}
\subsection{The bicategory $\FM$}~\\

Let $k$ be a field. Consider the bicategory $\FM$ whose objects are quasi-compact, quasi-separated $k$-schemes. The category of morphisms $X\to Y$ is  $\Hom_{\FM}(X,Y) =\dD_{\textrm{qc}}(X\times Y)$ and the composition is given by the convolution: for schemes $X$, $Y$, $Z$, $K\in \dD_{\textrm{qc}}(X\times Y)$ and $L\in\dD_{\textrm{qc}}(Y\times Z)$ their composition is 
$$
K\ast L :={\pi_{XZ}}_*(\pi_{XY}^*K \otimes \pi_{YZ}^*L),
$$
where $\pi_{XY}\colon X\times Y \times Z \to X\times Y$, $\pi_{YZ}\colon X\times Y \times Z \to Y\times Z$, $\pi_{XZ}\colon X\times Y \times Z \to X\times Z$ are the projections.

Similarly to  $\Bimod$, we consider $\FM$ as a 1-triangulated bicategory.

Consider the (contravariant) bifunctor 
\begin{align*}
&\Xi \colon \FM \to \Cat,& &\Xi(X) =\dD_{\textrm{qc}}(X)& 
\end{align*}
which to a 1-morphism  $K\in \dD_{\textrm{qc}}(X\times Y)$ assigns the \emph{Fourier-Mukai functor} with kernel $K$:
$$
\Xi_K(-) = {Rp_Y}_*(Lp_X^*(-)\otimes K),
$$
where $p_X\colon X\times Y \to X$, $p_Y \colon X\times Y \to Y$ are the projections. Finally, 2-morphism $\alpha \colon K \to K'$ induces a natural transformation $\Xi_\alpha \colon \Xi_{K} \to \Xi_{K'}$.

An exact triangle in the category $\dD_{\textrm{qc}}(X\times Y)$
$$
\eE_1\to \eE_2\to \eE_3 \to \eE_1[1]
$$ 
induces 
via the functor $\Xi$ 
 a functorial exact triangle of FM functors: 
\begin{equation}\label{eqtn_exact_FM_func}
\Xi_{\eE_1} \to \Xi_{\eE_2} \to \Xi_{\eE_3}\to \Xi_{\eE_1}[1].
\end{equation}

\vspace{0.3cm}
\subsection{2-adjunctions, spherical couples and spherical 1-morphisms in $\Bimod$}\label{ssec_sph_funct}~\\

Any 2-functor preserves 2-adjunctions, in particular the 2-functor $\Phi$ maps a 2-categorical adjunction $(M, R, \eta, \varepsilon)$ in $\Bimod$ to a pair of adjoint functors $\Phi_M\dashv \Phi_R$ between the corresponding derived categories. 

Let $A$ and $B$ be unital DG algebras. We discuss when a pair of adjoint functors between $\dD(A)$ and $\dD(B)$ can be lifted to a 2-categorical adjunction in $\Bimod$. For  $M\in \dD(A^{\opp}\otimes_k B)$, define its $A$ and $B$ duals $M^A, M^B\in \dD(B^{\opp} \otimes_k A)$ \cite{AnnLog2}: 
\begin{align*}
&M^A :=R\Hom_A(M,A),& &M^B:=R\Hom_B(M,B).
\end{align*}
We have several (derived) evaluation  and action morphisms \cite{AnnLog2}:
\begin{equation}\label{eqtn_act_and_ev_maps} 
\begin{aligned}
&\varepsilon_R\colon M^B \otimes_A^L M \to B,& &\varepsilon_L \colon M\otimes_B^L M^A \to A,& \\
&\eta_R\colon A \to M \otimes_B^LM^B,& &\eta_L \colon B \to M^A\otimes_A^L M.&
\end{aligned}
\end{equation} 
Note that $\eta_R$ is defined if $M$ is $B$-perfect, i.e. if its image in $\dD(B)$ lies in the full subcategory of compact objects. Respectively, $\eta_L$ is defined if $M$ is $A$-perfect. 

The following proposition allows us to lift suitable pairs of  adjoint functors between triangulated categories to 2-categorical adjunctions in $\Bimod$.
\begin{PROP}\label{prop_AnnLog}
	Let $M$ be in $\dD(A^{\opp}\otimes_kB)$. Then
	\begin{enumerate}
		\item[(i)] $M$ is $A$-perfect if and only if the left adjoint $\Phi_M^*$ to $\Phi_M$ exists. Under these assumptions functor $\Phi_M^*$ is isomorphic to $\Phi_{M^A}$ and the evaluation and action maps yield a 2-categorical adjunction $(M^A, M, \eta_L, \varepsilon_L)$ in the bicategory $\Bimod$.
		\item[(ii)] $M$ is $B$-perfect if and only if $\Phi_M$ maps compact objects to compact ones if and only if the right adjoint $\Phi_M^!$ is isomorphic to $\Phi_{M^B}$. Under these assumptions the evaluation and action maps yield a 2-categorical adjunction $(M, M^B, \eta_R, \varepsilon_R)$.
	\end{enumerate}
\end{PROP}
\begin{proof}
	This is basically the statement of \cite[Proposition 4.2 and 4.7]{AnnLog2}. In \cite{AnnLog2} object $M$ is lifted to $\ol{M}$ in a suitable (weak, in the sense of Drinfeld \cite{Dri}) DG enhancement for $\dD(A^{\opp}\otimes_kB)$. 
	The $A$ and $B$ duals of $\ol{M}$ are defined, which lift $M^A$ and $M^B$. Also, the evaluation and action maps are constructed on the DG level. With \cite[Proposition 4.7]{AnnLog2} the authors check that the composites in (\ref{eqtn_triangle_equalities}) are identities of objects in the derived categories of bimodules.
\end{proof}

Consider $B$-perfect $M\in \dD(A^\textrm{op} \otimes_k B)$. The 2-categorical adjunction $(M,M^B, \eta_R, \ve_R)$ in $\Bimod$ (Proposition \ref{prop_AnnLog}$(ii)$) allows us to consider the twist $T_B$ and the cotwist $C_A$, defined with triangles (\ref{eqtn_def_of_tw_and_cotw}). Lemma \ref{lem_uniqueness_of_adjoint} implies that $T_B$ and $C_A$ are unique up to isomorphisms. 
Applying functor $\Phi$ as in (\ref{eqtn_func_Phi}) we get endo-functors $\Omega_{A,M}$ and $\Theta_{B,M}$ of $\dD(A)$, resp. $\dD(B)$.
They fit into functorial exact triangles: 
\begin{equation*}
\begin{aligned}
&\Omega_{A,M}\rightarrow \Id_{\dD(A)} \xrightarrow{\eta_R} \Phi_M^!\Phi_M \rightarrow \Omega_{A,M}[1],&\\
&\Phi_M \Phi_M^! \xrightarrow{\ve_R} \Id_{\dD(B)} \rightarrow \Theta_{B,M} \rightarrow \Phi \Phi^![1].&
\end{aligned}
\end{equation*} 
Uniqueness of $T_A$ and $C_B$ implies that $\Omega_{A,M} = \Phi_{C_A}$ and $\Theta_{B,M} = \Phi_{T_B}$ are defined uniquely by $M$ up to a (non-unique) functorial isomorphism. Note that we do not claim that the cotwist functor $\Omega_{A,M}$ and twist functor $\Theta_{B,M}$ are uniquely defined by the functor $\Phi_M$. 

Assume now that $M$ is $A$-perfect. According to Proposition \ref{prop_AnnLog}$(i)$, we can consider the twist $T_A$ and the cotwist $C_B$ for the 2-categorical adjunction $(M^A,M,\eta_L,\ve_L)$. 
By applying functor $\Phi$ as in (\ref{eqtn_func_Phi}), we get endo-functors $\Omega_{B,M}=\Phi_{C_B}$
and $\Theta_{A,M}= \Phi_{T_A}$.
 Those fit into functorial exact triangles:
\begin{align*}
&	\Omega_{B,M}\rightarrow \Id_{\dD(B)} \xrightarrow{\eta_L} \Phi_M\Phi_M^* \rightarrow \Omega_{B,M}[1],&\\
&\Phi_M^* \Phi_M \xrightarrow{\ve_L} \Id_{\dD(A)} \rightarrow \Theta_{A,M} \rightarrow \Phi_M^* \Phi_M[1].&
\end{align*}

We say that $\Phi_M$ is a spherical functor if $M$ is a spherical 1-morphism in $\Bimod$.
\begin{THM}\cite[5.1 and 5.2]{AnnLog}\label{thm_Ann_Log}
	Let $M$ be an $A$ and $B$-perfect bimodule.
	Functor $\Phi_M$ is spherical if any two of the following hold
	\begin{itemize}
		\item[$(S_1)$] $\Theta_{B,M}$ is an equivalence,
		\item[$(S_2)$] $\Omega_{A,M}$ is an equivalence,
		\item[$(S_3)$] composite $\Phi^*_M \Theta_{B,M}[-1] \to \Phi_M^* \Phi_M \Phi_M^! \to \Phi_M^!$ is an isomorphism of functors,
		\item[$(S_4)$] composite $\Phi_M^! \to \Phi_M^!\Phi_M \Phi_M^* \to \Omega_{A,M} \Phi_M^*[1]$ is an isomorphism of functors.
	\end{itemize}
	Then also $\Omega_{B,M}$ and $\Theta_{A,M}$ are also equivalences of categories, quasi-inverse to $\Theta_{B,M}$ and $\Omega_{A,M}$ respectively.
\end{THM}

\vspace{0.3cm}
\subsection{Lifting push-forwards, pull-backs and the base-change to $\Bimod$}\label{ssec_geometric_liftings}~\\

We fix a DG enhancement \cite{BK} for the category $\dD_\textrm{qc}(X)$, for example, by $h$-injective complexes \cite{Spa, KasSch2}.
Let $P \in \dD_{\textrm{qc}}(X)$ be a compact generator (see \cite{BvdB}) and $A$ the DG endomorphism algebra of its lift to the DG enhancement. Then $\dD(A)\simeq \dD_\textrm{qc}(X)$ by \cite{Kel2}, which paves the way to performing the necessary constructions in the category $\Bimod$. With Proposition \ref{prop_cotwist_up_to_iso} we check that the constructed functorial exact triangles are independent of the choice of $P$.

Assume that $X$ and $Y$ are Noetherian separated schemes such that any perfect complex on both $X$ and $Y$ is isomorphic to a bounded complex of locally free sheaves. Choose compact generators in $\dD_\textrm{qc}(X)$ and in $\dD_\textrm{qc}(Y)$ and fix lifts of the generators to DG enhancements of both categories. Denote by $A_X$ and $A_Y$ their DG endomorphism algebras.

By \cite{LunSchnu} there exists a choice of equivalences $\Upsilon_X \colon \dD_\textrm{qc}(X) \to \dD(A_X)$, $\Upsilon_Y \colon \dD_\textrm{qc}(Y) \to \dD(A_Y)$ and $\Upsilon_{X,Y}\colon \dD_{\textrm{qc}}(X\times Y) \simeq \dD(A_X^{\textrm{op}} \otimes A_Y)$ which map FM functors $\dD_{\textrm{qc}}(X) \to \dD_{\textrm{qc}}(Y)$ to bimodule functors $\dD(A_X)\to \dD(A_Y)$. More precisely, for any $\eE\in \dD_\textrm{qc}(X \times Y)$ the diagram 
\begin{equation}\label{eqtn_diag_LS1}
\xymatrix{\dD(A_X) \ar[rr]^{\Phi_{\Upsilon_{X,Y}(\eE)}} && \dD(A_Y)\\\dD_{\textrm{qc}}(X) \ar[u]^{\Upsilon_X}_\simeq \ar[rr]^{\Xi_\eE}&& \dD_{\textrm{qc}}(Y) \ar[u]^{\Upsilon_Y}_\simeq}
\end{equation}
commutes up to a functorial isomorphism.

Commutativity of (\ref{eqtn_diag_LS1}) implies that having fixed compact generators for $\dD_{\textrm{qc}}(X)$ and $\dD_{\textrm{qc}}(Y)$, we can view FM functors $\dD_{\textrm{qc}}(X) \to \dD_{\textrm{qc}}(Y)$ as bimodule functors and vice versa. 

Fix compact generators for $\dD_{\textrm{qc}}(X)$ and $\dD_{\textrm{qc}}(Y)$ and consider $\eE\in \dD_{\textrm{qc}}(X\times Y)$.  In view of (\ref{eqtn_diag_LS1}), for 
$$
M=\Upsilon_{X,Y}(\eE)
$$
functors  $\Upsilon_Y\Xi_\eE$ and $\Phi_{M} \Upsilon_X$ are isomorphic. 

\begin{LEM}
	If $\Xi_\eE$ has left adjoint $\Xi_\eE^*$ then $\Xi_\eE^*\simeq \Xi_{\Upsilon_{Y,X}^{-1}(M^{A_X})}$ is an FM functor.
	If $\Xi_\eE$ maps compact objects to compact objects, then the right adjoint exists and it is an FM functor, $\Xi_{\eE}^!\simeq \Xi_{\Upsilon_{Y,X}^{-1}(M^{A_Y})}$.	
\end{LEM}
\begin{proof}
	If $\Xi_{\eE}^*$ exists then $\Upsilon_X \Xi_{\eE}^*\Upsilon_Y^{-1}$ is left adjoint to $\Phi_M$. Proposition \ref{prop_AnnLog}$(i)$ implies that $\Upsilon_X \Xi_{\eE}^*\Upsilon_Y^{-1}\simeq \Phi_{M^{A_X}}$.
	Commutativity of (\ref{eqtn_diag_LS1}) implies $\Xi_{\eE}^*\simeq \Xi_{\Upsilon_{Y,X}^{-1}(M^{A_X})}$.
	
	If $\Xi_\eE$ maps compact objects to compact objects then so does $\Phi_M\simeq \Upsilon_Y \Xi_\eE \Upsilon_X^{-1}$. By Proposition \ref{prop_AnnLog}$(ii)$, $\Phi_M^!\simeq \Phi_{M^{A_Y}}$. Moreover, $\Upsilon_X^{-1}\Phi_{M^{A_Y}} \Upsilon_Y$ is right adjoint to $\Xi_{\eE}$, i.e. $\Xi_\eE^!\simeq \Upsilon_X^{-1}\Phi_{M^{A_Y}} \Upsilon_Y$. As above we conclude that  $\Xi_{\eE}^! \simeq \Xi_{\Upsilon_{Y,X}^{-1}(M^{A_Y})}$.
\end{proof}

The 2-categorical adjunctions $(M^{A_X}, M, \eta_L,\varepsilon_L )$, $(M, M^{A_Y}, \eta_R, \varepsilon_R)$ in $\Bimod$ 
define the twist and  the cotwist by triangles  (\ref{eqtn_def_of_tw_and_cotw}). In view of the isomorphism $\Xi_{\eE}\simeq \Phi_{M}$ and the uniqueness of adjoint functors (see Lemma \ref{lem_uniqueness_of_adjoint}), the functorial exact triangles read:
\begin{equation}\label{eqtn_geome_cotwist}
\begin{aligned}
&\Xi_\eE^*\Xi_\eE \to \Id_{\dD_{\textrm{qc}}(X)}\to \Theta_{X,\eE}\to \Xi_\eE^*\Xi_\eE[1],&\\
&\Omega_{Y,\eE} \to\Id_{\dD_{\textrm{qc}}(Y)} \to \Xi_\eE \Xi_\eE^*\to \Omega_{Y,\eE}[1],&\\
&\Xi_{\eE}\Xi_{\eE}^! \to \Id_{\dD_{\textrm{qc}}(Y)}\to \Theta_{Y,\eE}\to \Xi_{\eE}\Xi_{\eE}^![1],&\\
&\Omega_{X,\eE}\to\Id_{\dD_{\textrm{qc}}(X)} \to \Xi_{\eE}^!\Xi_{\eE}\to \Omega_{X,\eE}[1].&
\end{aligned}
	\end{equation}
Note that we construct the triangles via the category $\Bimod$. We do not check that (\ref{eqtn_geome_cotwist}) are induced by morphisms between the convolutions of $\eE$ with FM kernels of the adjoint functors and the structure sheaf of the diagonal.

\begin{PROP}\label{prop_cotwist_up_to_iso}
	Consider $\eE \in \dD_{\textrm{qc}}(X\times Y)$. Then, up to an isomorphism, the functorial exact triangles of FM functors (\ref{eqtn_geome_cotwist}) do not depend on the choice of compact generators for $\dD_{\textrm{qc}}(X)$ and $\dD_{\textrm{qc}}(Y)$. 
\end{PROP}
\begin{proof}
	Another choice of generators for $\dD_{\textrm{qc}}(X)$ and $\dD_{\textrm{qc}}(Y)$ gives equivalences $\Upsilon_X^B\colon \dD_{\textrm{qc}}(X)\xrightarrow{\simeq} \dD(B_X)$, $\Upsilon_Y^B\colon \dD_{\textrm{qc}}(Y) \xrightarrow{\simeq} \dD(B_Y)$, for some DG algebras $B_X$, $B_Y$. The composite $\Upsilon_X^B\circ \Upsilon_X^{-1}$, resp. $\Upsilon_Y^B\circ \Upsilon_Y^{-1}$, is DG enhanceable and commutes with direct sums, hence it admits a lift to a bimodule functor, \emph{cf.} \cite{Toen}. Functor $\Upsilon_X^B\circ \Upsilon_X^{-1}$, resp. $\Upsilon_Y^B\circ \Upsilon_Y^{-1}$, together with its quasi-inverse yields $(A_X, B_X)$-, resp. $(A_Y,B_Y)$-, equivalence in $\Bimod$. By Proposition \ref{prop_eqiuv_and_tw}, the induced pseudo-equivalence $\theta \colon \Bimod_{A_X,A_Y} \to \Bimod_{B_X,B_Y}$ preserves the twist and the cotwist of the adjunction. The image of $\theta$ under functor $\Phi$ yields the required isomorphism of functorial triangles induced by $\Upsilon$ and $\Upsilon^B$.
\end{proof}

Let $f\colon X\to Y$ be a proper morphism. The FM kernel of the functor $Rf_*\colon \dD_{\textrm{qc}}(X) \to \dD_{\textrm{qc}}(Y)$ is the structure sheaf of the graph of $f$. As functor $Lf^*$, left adjoint to $Rf_*$, exists, 
there are functorial exact triangles of endo-functors of $\dD_{\textrm{qc}}(X)$ and $\dD_{\textrm{qc}}(Y)$:
\begin{align*} 
&Lf^*Rf_* \to \Id_{\dD_{\textrm{qc}}(X)}\to \Theta_{X,f}\to Lf^*Rf_*[1],&\\
&\Omega_{Y,f} \to\Id_{\dD_{\textrm{qc}}(Y)} \to Rf_*Lf^*\to \Omega_{Y,f}[1].&
\end{align*} 
Assume further that $Y$ is smooth. Then $Rf_*\colon \dD_{\textrm{qc}}(X)\to \dD_{\textrm{qc}}(Y)$ maps compact objects to compact objects
and we have functorial exact triangles:
\begin{align*}
&Rf_*f^! \to \Id_{\dD_{\textrm{qc}}(Y)}\to \Theta_{Y,f}\to Rf_*f^![1],&\\
&\Omega_{X,f}\to\Id_{\dD_{\textrm{qc}}(X)} \to f^!Rf_*\to \Omega_{X,f}[1].&
\end{align*} 
Without the smoothness assumption the pair $Rf_*\dashv f^!$ admits a lift to DG functors of the categories of ind-coherent sheaves, \cite{GaiRozI}.

Below we discuss compatibility of liftings of push-forward functors to $\Bimod$ with the composition of 1-morphisms.

Let $X$ be a quasi-compact quasi-separated scheme over a field $k$. 
Let $\mathscr{C}(X)$ denote the category of complexes of sheaves of $\oO_X$-modules with quasi-coherent cohomology and $\mathscr{I}(X)$ its full subcategory of $h$-injective objects. The category $\mathscr{I}(X)$ is a DG enhancement for $\dD_{\textrm{qc}}(X)$. By \cite[Corollary 2.3]{Schnu} there exists a DG functor $\iota_X\colon \mathscr{C}(X) \to \mathscr{I}(X)$ together with a natural transformation $\alpha_X \colon \Id_{\mathscr{C}(X)} \to \iota_X$ whose evaluation at every object of $\mathscr{C}(X)$ is a quasi-isomorphism. In fact, functor $\iota_X$ is defined in  \cite{Schnu} on the bigger category of complexes of $\oO_X$-modules, here we denote by the same letter its restriction to $\mathscr{C}(X)$. 

Let $f\colon X\to Y$ be a morphism of quasi-compact quasi-separated schemes. The composite 
$$
\wt{f}_*\colon \mathscr{I}(X) \xrightarrow{f_*}\mathscr{C}(Y) \xrightarrow{\iota_Y} \mathscr{I}(Y)
$$
defines a DG functor such that the induced functor on homotopy categories is isomorphic to $Rf_*$. For $g\colon Y \to Z$, the natural transformation $\alpha_Y \colon \Id_{\mathscr{C}(Y)} \to \iota_Y$ yields a morphism 
\begin{equation}\label{eqtn_composition_push-forw} 
\beta \colon \wt{gf}_* = \iota_Zg_*f_*\to \iota_Z g_*\iota_Yf_* = \wt{g}_*\wt{f}_*
\end{equation}
which induces an isomorphism of the induced functors of homotopy categories.

Let $E_X\in \mathscr{I}(X)$ be an object whose image in the homotopy category is a compact generator. Denote by $A_X$  its DG endomorphism algebra. Then 
$$
\Hom_{\mathscr{I}(X)}(E_X,-)\colon \mathscr{I}(X) \to \textrm{DGMod-}A_X
$$ 
induces an equivalence 
$$
\zeta_X \colon \dD_{\textrm{qc}}(X) \xrightarrow{\simeq} \dD(A_X).
$$

Analogously, a lift to $\mathscr{I}(Y)$ of a compact generator for $\dD_{\textrm{qc}}(Y)$ to $\mathscr{I}(Y)$ yields an equivalence $\zeta_Y\colon \dD_{\textrm{qc}}(Y) \to \dD(A_Y)$. Consider an $A_X^{\opp} \otimes A_Y$-bimodule: 
\begin{equation}\label{eqtn_M_f}
M_f := \Hom_{\mathscr{I}(Y)}(E_Y,\wt{f}_*E_X).
\end{equation}
Then 
\begin{equation}\label{eqtn_Phi_M_f_is_push_forw}
\Phi_{M_f} \circ \zeta_X\simeq \zeta_Y \circ Rf_*.
\end{equation}

\begin{PROP}\label{prop_convolution}
	Consider morphisms $X\xrightarrow{f} Y \xrightarrow{g}Z$ of quasi-compact separated schemes and denote by $A_Z$ the DG endomorphism algebra of a lift $E_Z$ of a compact generator for $\dD_{\textrm{qc}}(Z)$ to $\mathscr{I}(Z)$. Then there exists a canonical isomorphism in $\dD(A_X^{\opp} \otimes A_Z)$:
	\begin{equation}\label{eqtn_isom_conv}
	M_f\otimes_{A_Y}^LM_g \to M_{gf}
	\end{equation} 
\end{PROP}
\begin{proof} 
	Morphism $\beta$ (\ref{eqtn_composition_push-forw}) applied to $E_X$ is a quasi-isomorphism of $h$-injective complexes, hence it has a homotopy inverse. It follows that morphism of DG bimodules 
	$$
	\delta \colon M_{gf} = \Hom_{\mathscr{I}(Z)}(E_Z, \wt{gf}_*E_X)  \xrightarrow{\beta \circ (-)}  \Hom_{\mathscr{I}(Z)}(E_Z, \wt{g}_*\wt{f}_*E_X) =: N_{gf}
	$$ 
	has a homotopy inverse, i.e. the induced morphism $\ol{\delta}$ in $\dD(A_X^{\opp} \otimes A_Z)$ is an isomorphism. In particular, for equivalence $\zeta_Z\colon \dD_{\textrm{qc}}(Z) \to \dD(A_Z)$ induced by $E_Z$, formula (\ref{eqtn_Phi_M_f_is_push_forw}) with $f$ replaced by $gf$, $\zeta_Y$ by $\zeta_Z$, and $M_f$ by $M_{gf}$ implies $\zeta_Z \circ R(gf)_*\simeq \Phi_{N_{gf}}\circ \zeta_X$.
	
	Functor $\wt{g}_*$ together with the composition of morphisms in $\mathscr{I}(Z)$ defines a morphism of bimodules
	\begin{align*} 
	&\gamma \colon M_f\otimes_{A_Y} M_g = \Hom_{\mathscr{I}(Y)}(E_Y, \wt{f}_*E_X) \otimes_{A_Y} \Hom_{\mathscr{I}(Z)}(E_Z,\wt{g}_* E_Y) \to\\
	&\to  \Hom_{\mathscr{I}(Z)}(E_Z, \wt{g}_*\wt{f}_*E_X) =: N_{gf}.
	\end{align*}
	Let $\ol{\gamma}$ be the induced morphism in $\dD(A_X^{\opp} \otimes A_Z)$.

	Formula (\ref{eqtn_Phi_M_f_is_push_forw}) and its analogue for $g\colon Y \to Z$ yields isomorphisms 
	 $\zeta_{Z} \circ (Rg_*\circ Rf_*)\simeq \Phi_{M_g}\circ \zeta_Y \circ Rf_*\simeq \Phi_{M_g}\circ \Phi_{M_f} \circ \zeta_X \simeq  \Phi_{M_f\otimes_{A_Y}^LM_g} \circ \zeta_X$. 
	 Therefore, $\Phi_{N_{gf}}$ and $\Phi_{M_f\otimes_{A_Y}^LM_g}$ correspond via $\zeta_X$ and $\zeta_Z$ to $R(gf)_*$ and $Rg_*\circ Rf_*$. 
	As these functors are isomorphic, we conclude that applying $\Phi$ to $\ol{\gamma}$ yields an isomorphism of functors. Since, by Lemma \ref{lem_Psi_conservative}, $\Phi$ is conservative, $\ol{\gamma}$ is an isomorphism. Then $\ol{\delta}^{-1} \circ \ol{\gamma}$ is the sought isomorphism (\ref{eqtn_isom_conv}).
\end{proof} 

Note that as functor $\Phi_{M_f} \colon \dD(A_X) \to \dD(A_Y)$ has left adjoint, Proposition \ref{prop_AnnLog}$(ii)$ implies a 2-categorical adjunction $(M_f^{A_X}, M_f, \eta_L, \ve_L)$ in the category $\Bimod$.

Consider a commutative diagram of morphisms of quasi-compact quasi-separated schemes
\begin{equation}\label{eqtn_base-change_diag}
\xymatrix{&Z \ar[dl]_p \ar[dr]^q & \\ X\ar[dr]^f && W \ar[dl]_g \\ & Y &}
\end{equation}
We search for a 1-morphism in $\Bimod$ whose image under the functor $\Phi$ as in (\ref{eqtn_func_Phi}) is the base-change $Lg^*Rf_*\to Rq_*Lp^*$.

Choose $E_X\in \mathscr{I}(X)$, $E_Y\in \mathscr{I}(Y)$, $E_Z\in \mathscr{I}(Z)$ and $E_W\in \mathscr{I}(W)$ whose images in the homotopy categories are compact generators. Denote by $A_X$, $A_Y$, $A_Z$ and $A_W$ their DG endomorphism algebras. Let $M_f$, $M_q$, $M_g$ and $M_p$ be DG bimodules as in (\ref{eqtn_M_f}).

In view of Proposition \ref{prop_convolution} equality $f\circ p = g \circ q$ implies an isomorphism 
$\nu \colon  M_p\otimes^L_{A_X} M_f\to M_q\otimes^L_{A_W}M_g$.
Consider the composite 
\begin{equation} \label{eqtn_base-change}
\begin{aligned}
\omega \colon M_f \otimes^L_{A_Y}M_g^{A_W} &\to M_p^{A_Z} \otimes^L_{A_Z} M_p\otimes^L_{A_X}M_f \otimes^L_{A_Y}M_g^{A_W}\to \\ &\to M_p^{A_Z} \otimes^L_{A_Z}M_q\otimes^L_{A_W}M_g\otimes^L_{A_Y}M_g^{A_W}\to M_p^{A_Z} \otimes^L_{A_Z}M_q
	\end{aligned}
\end{equation}
where the first map is 
induced by the lift of the $Lp^*\dashv Rp_*$ adjunction unit to the unit $A_X \to M_p^{A_Z}\otimes_{A_Z}^L M_p$ of the 2-categorical $(A_X,A_Z)$-adjunction, the second by $\nu$, and the third by the lift of the $Lg^*\dashv Rg_*$ adjunction counit to the counit $M_g\otimes_{A_Y}^LM_g^{A_W} \to A_W$ of the  $(A_Y,A_W)$-adjunction. Then $\omega$ is a lift of the 'base change' $Lg^*Rf_* \to Rq_*Lp^*$ to a 1-morphism in $\Bimod$. 

The base-change morphism $\omega$ (\ref{eqtn_base-change}) is the composite of appropriate adjunction units and counits. Hence, Proposition 
\ref{prop_cotwist_up_to_iso}
implies that $\omega$ yields a unique up to isomorphism morphism $Lg^*Rf_*\to Rq_*Lp^*$ of FM functors.

\vspace{0.3cm}
\subsection{Isomorphisms of functors via the restriction to one object}\label{ssec_iso_of_fun}~\\

Throughout the paper we repeatedly use the fact that, for \emph{algebras} $A$, $B$ and functors $F,G\colon \dD(A) \to \dD(B)$ defined by pure modules,
an isomorphism of $F$ and $G$ on generators yields an isomorphism $F\xrightarrow{\simeq} G$.

\begin{LEM}\label{lem_determining_functor}
	Let $A$, $B$ be algebras and $M$ a complex of $A-B$ bimodules such that $H^i(M) = 0$, for $i\neq j$, for a fixed $j\in \mathbb{Z}$. Let also $\eE \subset \dD(A)$ be the full subcategory with one object $A$. Then the quasi-isomorphism class of $M$ is determined by the isomorphism class of the restriction functor $\Phi_{M}|_{\eE} \colon \eE\to \dD(B)$.
\end{LEM}
\begin{proof}
	By definition, $M \simeq \Phi_{M}(A)$ as a right $B$ module. Functor $\Phi_{M}|_{\eE}$ gives a map $A \simeq \Hom_A(A,A) \xrightarrow{\beta} \Hom_{B}(M,M)$. By assumption on $M$, we have a quasi-isomorphism $M \simeq H^j(M)$, hence $\Hom_{B}(M,M)\simeq \Hom_{B}(H^j(M), H^j(M))$. Then morphism $\beta$ recovers the left $A$ module structure of $H^j(M)$, hence the quasi-isomorphism class of $M$ as an $A-B$ bimodule.
\end{proof}

As an easy corollary, we get the following
\begin{LEM}\label{lem_iso_on_A_mod}
	Let $A$, $B$ be algebras and $M_1$, $M_2$ complexes of $A-B$ bimodules such that $H^i(M_1)  =0$, for $i \neq j$, for a fixed $j\in \mathbb{Z}$. Let $\eE \subset \dD(A)$ denote the full subcategory with one object $A$. Then functors $\Phi_{M_1}$, $\Phi_{M_2}$ are isomorphic if and only if $\Phi_{M_1}|_{\eE}$ and $\Phi_{M_2}|_{\eE}$ are.
\end{LEM}
\begin{proof}
	Since $\Phi_{M_1}(A) \simeq M_1$ and $\Phi_{M_2}(A) \simeq M_2$, isomorphism of functors $\Phi_{M_1}|_{\eE}\simeq \Phi_{M_2}|_{\eE}$ implies a quasi-isomorphism of $M_1$ and $M_2$ as complexes of right $B$ modules. Therefore, $H^i(M_2) = 0$, for $i\neq j$. We conclude by Lemma \ref{lem_determining_functor}.
\end{proof}

	\section{The structure of the reduced fiber}\label{sec_fiber_of_f}
	
	Let $f\colon X \to Y$ be a proper morphism with fibers of relative dimension bounded by one such that $Rf_*(\oO_X) = \oO_Y$. Let $C$ denote the reduced fiber of $f$ over a closed point of $Y$.
	
	Applying functor $Rf_*$ to short exact sequence
	$$
	0 \to I_C \to \oO_X \to \oO_C \to 0
	$$
	of sheaves on $X$ we obtain a surjective morphism $R^1f_* \oO_X \to R^1f_* \oO_C$, which implies
	$$
	H^1(\oO_C) =0.
	$$
	Moreover, by \cite[Corollary II.11.3]{Har}, $C$ is connected, i.e. $H^0(\oO_C) = k$.
	
	Let $C = \bigcup C_i$ be a reduced curve such that all of its irreducible components $C_i$ are smooth. An \emph{incidence graph} is a bipartite graph whose vertices correspond to irreducible components and singular points of $C$. An edge connects a vertex corresponding to an irreducible component $C_i$ with a vertex corresponding to a singular point $c \in \textrm{Sing}(C)$ if and only if $c \in C_i$. Note that the incidence graph can be simply connected even though the dual intersection graph might be not, as the example of three curves meeting in one point shows: the incidence graph is a tree - a point connected with three other points, while the dual graph is a triangle.
	
	We say that a reduced curve has {\em normal crossing} singularities if all components of the curve are smooth and the Zariski tangent space at every singular point is the direct sum of tangent subspaces corresponding to components that meet at this point.

	\begin{THM}\label{thm_fiberstructure}
		Let $C$ be a reduced proper algebraic curve over field $k$. Then $H^1(\oO_C) =0$ if and only if the following conditions are satisfied:
		\begin{itemize}
			\item[(i)] Every irreducible component $C_i$ of $C$ is a smooth rational curve,
			\item[(ii)] The incidence graph of $C$ has no cycles,
			\item[(iii)] The curve has normal crossing singularities.
		\end{itemize}
	\end{THM}

	\begin{proof} Assume that $H^1(\oO_C) =0$. Let $C_i$ be an irreducible component of $C$. The restriction morphism $\oO_C \to \oO_{C_i}$ gives a surjection on cohomology $H^1(\oO_C) \to H^1(\oO_{C_i})$, hence $H^1(\oO_{C_i}) =0$.
		Let
		$$
		\pi_i\colon \wt{C_i} \to C_i
		$$
		denotes the normalization of $C_i$.  Consider a short exact sequence of sheaves on $C_i$:
		$$
		0 \to \oO_{C_i} \to \pi_{i*} \oO_{\wt{C_i}} \to \fF \to 0.
		$$
		Since $\fF$ is supported at singular points of $C_i$, the group $H^1(\fF)$ vanishes, hence $H^1(\oO_{\wt{C_i}}) = 0$. As $\wt{C_i}$ is smooth,
		it is isomorphic to $\mathbb{P}^1_k$.
		
		Sequence
		$$
		0 \to H^0(\oO_{C_i}) \to H^0(\oO_{\wt{C_i}}) \to H^0(\fF) \to 0
		$$
		is exact and the first morphism is an isomorphism. It follows that $H^0(\fF)=0$, i.e. sheaf $\fF$ is trivial. This proves that $C_i$ is isomorphic to its normalization $\wt{C_i}$, which is a smooth rational curve. This proves (i).
		
		Assume that the incidence graph of $C$ has a cycle. Let $C'\subset C$ be the corresponding (minimal) cycle  of smooth rational curves, a subscheme in $C$ with irreducible components $C_1,\ldots,C_l$.
		Again, we know that $H^1(\oO_{C'}) = 0$.
		Then we have a short exact sequence
		$$
		0 \to\oO_{C'} \to \bigoplus_{i=1}^l \oO_{C_i} \to \bigoplus_{i=1}^l \oO_{c_i} \to 0,
		$$
		where $c_1, \ldots, c_l$ are singular points of $C'$. From long exact sequence of cohomology groups it follows that $H^1(\oO_{C'}) = k$. This contradiction proves (ii).
		
		Note that, for curves $C$ satisfying conditions (i) and (ii), there exists a universal curve ${\bar C}$ and morphism
		$$
		{\bar \pi}: {\bar C}\to C,
		$$
		with the property that ${\bar C}$ has normal crossing singularities and, for every curve $D$ with normal crossing singularities and a map $D\to C$, there is a unique lifting map $D\to {\bar C}$.
		
		The construction of the curve ${\bar C}$ is simple. Consider an affine neighbourhood $U\subset C$ of every point of the curve which does not contain more than one singular point. Assume there is one. The irreducible components that meet at the singular point give some affine subschemes in $U$. Since the category of affine schemes is opposite to the category of unital commutative algebras, we can define the open chart ${\bar U}$ of the universal curve over $U$
		as a colimit over the diagram of embedding of the singular point into all these affine subschemes
		(by taking the spectrum of the limit over the diagram of the corresponding commutative algebras). It has one normal crossing singularity.
		
		If $U$ has no singular points, then we put ${\bar U}=U$. We can glue ${\bar U}$'s over $C$ into a curve, ${\bar C}$, with normal crossing singularities. Conditions (i) and (ii) guarantee  that ${\bar C}$ has the required universal property. Since we do not need this for our purpose, we skip the proof.
		
		Clearly, ${\bar \pi}$ is a set-theoretic isomorphism.
		We again have a short exact sequence with sheaf $\fF$ supported at singular points of $C$:
		$$
		0 \to \oO_C \to {\bar \pi}_{*} \oO_{\bar C} \to \fF \to 0,
		$$
		whose long exact sequence of cohomology shows that $\fF$ is trivial sheaf. This implies that the schematic structure of $C$ and ${\bar C}$ coincide, which proves (iii).
		
		Conversely, assume that $C$ is a reduced curve satisfying conditions (i), (ii) and (iii). We can assume that the curve is connected. We proceed by induction on the number of irreducible components of $C$. The base of induction, the case of a single component, is obvious. Since the incidence graph is a tree, then we can choose an irreducible component $C_1$ which has only one singular point, $c$. The curve $C'$, the union of the other components,  satisfies the same assumptions as $C$. Hence, $H^1(\oO_{C'}) = 0$ by induction hypothesis. Since the point $c$ is a normal crossing singularity, then the kernel of the restriction morphism $\oO_C\to \oO_{C'}$ is easily seen to be $\oO_{C_1}(-1)$. Then looking at the cohomology sequence for the short exact sequence
		$$
		0\to \oO_{C_1}(-1) \to  \oO_C \to \oO_{C'} \to 0,
		$$
		we see that $H^1(\oO_C) = 0$.
	\end{proof}
	
	\section{Calculation of Ext-groups for bounded above complexes}\label{sec_spectr_seq}

	Let $\aA$ be an abelian category and $A^{\bcdot}$ a bounded above complex over $\aA$. Its "stupid" truncations $\sigma_{\ge i}A$ define direct system $\sigma_{\ge i}A^{\bcdot} \to \sigma_{\ge i-1}A^{\bcdot}$. Given a complex $B^{\bcdot}$, we will be interested in the group ${\rm Hom}^{\bcdot}(A^{\bcdot}, B^{\bcdot})$ in terms of the stupid truncations of both complexes. There is a spectral sequence with ${\rm E}_1$-layer:
	\begin{equation}\label{eqtn_specseqfin}
	{\rm E}_1^{pq}=\bigoplus_{j-i=p}{\rm Ext}^q(A^i,B^j).
	\end{equation}
	If both complexes $A^{\bcdot}$, $B^{\bcdot}$ have finite number of non-trivial components then it is well-known that spectral sequence (\ref{eqtn_specseqfin}) converges. If complex $B^{\bcdot}$ is infinite then we need to put some conditions on $A^{\bcdot}$ in order to guarantee the spectral sequence is converging to ${\rm Hom}^{\bcdot}(A^{\bcdot}, B^{\bcdot})$.
	We shall show that if, in addition, both complexes are bounded above and $\Ext^q(A^i, B^j) =0$, for $q>0$ and $i,j \in \mathbb{Z}$, then we have a graded complex with terms
	\begin{equation}\label{eqtn_complex_calc_Ext}
	{\rm C}^p = \prod_{j-i = p} \Hom(A^i, B^j)
	\end{equation}
	whose cohomology calculate $\Ext^p(A^\bcdot, B^\bcdot)$. Here the differentials are the limits of the relevant differentials in the sequence (\ref{eqtn_specseqfin}) when we allow the truncations of $A^{\bcdot}$ and $B^{\bcdot}$ to go to infinity.
	
	Recall the following fact about cohomology of limits. Let $(M_i)_{i \in \mathbb{Z}}$ be an inverse system of abelian groups and $M = \varprojlim M_i$. We say that system $(M_i)_{i \in \mathbb{Z}}$ satisfies $(\star)$ if morphism $M_i \to M_{i-1}$ is surjective, for any $i$.
	\begin{LEM}\cite[Lemma 0.11]{Spa}\label{lem_Spalten}
		Let $(A_i \xrightarrow{f_i} B_i \xrightarrow{g_i} C_i \xrightarrow{h_i} D_i)_{i \in \mathbb{Z}}$ be an inverse system of complexes of abelian groups such that systems $(A_i)$, $(B_i)$, $(C_i)$, $(D_i)$ satisfy $(\star)$ for all $i$. Denote by $(A \xrightarrow{f} B \xrightarrow{g} C \xrightarrow{h} D)$ the limit complex. Let $A_i'$, $B_i'$, $C_i'$ and $D'_i$ be kernels of $A_i \to A_{i-1}$, $B_i \to B_{i-1}$, $C_i \to C_{i-1}$ and $D_i \to D_{i-1}$ respectively. Assume there is $j\in \mathbb{Z}$ such that sequence $A'_i \to B'_i \to C'_i \to D'_i$ is exact for all $i >j$. Then the natural homomorphism $\textrm{Ker} g/ \textrm{Im} f \to \textrm{Ker} g_j/ \textrm{Im} f_j$ is an isomorphism.
	\end{LEM}

	\begin{LEM}\label{lem_limit_of_hom_is_hom}
		Let $A^\bcdot$ and $B^\bcdot$ be  complexes over $\aA$,  $A^\bcdot$ bounded above. Assume $l\in \mathbb{Z}$ such that $\Ext^j_{\dD^b(\aA)}(A^i, B^{\bcdot}) = 0$, for $j<l$ and all $i$. 
		Then the canonical morphism 
		$$
		\Ext^j(A^{\bcdot},B^{\bcdot}) \to \varprojlim \Ext^j(\sigma_{\gge i} A^{\bcdot}, B^{\bcdot})
		$$
is an isomorphism for any $j\in \mathbb{Z}$. Here $\sigma_{\gge i} A^{\bcdot}$ is the 'stupid' truncation of $A^\bcdot$.
	\end{LEM}
	
	\begin{proof}
		Fix $I^{\bcdot}$, an h-injective resolution of $B^{\bcdot}$ \cite{Spa}. The direct system $\sigma_{\gge i+1}A^\bcdot \to \sigma_{\gge i}A^\bcdot$ induces an inverse system $\alpha_i \colon \Hom_{\textrm{Com}}(\sigma_{\gge i}A^\bcdot, I^\bcdot) \to \Hom_{\textrm{Com}}(\sigma_{\gge i+1}A^\bcdot, I^\bcdot)$. Morphism $\alpha_i$ is surjective, i.e. property $(\star)$ is satisfied, and its kernel is complex $\Hom_{\textrm{Com}}(A^{i}[-i], I^\bcdot)$. Cohomology of this complex is $H^j(\Hom_{\textrm{Com}}(A^{i}[-i], I^\bcdot)) \simeq \Ext^{j+i}_{\dD^b(\aA)}(A^{i}, B^\bcdot)$.
		
		Our assumption implies that, for any $j$, there exists $N_j\in \mathbb{Z}$ such that, for $i<N_j$, complex $\Hom_{\textrm{Com}}(A^{i}[-i], I^\bcdot)$ is exact at degree $j$. By Lemma \ref{lem_Spalten}, we conclude that
		\begin{equation}\label{eqtn_cohomo_of_lim}
		H^j(\varprojlim \Hom_{\textrm{Com}}(\sigma_{\gge i} A^\bcdot, I^\bcdot)) \simeq H^j(\Hom_{\textrm{Com}}(\sigma_{\gge i} A^\bcdot, I^\bcdot)) \simeq \Ext^j_{\dD^b(\aA)}(\sigma_{\gge i} A^\bcdot, B^\bcdot).
		\end{equation}
		In particular, the inverse system $\Ext^j_{\dD^b(\aA)}(\sigma_{\gge i}A^\bcdot, B^\bcdot) \to \Ext^j_{\dD^b(\aA)}(\sigma_{\gge i+1} A^\bcdot, I^\bcdot)$ stabilizes, hence RHS of (\ref{eqtn_cohomo_of_lim}) is isomorphic to $\varprojlim \Ext^j_{\dD^b(\aA)}(\sigma_{\gge i} A^\bcdot, B^\bcdot)$, for $i$ sufficiently negative. Moreover, we have $\varprojlim \Hom_{\textrm{Com}}(\sigma_{\gge i} A^\bcdot, I^\bcdot) \simeq \Hom_{\textrm{Com}}(A^\bcdot, I^\bcdot)$. Hence, isomorphism $(\ref{eqtn_cohomo_of_lim})$ yields
		$$
		\Ext^j_{\dD^b(\aA)}(A^\bcdot,B^\bcdot) \simeq H^j(\Hom_{\textrm{Com}}(A^\bcdot, I^\bcdot)) \simeq \varprojlim \Ext^j_{\dD^b(\aA)}(\sigma_{\gge -i} A^\bcdot, B^\bcdot).
		$$
	\end{proof}
	For the reader's convenience, we also give a proof of the following well-known fact.
	
	\begin{LEM}\label{lem_spectral_sequence_for_bounded}
		Let $A^\bcdot$ and $B^\bcdot$ be complexes over $\aA$ such that $A^\bcdot$ is finite and $B^\bcdot$ is bounded above and with bounded cohomology. Assume there exists $N\in \mathbb{Z}$ such that $\Ext^j_{\aA}(A^i, A) = 0$, for all $j>N$, $i\in \mathbb{Z}$ and $A\in \aA$. Then, there exists a spectral sequence
		\begin{equation}\label{eqtn_spec_seq_bound_unbound}
		E_1^{p,q} = \bigoplus_k \Ext^p(A^k, B^{k+q})\, \Rightarrow \, \Ext^\bcdot_{\dD^b(\aA)}(A^\bcdot, B^\bcdot).
		\end{equation}
	\end{LEM}
	\begin{proof}
		For any $n\in \mathbb{Z}$, injective resolutions $I^{i,\bcdot}$ of $B^i$ give an injective resolution $J_n^\bcdot$ of $\sigma_{\gge n}B^\bcdot$ such that complex $J_n^\bcdot$ admits a filtration with graded factors $I^{i,\bcdot}$. The filtered complex $J_n^\bcdot$ yields a spectral sequence
		\begin{equation}\label{eqtn_spec_seq_for_truncation}
		{}^nE_1^{p,q} = \bigoplus_{\{k\,|\, k+q \geq n\}} \Ext^p(A^k, B^{k+q})\, \Rightarrow \, \Hom^\bcdot_{\dD^b(\aA)} (A^\bcdot, \sigma_{\gge n} B^\bcdot).
		\end{equation}
		It follows from our assumptions that ${}^nE_1^{p,q}$ is stationary for fixed $p$, $q$ and $n \rightarrow - \infty$. Moreover, the term ${}^nE_1^{p,q}$ is zero, for any $p\notin [0,N]$ or for any sufficiently positive $q$. This confirms existence of the spectral sequence (\ref{eqtn_spec_seq_bound_unbound}) and also implies that the $l$-th layer ${}^nE^{p,q}_{l}$ is stationary for fixed $p$, $q$ and $n \rightarrow - \infty$ and $l \rightarrow \infty$.
		
		Since $B^\bcdot$ has bounded cohomology, for $n$ sufficiently negative, we have $\tau_{\gge n+1}\sigma_{\gge n} B^\bcdot \simeq B^\bcdot$ and $\tau_{\lle n} \sigma_{\gge n} B^\bcdot \simeq \hH^n(\sigma_{\gge n} B^\bcdot)[-n]$. Then triangle $\tau_{\lle n} \sigma_{\gge n} B^\bcdot \to \sigma_{\gge n} B^\bcdot \to \tau_{\gge n+1} \sigma_{\gge n} B^\bcdot$ reads:
		$$
		\hH^n(\sigma_{\gge n}B^\bcdot)[-n] \to \sigma_{\gge n}B^\bcdot \to B^\bcdot \to \hH^n(\sigma_{\gge n}B^\bcdot)[-n+1].
		$$
		
		Since $A^\bcdot$ is finite, the assumptions on $A^i$ imply that there exists $\wt{N}$ such that $\Hom(A^\bcdot, C[k]) = 0$, for $k> \wt{N}$ and all $C \in \aA$. Thus, $\Ext^j(A^\bcdot, \hH^n(\sigma_{\gge n} B^\bcdot)[-n]) \simeq \Hom(A^\bcdot, \hH^n(\sigma_{\gge n} B^\bcdot)[j-n]) = 0$, for $n< j- \wt{N}$. Then, for $n<j - \wt{N}$, we have from the above triangle an isomorphism $\Ext^{j-1}_{\dD^b(\aA)}(A^\bcdot, \sigma_{\gge n} B^\bcdot) \simeq \Ext^{j-1}_{\dD^b(\aA)}(A^\bcdot, B^\bcdot)$.
		
		This shows that the spectral sequence in (\ref{eqtn_spec_seq_bound_unbound}) converges to $\Hom^\bcdot_{\dD^b(\aA)}(A^\bcdot, B^\bcdot)$.
	\end{proof}
	
	\begin{PROP}\label{prop_exist_of_spect_seq}
		Let $A^\bcdot$, $B^\bcdot$ be bounded above complexes of objects in $\aA$ with bounded cohomology. Assume that there exists $N\in \mathbb{Z}$ such that $\Ext^j(A^i, A) = 0$, for all $j>N$, $i \in \mathbb{Z}$ and $A\in \aA$. Assume further that $\Ext^j(A^i, B^k) = 0$, for $j>0$ and all $i, k \in \mathbb{Z}$. Then, $\Ext^j_{\dD^b(\aA)}(A^\bcdot, B^\bcdot)$ is isomorphic to the $j$'th cohomology group of complex (\ref{eqtn_complex_calc_Ext}).
	\end{PROP}
	
	\begin{proof}
		Lemma \ref{lem_spectral_sequence_for_bounded} applied to $\sigma_{\gge i} A^\bcdot$ and $B^\bcdot$, for any $i\in \mathbb{Z}$, implies that the spectral sequence (\ref{eqtn_spec_seq_bound_unbound})  reduces to a single row. Hence, we have $\Ext^j(\sigma_{\gge i}A^\bcdot, B^\bcdot) \simeq H^j(\oplus_{k>i} \Hom(A^k, B^{\bcdot+k}))$. The direct system $\sigma_{\gge i+1} A^\bcdot \to \sigma_{\gge i} A^\bcdot$ induces an inverse system of complexes $\oplus_{k>i}\Hom(A^k, B^{\bcdot+k}) \to \oplus_{k>i+1} \Hom(A^k, B^{\bcdot+k})$ satisfying condition $(\star)$. The kernel of $\oplus_{k>i}\Hom(A^k, B^{\bcdot+k}) \to \oplus_{k>i+1} \Hom(A^k, B^{\bcdot+k})$ is a complex $\Hom(A^i, B^\bcdot)[i]$. By Lemma \ref{lem_spectral_sequence_for_bounded}, we have $H^j(\Hom(A^i, B^\bcdot)[i]) \simeq \Ext^{j+i}(A^i, B^\bcdot)$, thus, for a fixed $j$ and any sufficiently negative $i$, complex $\Hom(A^i, B^\bcdot)[i]$ is exact at degree $j$. Lemma \ref{lem_Spalten} implies that, for such $i$, we have:
		$$
		H^j(\prod_k \Hom(A^k, B^{\bcdot +k})) \simeq H^j(\bigoplus_{k>i}\Hom(A^k, B^{\bcdot+k})) \simeq \Ext^j(\sigma_{\gge i} A^\bcdot, B^\bcdot).
		$$
		Thus, $\Ext^j(\sigma_{\gge i} A^\bcdot, B^\bcdot)$ is stationary, and $H^j(\prod_k \Hom(A^k, B^{\bcdot +k})) \simeq \varprojlim \Ext^j(\sigma_{\gge i}A^\bcdot, B^\bcdot)$. Since $B^\bcdot$ has bounded cohomology, the assumptions of Lemma \ref{lem_limit_of_hom_is_hom} are satisfied. It follows that the latter space is isomorphic to $\Ext^j(A^\bcdot, B^\bcdot)$.
	\end{proof}

	\bibliographystyle{alpha}
	\bibliography{../ref}
\end{document}